\documentclass[11pt]{article}
\usepackage[utf8]{inputenc}
\usepackage{amsmath}
\usepackage{amsfonts}
\usepackage{amssymb}
\usepackage{mathrsfs}
\usepackage{tikz-cd}
\usepackage{amsthm}
\usepackage{mathtools}
\usepackage{hyperref}
\usepackage{graphicx}
\usepackage[labelfont=bf]{caption}
\usepackage{xcolor}
\usepackage{authblk}
\usepackage{comment}

\graphicspath{ {./images/} }

\newcommand{\bq}{\mathbb{Q}}

\newcommand{\qp}{\mathbb{Q}_p}
\newcommand{\zp}{\mathbb{Z}_p}
\newcommand{\zetap}{\zeta_{p-1}}
\newcommand{\zetapn}{\zeta_{p^n}}
\newcommand{\ap}{a_p}
\newcommand{\calR}{\mathcal{R}}
\newcommand{\sT}{\mathcal{T}}
\newcommand{\norm}[1]{\vert #1 \vert}
\newcommand{\til}[1]{\widetilde{#1}}
\newcommand{\R}{\mathcal{R}}
\newcommand{\Dr}{\mathbf{D}_{\mathrm{rig}}}
\newcommand{\br}[1]{\overline{#1}}
\newcommand{\U}{U_r}

\newcommand{\gm}{\mathfrak{m}}
\newcommand{\gn}{\mathfrak{n}}

\newcommand{\co}{\mathcal{O}}
\newcommand{\RU}{\calR_{\co(\U)}}
\newcommand{\sL}{\mathcal{L}}
\newcommand{\brqp}{\br{\mathbb{Q}}_p}
\newcommand{\tU}{\til{U}_r}
\newcommand{\un}[1]{\underline{#1}}

\newtheorem{theorem}{Theorem}[section]
\newtheorem{definition}[theorem]{Definition}

\newtheorem{Corollary}[theorem]{Corollary}

\newtheorem{conjecture}[theorem]{Conjecture}
\newtheorem{Remark}[theorem]{Remark}

\newtheorem{lemma}[theorem]{Lemma}
\newtheorem{Proposition}[theorem]{Proposition}

\title{Semi-stable representations as limits of crystalline representations}

\author{Anand Chitrao}

\author{Eknath Ghate} 
\affil{School of Mathematics, Tata Institute of Fundamental Research, \\ Homi Bhabha Road, Mumbai 400005, India}

\author{Seidai Yasuda}
\affil{Department of Mathematics, Hokkaido University, \\  Kita 10, Nishi 8, Kita-Ku, Sapporo, Hokkaido, 060-0810, Japan}

\date{}

\begin{document}
\maketitle
\vspace{-.9cm}
\begin{abstract}  
  We construct an explicit sequence $V_{k_n,a_n}$ of crystalline representations of exceptional weights
  converging to a given
  irreducible two-dimensional semi-stable representation $V_{k,\sL}$ of $\mathrm{Gal}(\brqp/\qp)$.
  The convergence takes place in the blow-up space of two-dimensional
  trianguline representations studied by Colmez and Chenevier.
  The process
  of blow-up is described in detail in the rigid analytic setting and 
  may be of independent interest.  
  Also, we recover a formula of Stevens
  expressing the $\sL$-invariant as a logarithmic derivative.

  Our result can be used to compute the reduction of $V_{k,\sL}$ in terms of the reductions of the $V_{k_n,a_n}$.
  For instance, using the zig-zag conjecture we recover (resp. extend) the work of Breuil-M\'ezard and Guerberoff-Park computing
  the reductions of the $V_{k,\sL}$ for weights at most $p-1$ (resp. $p+1$), at least on
  the inertia subgroup. In the cases where zig-zag is known, we are further able to obtain some
  new information
  about the reductions for small odd weights.
  Finally, we explain some apparent violations to local constancy
  in the weight of the reductions of crystalline representations of small weight.
\end{abstract}

{\renewcommand\thefootnote{}}
{\renewcommand\thefootnote{}}


\vspace{-.7cm}
\tableofcontents

\section{Introduction}
Let $p$ be an odd prime. Let $E$ be a finite extension of $\qp$ containing $\sqrt{p}$.
Let $D_{\mathrm{st}}$ be Fontaine's functor inducing an equivalence of categories
between semi-stable representations of the Galois group $\mathrm{Gal}(\br{{\mathbb Q}}_p/\qp)$ over $E$ and admissible filtered
$(\varphi, N)$-modules over $E$. We introduce two kinds of representations using this functor.

For every integer $k \geq 2$ and $\ap \in E$ of positive valuation, there is an irreducible two-dimensional {\it crystalline} representation $V_{k, \ap}$ over $E$ of 
$\mathrm{Gal}(\br{{\mathbb Q}}_p/\qp)$ with Hodge-Tate weights $(0, k-1)$ and $D_\mathrm{st}(V_{k,a_p}^*) = D_{k,a_p}$, where $D_{k,a_p}  = E e_1 \oplus E e_2$ is the filtered $\varphi$-module defined by:
\begin{eqnarray*}
\begin{cases}
   \varphi(e_1) =  p^{k-1} e_2  \\  
              \varphi(e_2) = - e_1 + a_p e_2, 
\end{cases}
& \text{and } &
\mathrm{Fil}^i{D_{k, a_p}} =  \begin{cases} 
                                    D_{k,a_p},   &   \text{ if } i \leq 0 \\
                                    E e_1,         &   \text{ if } 1 \leq i \leq k-1 \\
                                    0,               &   \text{ if } i \geq k.
                                         \end{cases}
\end{eqnarray*}
Similarly, for every integer $k \geq 2$ and $\sL \in \mathbb{P}^1(E)$ (called the $\sL$-invariant), there is a two-dimensional {\it semi-stable} representation $V_{k, \sL}$ over $E$ of 
$\mathrm{Gal}(\br{{\mathbb Q}}_p/\qp)$ with Hodge-Tate weights $(0, k-1)$ and
$D_\mathrm{st}(V_{k,\sL}^*) = D_{k,\sL}$, where $D_{k,\sL}  = E e_1 \oplus E e_2$ is the filtered $(\varphi,N)$-module defined by:
\begin{eqnarray*}
   \begin{cases} \varphi(e_1) =  p^{k/2} e_1  \\  
                 \varphi(e_2) = p^{(k-2)/2} e_2, 
   \end{cases}
& \text{and } &
\mathrm{Fil}^i{D_{k, \sL}} =  \begin{cases} 
                                    D_{k, \sL},   &   \text{ if } i \leq 0 \\
                                    E (e_1 + \sL e_2),         &   \text{ if } 1 \leq i \leq k-1  \text{ and } \sL \neq \infty\\
                                    E (e_1+e_2), & \text{ if } 1 \leq i \leq k-1 \text{ and } \sL = \infty \\
                                    0,               &   \text{ if } i \geq k,
                                         \end{cases}
\end{eqnarray*}
and
\begin{eqnarray*}
  \begin{cases}
    N(e_1) =   e_2 \\
    N(e_2) =  0,
  \end{cases} \text{ if } \sL \neq \infty, \quad \text{ and } \quad
   N = 0 \> \text{ if } \sL  = \infty.
\end{eqnarray*}
If $k \geq 3$, then the semi-stable  representation $V_{k, \sL}$ is irreducible and when $\sL = \infty$ is isomorphic to
the representation $V_{k, a_p}$ with $a_p = p^{k/2} + p^{(k-2)/2}$.

This paper studies several relationships between the crystalline representations $V_{k,a_p}$ and the semi-stable representations $V_{k, \sL}$ for $k \geq 3$.

In particular, we show how information about the reductions of the former representations
implies information about the reductions of the latter. In general, computing the reductions of Galois representations
has applications to computing deformation rings, to the weight part of Serre's conjecture, to the Breuil-M\'ezard conjecture
and to modularity lifting theorems.

\subsection{Notation}
%
\begin{itemize}
  \item $p$ is an odd prime and $\sqrt{p}$ is a fixed square root of $p$
  \item $E$ is a $p$-adic number field, i.e., a finite extension of $\qp$ 
  \item $v_p$ is the $p$-adic valuation normalized such that $v_p(p) = 1$ 
  \item $\zetap$ is a fixed primitive $(p - 1)^{\mathrm{th}}$ root of unity in $\qp^*$
  \item $\log$ is the $p$-adic logarithm, normalized by setting $\log(p) = 0$
  \item $\sT$ is the rigid analytic space parameterizing continuous characters of $\qp^*$
  \item $x^i \in \sT(\qp)$ for $i \geq 0$ is the character $\qp^* \to \qp^*$ that sends an element to its $i^{\mathrm{th}}$ power
  \item $\chi \in \sT(\qp)$ is the $p$-adic cyclotomic character  $\qp^* \to \qp^*$  which maps $p$ to $1$ and which is the identity on $\zp^*$
  \item Characters of the form $x^i\chi$ for $i \geq 0$ are called exceptional characters
  \item $\mu_{\lambda}$ is the character of $\qp^*$ sending $p$ to $\lambda \in \br{\mathbb{F}}_p^*$ or $\brqp^*$ and $\zp^*$ to $1$
  \item Normalize the map $\qp^* \rightarrow  \mathrm{Gal}(\brqp/\qp)^{\mathrm{ab}}$ of class field theory by sending $p$ to a geometric Frobenius.
        We sometimes think of $\chi$ and $\mu_{\lambda}$ as characters of $\mathrm{Gal}(\br{{\mathbb Q}}_p/\qp)$. 
  \item Let $\Gamma = \mathrm{Gal}(\qp(\mu_{p^\infty})/\qp)$. We also think of $\chi$ as a character $\Gamma \xrightarrow{\sim} \zp^*$.
        From Section~\ref{section G}, we fix a topological generator $\gamma$ of $\Gamma$ such that $\chi(\gamma) = \zeta_{p-1}^a  (1+ p)$ for a fixed integer $a$.    
  \item $I_{\qp}$ is the inertia subgroup of $\mathrm{Gal}(\brqp/\qp)$  
  \item $\omega$ is the fundamental character of $I_{\qp}$ of level $1$;
        it has a canonical extension to $\mathrm{Gal}(\br{{\mathbb Q}}_p/\qp)$
  \item $\omega_2$ is the fundamental character of $I_{\qp}$ of level $2$;
        choose an extension of $\omega_2$ to $\mathrm{Gal}(\brqp/\bq_{p^2})$ so that for an integer $c$ with $p + 1 \nmid c$ the representation
        $\mathrm{ind}(\omega_2^c)$ obtained
        by inducing this extension from $\mathrm{Gal}(\brqp/\bq_{p^2})$ to $\mathrm{Gal}(\brqp/\qp)$ has determinant $\omega^c$
      \item $\calR_A$ is the Robba ring with coefficients in $A$, for $A$ 
        an affinoid algebra
  \item $\calR_A(\delta)$ for $\delta \in \sT(A)$ is the $(\varphi, \Gamma)$-module of rank $1$ over $\calR_A$ with action of $\varphi$ and $\Gamma$:
        \[
          \varphi f(T) = \delta(p)f((1 + T)^p - 1) \quad \text{ and } \quad  \gamma f(T) = \delta(\chi(\gamma)) f((1 + T)^{\chi(\gamma)} - 1) \text{ for } \gamma \in \Gamma.
        \]
      \item $k$ is an integer greater than or equal to $2$ and $r = k - 2$ 
      \item $v_{-}$ and $v_{+}$ are the largest and smallest non-negative integers, respectively, such that $v_{-} < (k - 2)/2 < v_{+}$ for $k \in [3, p + 1]$
      \item $H_0 = 0$ and $H_{l} = \sum_{i = 1}^{l}\frac{1}{i}$ is the $l^{\mathrm{th}}$ partial harmonic sum for $l \geq 1$;  write $H_{\pm} = H_{v_{\pm}}$ 
      \item $\nu = v_p (\sL - H_- - H_+)$
        is the $p$-adic valuation of $\sL$  in a finite extension of $\qp$ shifted by the partial harmonic sums $H_-$ and $H_+$.
        Note $\nu$ equals $v_p (\sL)$ if either quantity is negative.
  \item  $\mathbb{P}(V)$ is the projectivization of a vector space $V$
  \item $\phi$ is Euler's totient function
  \item $\phi_n(T)$ for $n \geq 1$ is the $p^{n}$-th cyclotomic polynomial.
\end{itemize}

\subsection{
  Limits of crystalline representations}

Colmez \cite{Col} and Chenevier \cite{Che} have constructed a moduli space of non-split trianguline
$(\varphi, \Gamma)$-modules of rank $2$ over the Robba ring (assuming that the
quotient of the characters occurring in the triangulation is not of a certain kind). The first goal
of this paper is to construct for $k \geq 3$ an explicit sequence of crystalline representations converging in this space to the (dual of the) semi-stable representation $V_{k, \sL}$ for a prescribed $\sL$-invariant $\sL$. We will use this in conjunction with a local constancy result 
to study the reductions of semi-stable representations. 

In order to state our result, let us recall the definition of the rigid-analytic space constructed by Colmez and Chenevier. Let $\sT$ be the parameter space for characters of $\qp^*$. Let $x : \qp^* \to \qp^*$ be the identity character. Let $\chi : \qp^* \to \zp^*$ is the $p$-adic cyclotomic character, sending $p$ to $1$
and such that $\chi |_{\zp^*}$ is the identity character. We call the characters $x^i\chi$ 
for $i \geq 0$ {\it exceptional}.
For each $i \geq 0$, let $F_i$ and $F_i'$ be the closed analytic subvarieties of $\sT \times \sT$ such that for every finite extension $E$ of $\qp$, we have
\[\arraycolsep=2pt
    \begin{array}{rl}
        F_i(E) \! & = \{(\delta_1, \delta_2) \in \sT(E) \times \sT(E) \text{ } \vert \text{ } \delta_1\delta_2^{-1} = x^i\chi\}, \\
        F_i'(E) \! & = \{(\delta_1, \delta_2) \in \sT(E)\times\sT(E) \text{ } \vert \text{ } \delta_1\delta_2^{-1} = x^{-i}\}.
    \end{array}
  \]
  Let $F = \cup_{i \geq 0} \> F_i$ and $F' = \cup_{i \geq 0} \> F_i'$.
The Colmez-Chenevier space $\til{\sT}_2$ is the blow-up of $(\sT\times\sT) \setminus F'$ along $F$ in the category of rigid-analytic spaces.
Our first main theorem is the following:

\begin{theorem}\label{any l invariant}
Let $k \geq 3$, $r = k-2$ and $\sL \in \mathbb{P}^1(E)$. For $n \geq 1$, let
\begin{equation}
   \label{kn an}
    (k_n, a_n) =
    \begin{cases}
        (k + p^n(p - 1), \> p^{r/2}(1 + \sL p^n (p - 1)/2)), & \text{if $\sL \not = \infty$} \\
        (k + p^{n^2}(p-1), \> p^{r/2}(1 + p^n)), & \text{if $\sL = \infty$}.
      \end{cases}
\end{equation} 
Then $$V^*_{k_n,a_n} \to V^*_{k, \sL},$$
  i.e., the sequence of points in $\til{\sT}_2$ associated to the crystalline representations $V^*_{k_n,a_n}$ 
  converges to the point in $\til{\sT}_2$ associated to the semi-stable representation $V^*_{k, \sL}$.
\end{theorem}


\newpage
\noindent {\bf Remarks.} 
%
%
\begin{enumerate}
\item The weights $k_n$ appearing in the theorem are for sufficiently large $n$ {\it exceptional} in the sense that they are two more than twice the valuation $v_p(a_n)$  
of $a_n$ modulo $(p-1)$ \cite{Gha}. This will be of key importance in what follows. 
  In fact, if $(k_n, a_n)$ is an arbitrary sequence of points such that $V^*_{k_n,a_n} \rightarrow  V^*_{k, \sL}$, then the
  $k_n$ are eventually exceptional weights. Indeed, 
  it is not hard to see that  $(k_n, a_n) \rightarrow (k,p^{r/2})$ and $k_n \equiv k \mod (p-1)$ for  large $n$.   

\item In the case $\sL \neq \infty$, the sequence of points $(k_n, a_n)$ in Theorem~\ref{any l invariant} lies on the line
$a_p(l) = p^{r/2}(1 + \frac{\sL}{2}(l - k))$. Clearly  
\begin{eqnarray}
  \label{log-derivative-intro}
  \sL = 2 \dfrac{ a'_p(k)}{a_p(k)}.
\end{eqnarray}
In Section~\ref{Comparing with Greenberg-Stevens}, we show that
\eqref{log-derivative-intro}
holds more generally for an arbitrary sequence $(k_n, a_n)$ 
of points on {\it any} smooth curve $a_p(l)$ such that $V^*_{k_n,a_n}
\rightarrow  V^*_{k, \sL}$
in the blow-up space $\til{\sT}_2$. The formula
\eqref{log-derivative-intro} is a variant of a classical formula initially
proved by Greenberg-Stevens \cite[Theorem 3.18]{GS} for elliptic curves
with split multiplicative reduction (a weight $2$, slope 0 case) and
extended by Stevens \cite[Theorem B]{Ste}
(see also Bertolini-Darmon-Iovita \cite[Theorem 4]{BDI}) to higher weights
and slopes. Further generalizations were proved by Colmez
\cite[Th\'eor\`eme 0.5, Corollaire 0.7]{Col10}, Benois \cite[Theorem
2]{Ben10} and others. This classical formula was a key local ingredient in the proof
of the Mazur-Tate-Teitelbaum conjecture for elliptic curves due to
Greenberg-Stevens.
Our proof of \eqref{log-derivative-intro}
is essentially geometric.
%
Recall the classical picture in algebraic geometry
  of the blow-up of ${\mathbb A}^2$ at a point:


\qquad \qquad {\centering
{\tiny
\begin{tikzpicture}
  \draw[<->] (-1.5, 0) -- (1.5,0) node[right] {$x$};
  \draw[<->] (-1.5, -0.75) -- (1.5,0.75) node[right, right, above] {$y$};
  \draw[-] (0.05,-0.10) -- (0.05,-0.1) node[right,  below] {\small{$\>\>\> (0,0)$}};
  \draw[<->] (-1.5, 3) -- (1.5, 3) node[right] {$x$};
  \draw[-] (0,2.97) -- (0,2.97) node[] {$\bullet$};
  \draw[-] (-0.3,2.93) -- (-.3,2.93) node[right,  below] {\small{$\>\>\> 0$}};
  \draw[<->] (-1.5, 4.00) -- (1.5,5.50) node[right, right, above] {$y$};
   \draw[-] (0,4.75) -- (0,4.75) node[] {$\bullet$};
   \draw[-] (-.3,5.08) -- (-.3,5.08) node[right,  below] {\small{$\> \infty$}};
  \draw[->, thick, green] (0, 1.4) -- (0,0.6) node[above] {}; 
  \draw[-] (0,0) -- (0,0) node[] {$\bullet$}; 
  \draw[-, thick] (0, 1.8) -- (0,5.2) node[above] {}; 
  \draw[scale=1.3, domain=-.85:.85, smooth, variable=\x, blue] plot ({\x}, {\x*\x*\x+0.15*\x}); 
  \draw[-, blue] (1.4,1.4) -- (1.4,1.4) node[] {\small{$y=f(x)$}};
  \draw[scale=1.3, domain=-.9:.9, smooth, variable=\x, blue] plot ({\x}, {1.5 * \x*\x*\x+0.05*\x+2.78}); 
  \draw[-] (0,3.6) -- (0,3.6) node[] {$\bullet$};
  \draw[red] (1.6, 4.6) -- (1.6,4.6) node[right] {\small{$\dfrac{b}{a} =  \dfrac{f(x)}{x}$}};
  \draw[-, red] (-0.6,3.6) -- (-0.6,3.6) node[above] {\small{$\therefore f'(0)$}};
\end{tikzpicture}
} \qquad
{\small
\begin{tikzpicture}
  \draw[-] (0, 0) -- (0,0) node[] {\small{Blow-up of ${\mathbb A}^2$ at $(0,0)$}};
  \draw[-] (0, 1.1) -- (0,1.1) node[] {$\{(x,y) \in {\mathbb A}^2\}$}; 
  \draw[->, thick, green] (0, 4.5) -- (0,1.7) node[right]  {}; 
  \draw[-] (0, 5.0) -- (0, 5.0) node[] {$\{(x,y, a:b) \in {\mathbb A}^2  \times {\mathbb P}^1 : xb = ya\}$};   
\end{tikzpicture}
}
}


The strict transform 
of the curve $y = f(x)$ passing through the origin $(0,0)$
in ${\mathbb A}^2$ passes through the exceptional divisor ${\mathbb P}^1$ above the origin at `height'
the derivative of $f$ at $0$. The proofs of \eqref{log-derivative-intro}
and Theorem~\ref{any l invariant} are based on an analogous principle in a rigid analytic setting.

\item The techniques used to prove  Theorem~\ref{any l invariant}  can also be used to prove that the limit of a sequence of irreducible two-dimensional crystalline representations of  $\mathrm{Gal}(\br{{\mathbb Q}}_p/\qp)$ with Hodge-Tate weights belonging to an
interval $[a, b]$ is also irreducible crystalline with Hodge-Tate weights in $[a, b]$, at least if the difference of the Hodge-Tate weights of the representations in the sequence is at least $2$ infinitely often. This gives another (geometric) proof of a special case of a
general result of Berger \cite[Th\'eor\`eme 1]{Ber04} 
(see Section~\ref{section further implication}). 
\end{enumerate}

\subsection{Computing the $\sL$-invariant}
  \label{computing L}

  In this section, we explain the techniques involved in proving Theorem~\ref{any l invariant}. The discussion should also serve as an overview of the contents
  of Sections~\ref{blow-up} to \ref{Section proof of main theorem}. 

Let $E$ be a finite extension of $\qp$. Let $\calR_E$ be the Robba ring over $E$ consisting of bidirectional power series having coefficients in $E$ that converge on the elements of $\brqp$ with valuation in $]0, M]$, for some $M > 0$. For a more precise description, see Section \ref{link robba ring here}.


The $E$-valued points of $\til{\sT}_2$ are
tuples $(\delta_1, \delta_2, L)$ where $\delta_1$, $\delta_2$ are $E^*$-valued characters of $\qp^*$ (with $\delta_1\delta_2^{-1} \neq x^{-j}$ for any $j \geq 0$)
and $L \in  \mathbb{P}^1(E)$ is the $\sL$-invariant 
of the $(\varphi,\Gamma)$-module associated to the isomorphism class of the non-split extension
$$0 \rightarrow \calR_E(\delta_1) \rightarrow * \rightarrow \calR_E(\delta_2) \rightarrow 0$$
when $\delta_1\delta_2^{-1} = x^i\chi$ for $i \geq 0$ is an exceptional character, and is taken to be $\infty$ otherwise (more precisely,
in the former case, $L$ is the $\sL$-invariant defined by Colmez - see Section~\ref{link robba ring here} - of this extension twisted by $\delta_2^{-1}$).

The functor $\Dr$ sets up an equivalence of categories between the category of $E$-linear representations of $\mathrm{Gal}(\br{\mathbb{Q}}_p/\qp)$ and the
category of $(\varphi, \Gamma)$-modules over $\calR_E$ of slope $0$. Let $k \geq 3$ and $r = k-2$, and for $n \geq 1$, let  $(k_n, a_n)$ be as in \eqref{kn an}.
By \cite[Proposition $3.1$]{Ber}, $\Dr(V_{k_n, a_n}^*)$ is an extension
$$0 \rightarrow \calR_E(\mu_{y_n}) \rightarrow  \Dr(V_{k_n, a_n}^*) \rightarrow  \calR_E(\mu_{1/y_n}\chi^{1 - k_n}) \rightarrow 0,$$
where
\begin{equation}
    \label{yn}
    y_n = \frac{a_n + \sqrt{a_n^2 - 4p^{k_n - 1}}}{2}.
\end{equation}
This allows us to associate to  (the dual of) $V_{k_n, a_n}$ 
the point $(\mu_{y_n}, \mu_{1/y_n}\chi^{1 - k_n},\infty)$ of the blow-up $\til{\sT}_2$.
 We claim that for $n \geq 1$ this sequence of points converges in the blow-up
 to the point
 \begin{eqnarray*}
   \begin{split}
     (\mu_{p^{r/2}}, \mu_{1/p^{r/2}}\chi^{1 - k},-\sL), &  \qquad & \text{ if } \sL \neq \infty, \\
     (\mu_{p^{r/2}}, \mu_{1/p^{r/2}}\chi^{1 - k},\infty), & \qquad  & \text{ if } \sL = \infty.
   \end{split}
 \end{eqnarray*}
It turns out that the corresponding $(\varphi, \Gamma)$-module is \'etale (for any $\sL$) and that the corresponding
Galois representation is the semi-stable representation $V^*_{k, \sL}$  (see the end of Section~\ref{proof of main theorem}).
In the $\sL = \infty$ case, the representation $V_{k,\sL}^*$ is in fact crystalline. This proves Theorem~\ref{any l invariant}.



It remains to prove the claim. 
Let $\til{\sT}$ be the blow-up of $\sT\setminus\{x^{-j}\}_{j \geq 0}$ at $\{x^i \chi\}_{i \geq 0}$.
The $E$-valued points of $\til{\sT}$ are
tuples $(\delta, L)$ where $\delta$ is an $E^*$-valued character of $\qp^*$ (with $\delta \neq x^{-j}$ for any $j \geq 0$)
and $L \in  \mathbb{P}^1(E)$ is the $\sL$-invariant (defined by Colmez, see Section~\ref{link robba ring here})
of the $(\varphi,\Gamma)$-module associated to the isomorphism class of the non-split extension
$$0 \rightarrow \calR_E(\delta) \rightarrow * \rightarrow \calR_E  \rightarrow 0$$
when $\delta = x^i\chi$ for $i \geq 0$ is an exceptional character, and is taken to be $\infty$ otherwise.
Now a sequence of points $(\delta_{1,n}, \delta_{2,n}, L_n)$ in $\til{\sT}_2$ converges to a point $(\delta_1, \delta_2, L)$ in $\til{\sT}_2$ if and only if
\begin{itemize}
  \item $\delta_{1,n}$ and $\delta_{2,n}$ converge to $\delta_1$ and $\delta_2$, respectively, in $\sT$, and
  \item $(\delta_{1,n} \delta_{2,n}^{-1}, L_n)$ converges to $(\delta_1 \delta_2^{-1}, L)$ in $\til{\sT}$.
 \end{itemize}
 That is, with respect to the following commutative diagram 
\[
    \begin{tikzcd}[row sep = 0.3cm, column sep = 1.2cm]
        \til{\sT}_2  \arrow{dd} \arrow{r} & \til{\sT}  \arrow{dd} \\ & \\
        \sT \times \sT \setminus F' \arrow{r} & \sT \setminus \{x^{-j}\}_{j \geq 0}
    \end{tikzcd}
\]
where the bottom map is the restriction of the twisting map  $\sT \times \sT \to \sT$ sending $(\delta_1, \delta_2)$ to $\delta_1\delta_2^{-1}$,
the top map sends $(\delta_1, \delta_2, L)$ to $(\delta_1 \delta_2^{-1},L)$ and the vertical maps are the blow-up maps at $F$ and $\{x^i\chi\}_{i \geq 0}$, respectively,
the sequence  $(\delta_{1,n}, \delta_{2,n}, L_n)$ converges to $(\delta_1, \delta_2, L)$ in $\til{\sT}_2$ if and only if the projections of this
sequence under the left vertical map and the top horizontal map  converge to the corresponding projections of $(\delta_1, \delta_2, L)$.
Indeed, the top map is obtained from the
universal property of the blow-up map on the right (this can be checked on charts using the definition of the map $g$ constructed
after Lemma~\ref{zero div}), so there is an induced map  from $\til{\sT}_2$ to the fiber product of  $\sT \times \sT \setminus F'$ and
$\til{\sT}$ over $\sT \setminus \{x^{-j}\}_{j \geq 0}$ which is a closed immersion (\cite[Proposition 3.1.2]{Sch}), hence induces an inclusion on points
$\til{\sT}_2(E)$ to the fiber product of $(\sT \times \sT \setminus F')(E)$ and $\til{\sT}(E)$ over $(\sT \setminus \{x^{-j}\}_{j \geq 0})(E)$ with closed
image, and then this is the definition of convergence in the last fiber product.


An easy check shows that
the characters $\mu_{y_n}$ and $\mu_{1/y_n}\chi^{1 - k_n}$
converge to the characters $\mu_{p^{r/2}}$ and $\mu_{1/p^{r/2}}\chi^{1 - k}$ respectively. The ratio of these characters is the
exceptional character $\mu_{p^r}\chi^{k - 1} = x^{r}\chi$. Thus, if the sequence $(\mu_{y_n}, \mu_{1/y_n}\chi^{1 - k_n},\infty)$ converges in $\til{\sT}_2$,
then it converges to a point in the fiber over $F_r$.  Thus, it remains to check that $(\mu_{y_n^2} \chi^{k_n-1},\infty)$ converges in $\til{\sT}$ to
\begin{eqnarray}
  \label{limit in tildeT}
  \begin{split}
    (x^r\chi, -\sL), & \qquad &  \text{ if } \sL \neq \infty, \\
    (x^r\chi, \infty), & \qquad & \text{ if } \sL = \infty.
  \end{split}
\end{eqnarray}

In order to prove \eqref{limit in tildeT}, we set up local coordinates $\U$ around the point $x^r \chi$, describe the
blow-up $\tU$ of this coordinate patch with center at the exceptional character $x^r \chi$ explicitly (see Section~\ref{blow-up} for details)
and compute the limit in $\tU$.
%
%
%
Let $\zetap$ be a fixed primitive $(p - 1)^\mathrm{th}$ root of unity. Associating the tuple $(\delta(p), \delta(\zetap), \delta(1 + p) - 1)$ to a character $\delta \in \sT(\qp)$ identifies $\sT(\qp)$ with $\qp^* \times \mu_{p - 1} \times p\zp$. Under this identification, the exceptional character $\mu_{p^r}\chi^{k - 1}$ goes to the tuple $(p^r, \zetap^{k - 1}, (1 + p)^{k - 1} - 1)$.
The set $p^r\zp^* \times \{\zetap^{k - 1}\} \times p\zp$ is a neighborhood of $\mu_{p^r}\chi^{k - 1}$ in $\sT(\qp)$. This leads us to consider the following affinoid algebra
\[
\U = \text{Sp } \qp \langle S_1, S_2, T_1, T_2, T_3 \rangle /(p^rT_1 - S_1, 1 - T_1T_2, pT_3 - S_2)
\]
as a neighbourhood of $\mu_{p^r}\chi^{k - 1}$ in $\sT$ because clearly $\U(\qp) = p^r\zp^* \times p\zp$. The variable $S_1$ corresponds to the first factor and $S_2$ to the second factor. From now on, by fixing the tame part of the characters under consideration, we identify $\U(E)$ with the subset $p^r\co_{E}^* \times \{\zetap^{k-1}\} \times p\co_{E}$ of $\sT(E)$,
where $\co_E$ is the ring of integers of $E$.

The character $\mu_{p^r}\chi^{k - 1} = x^r \chi$ corresponds to the maximal ideal $$m = (S_1 - p^r, S_2 - ((1 + p)^{k - 1}-1))$$ of $\co(\U)$.
The blow-up $\tU$ of $\U$ at the maximal ideal $m$ turns out to have the following standard description (see \eqref{points-blow-up}):
\[
  \tU(E) = \{(s_1, s_2,\> \xi_1 : \xi_2) \in \U(E) \times \mathbb{P}^1(E) \text{ } \vert \text{ } (s_1 - p^r)\xi_2 = (s_2 - ((1 + p)^{k-1}-1))\xi_1\}.
\]

For large $n$, the  points  $(\mu_{y_n^2} \chi^{k_n-1},\infty)$ in $\til{\sT}$ 
lie in $\tU$. We prove that the sequence converges in the blow-up $\tU$ to the point (see Section~\ref{proof of main theorem}) \\

\qquad \qquad
\begin{tabular}{l l} 
  $\left(p^r, (1 + p)^{k - 1} - 1, \> \sL \> \dfrac{p^r}{(1 + p)^{k - 1} \> \mathrm{log}(1 + p)} : 1\right)$, & \text{if $\sL \not = \infty$}  \\ [7pt]
  $(p^r, (1 + p)^{k - 1} - 1, \> 1 : 0)$, & \text{if $\sL = \infty$}, 
\end{tabular}
\\ [2pt]

\noindent where log is normalized so that $\log(p) = 0$.   
The proof of \eqref{limit in tildeT} 
then follows immediately from Theorem~\ref{The L invariant}, a technical but important formula for the $\sL$-invariant of a point
in the exceptional fiber, noting that the fudge factor there cancels with the extra factor
appearing in the third coordinate of the limit point above when $\sL \neq \infty$ and flips the sign.


Theorem~\ref{The L invariant} is proved as follows. Given a point in the exceptional fiber, we convert it to a tangent direction in $\U$ at the point $(p^r, (1 + p)^{k - 1} - 1)$, i.e., an element of $\mathbb{P}(\mathrm{Hom}(m/m^2\otimes_{\qp} E, E))$. This is done using the map in Proposition~$\ref{fiber and tangents}$. We then explicitly describe the isomorphism $\mathbb{P}(\mathrm{Hom}(m/m^2\otimes_{\qp}E, E)) \to \mathbb{P}(H^1(\calR_{E}(x^r\chi)))$ stated in \cite[Theorem 2.33]{Che}, using some preparatory material on the cohomology of `big' $(\varphi, \Gamma)$-modules in Section~\ref{big phi gamma modules}. The image of the given point in the exceptional fiber under the composition of these two maps yields a cohomology class in $H^1(\calR_{E}(x^r\chi))$
(up to scalars).
The corresponding $(\varphi, \Gamma)$-module (up to isomorphism) is referred to as `the $(\varphi,\Gamma)$-module' associated to the given point (Definition~\ref{the phi,Gamma module}).
We then represent this cohomology class as an explicit linear combination of the basis elements of $H^1(\calR_{E}(x^r\chi))$ studied by Benois
in \cite[Proposition 1.5.4]{Ben}. The original formula for the $\sL$-invariant due to Colmez is, however, in terms of a different basis of
$H^1(\calR_{E}(x^r\chi))$, namely, the one constructed in \cite[Proposition 2.19]{Col} (see Section~\ref{link robba ring here}). Restating the
formula for the $\sL$-invariant in terms of Benois' basis (Definition~\ref{def of L inv}) allows us to  give a formula for the $\sL$-invariant of
the given point in the exceptional fiber.

\subsection{Reductions of semi-stable representations}

In \cite[Proposition 3.9]{Che}, Chenevier proved that points of the space $\til{\sT}_2$ lie in families of $(\varphi, \Gamma)$-modules. By Kedlaya-Liu \cite[Theorem 0.2]{KL10}, such
a family comes from a family of Galois representations at least affinoid locally around an \'etale point. Furthermore (see, e.g., the discussion on \cite[p. 1513]{Che}, which uses results of \cite{Che14} on pseudorepresentations), the semi-simplification of the reduction of Galois representations living in connected families are isomorphic. This means that if the points in $\til{\sT}_2$ corresponding to two Galois representations are close, then the semi-simplifications of the reductions of the two Galois representations are the same. Moreover, the dual of the reduction of a lattice in a $p$-adic representation is the same as the reduction of the dual lattice in the dual representation. Using these facts, along with Theorem~\ref{any l invariant}, we see that if one knows the reductions of the crystalline representations appearing in Theorem~\ref{any l invariant}, then one can compute the reduction of $V_{k,\sL}$ for any $k \geq 3$ and $\sL \neq \infty$ (the case $k = 2$ is not as interesting since
the reduction is always reducible). 

More generally, the above method allows one to compute the reduction of any irreducible two-dimensional non-crystalline semi-stable representation with distinct Hodge-Tate weights. Indeed,
suppose $V$ is such a semi-stable representation. Twisting by $\chi^a$ for some integer $a$, we may assume that the Hodge-Tate weights of the semi-stable representation $V \otimes \chi^a$ are $(0, k - 1)$ for an integer $k \geq 2$. By, for instance, Guerberoff-Park \cite[Lemma 3.1.2 (3)]{GP}, the filtered $(\varphi, N)$-module  $D_\mathrm{st}((V \otimes \chi^a)^*)$
is the module $D(\lambda, \sL)$
described in \cite[Example 3.1.1]{GP} with $\lambda = u p^{\frac{k-2}{2}}$ for some unit $u$ (since $r$ there is equal to $k - 1$).
Now $V \otimes \chi^a \otimes \mu_u$ is isomorphic to $V_{k, \sL}$ as can be seen by comparing the
corresponding filtered $(\varphi, N)$-modules.
By \cite[Lemma 3.1.2 (4)]{GP}, we must have $k \geq 3$.

This approach to computing the reduction of semi-stable representations using crystalline representations is of some importance because the reductions of these two classes of
representations are nowadays largely studied by completely different methods: the crystalline case uses the compatibility of reduction between
the $p$-adic and mod $p$ Local Langlands correspondences, or computes the reduction of the corresponding Wach module, whereas the reductions in the semi-stable case are
determined by studying the reductions of the corresponding strongly divisible modules.
In our experience, the former methods, while quite intricate,
are not as complicated as the latter method. Thus, in view of the remarks above, 
the techniques used in the crystalline case may be brought to bear on the study of the reductions of semi-stable representations. We note, however, that the former
method is only available for two-dimensional representations of $\mathrm{Gal}(\brqp/\qp)$, whereas the latter method is available in principle for
representations of $\mathrm{Gal}(\brqp/\qp)$ of any dimension (though in practical terms only for those of small Hodge-Tate weights).

\par Let us illustrate this with some examples. The reductions of semi-stable representations have been computed completely for even weights in the range $[2, p - 1]$ by Breuil and M\'ezard \cite{BM}, and for odd weights in the same range by Guerberoff and Park \cite{GP} at least on inertia. In \cite{Gha}, the second author made the following conjecture called the zig-zag conjecture\footnote{This version is mildly different from
   \cite[Conjecture 1.1]{Gha} in that there we require $k$ to be sufficiently far away
   from some weights which are strictly larger than $p+1$, whereas here
   $k$ is required to be sufficiently close to the weights $3 \leq k_0 \leq p+1$.}
 describing the reductions of crystalline representations of exceptional weights and half-integral
slopes in terms of an alternating sequence of reducible and irreducible mod $p$ representations. 

\begin{conjecture}[Zig-Zag Conjecture] 
  \label{zig-zag-conj}
  Say that 
  $k \equiv k_0 = 2v(a_p)+2 \mod (p-1)$ is an exceptional congruence class of weights
  for a particular half-integral slope $\frac{1}{2} \leq v_p(a_p) \in \frac{1}{2} {\mathbb Z} \leq \frac{p-1}{2}$. 
  Let $r = k-2$ and $r_0 = k_0-2$.
  Define two parameters:
  \begin{eqnarray*}
    \label{c}
     \tau =   v_p\left(\frac{a_p^2 - \binom{r-v_-}{v_+} \binom{r - v_+}{v_-} p^{r_0}}{p a_p}\right) & \text{and} &
      t  =  v_p(k - k_0),
 \end{eqnarray*}
 where $v_-$ and $v_+$ are the largest and smallest integers 
 such that $v(a_p)$ lies in $(v_-, v_+)$. 
 Then, for all weights $k > k_0$  with $t$ sufficiently large,
 the (semi-simplification of
 the) reduction $\bar{V}_{k,a_p}$ of the crystalline representation 
 $V_{k,a_p}$ on the inertia subgroup $I_{\qp}$ is given by:
 \begin{eqnarray*}
  \bar{V}_{k, a_p} |_{I _{\qp}} & \sim & 
  \begin{cases}
    \begin{array}{l}
      \mathrm{ind}(\omega_2^{r_0+1}),    
    \end{array}                                                              & \text{ if } \tau < t \\
    \begin{array}{l}
      \omega^{r_0}\,\oplus\,\omega, 
    \end{array}                                                              & \text{ if } \tau = t \\ 
    \begin{array}{l}
      \mathrm{ind}(\omega_2^{r_0+p}),     
    \end{array}                                                              & \text{ if } t < \tau < t+1 \\
    \begin{array}{l}
      \omega^{r_0-1}\,\oplus\, \omega^2, 
    \end{array}                                                              & \text{ if } \tau = t + 1 \\ 
    \begin{array}{l}
      \mathrm{ind}(\omega_2^{r_0+2p-1}),
    \end{array}                                                              & \text{ if } t +1  < \tau < t+2      \\
    \begin{array}{l} 
      \omega^{r_0-2}\,\oplus\,\omega^3, 
    \end{array}                                                              & \text{ if } \tau = t + 2\\
    \>\> \qquad\quad \vdots                                                                                                  &  \> \qquad \vdots \\
     \begin{array}{l}
       \mathrm{ind}(\omega_2^{r_0 + 1 + \frac{r_0-2}{2}(p-1) }), \\
        \omega^{\frac{r_0+2}{2}}\,\oplus\,\omega^{\frac{r_0}{2}},  \\     
      \mathrm{ind}(\omega_2^{r_0+1+\frac{r_0}{2}(p-1) }),                                                             
    \end{array} & 
    \begin{array}{l}
       \text{if } t + \frac{r_0-4}{2} < \tau < t + \frac{r_0-2}{2} \\
       \text{if } \tau = t + \frac{r_0-2}{2}  \\
       \text{if } \tau > t + \frac{r_0-2}{2}, 
    \end{array}
    \large{\Bigg\}}   \text{ and } r_0  \text{ is even},   \\[2pt]
    \qquad \qquad           \text{or}         &  \\[2pt]
    \begin{array}{l}
       \mathrm{ind}(\omega_2^{r_0 + 1 + \frac{r_0-1}{2}(p-1) }),  \\
       \omega^{\frac{r_0 + 1}{2}} \,\oplus\, \omega^{\frac{r_0+1}{2}},
    \end{array}                                                              &
    \begin{array}{l}
      \text{if } t + \frac{r_0-3}{2} < \tau < t + \frac{r_0-1}{2} \\
      \text{if } \tau \geq t + \frac{r_0-1}{2},
    \end{array}
    \large{\Bigg\}}  \text{ and }  r_0 \text { is odd}.
    \end{cases} 
\end{eqnarray*}
\end{conjecture}

This conjecture has been verified for some small slopes (cf. \cite{BG13} for slope $1/2$, \cite{BGR18} for slope $1$ and \cite{GR} for slope $3/2$)
even with $t = 0$.\footnote{Although, the condition that $t$ is sufficiently large is required for larger slopes due to some numerical observations made by Rozensztajn.}

This ``crystalline" conjecture is intimately connected to the ``semi-stable'' results in \cite[Theorem 1.2]{BM} and \cite[Theorem 5.0.5]{GP}.
More precisely, the zig-zag conjecture and Theorem \ref{any l invariant} can be used to completely recover the description of the reductions of semi-stable
representations in these theorems when
$k \neq  2$,  and even (conjecturally) extend it to the cases $k =p+1$ and $k=p$, respectively,
at least on the inertia subgroup (see Theorem~\ref{thm from conj} below). Moreover, in the odd weight cases for which the zig-zag conjecture has been proved,
we obtain new information about the reductions of semi-stable representations on the full Galois group $\mathrm{Gal}(\brqp/\qp)$.

To elaborate further, let us set up some notation. For an integer $k$ in the interval $[3, p + 1]$, define $v_{-}$ and $v_{+}$ to be the largest and smallest integers,
respectively, such that $v_{-} < (k - 2)/2 < v_{+}$. For $l \geq 1$, let $$H_{l} = \sum_{i = 1}^{l}\frac{1}{i}$$ be the
$l$-th partial harmonic sum and set $H_0 = 0$. For convenience, write
$H_{-} = H_{v_-}$ 
and $H_{+} = H_{v_+}$. 
 For any $\sL$ in a finite extension of $\qp$, let
\[
    \nu = v_p (\sL - H_- - H_+)
\]
be the $p$-adic valuation 
of the $\sL$-invariant shifted by the partial harmonic sums $H_-$ and $H_+$.   
Let $\omega$ and $\omega_2$ denote the mod $p$ fundamental characters of levels $1$ and $2$, respectively, on the inertia group  $I_{\qp}$. The character $\omega$
may be thought of as a character of  $\mathrm{Gal}(\brqp/\qp)$; we choose an extension of
$\omega_2$ to $\mathrm{Gal}(\brqp/\mathbb{Q}_{p^2})$ such that 
for any integer $c$ with $p + 1 \nmid c$, the representation $\mathrm{ind}(\omega_2^c)$ obtained by inducing $\omega_2^c$ from $\mathrm{Gal}(\brqp/\mathbb{Q}_{p^2})$ to $\mathrm{Gal}(\brqp/\qp)$ 
has determinant $\omega^c$. Let $\mu_{\lambda}$ be the unramified character of $\mathrm{Gal}(\brqp/\qp)$ mapping a geometric Frobenius
element at $p$ to $\lambda \in \br{\mathbb{F}}_p^*$ or $\brqp^*$.
Normalizing local class field theory by sending $p$ to a geometric Frobenius element at $p$,  
we obtain the character $\mu_\lambda$ of $\qp^*$ taking $p$ to $\lambda$ and $\zp^*$ to $1$ used in Section~\ref{computing L}.  

Our second main theorem is the following:

\begin{theorem}\label{thm from conj}
If the zig-zag conjecture is true, then for any weight $k$ satisfying $3 \leq k \leq p + 1$, we have the following $(k-1)$-fold description
of the (semi-simplification of the) reduction of the semi-stable representation $V_{k, \sL}$ for $\sL \in
{\mathbb P}^1(\br{\mathbb Q}_p)$
on the inertia subgroup:
\[
    \br{V}_{k, \sL}\vert_{I_{\qp}} \sim
    \begin{cases}
        \omega_2^{k - 1} \oplus \omega_2^{p(k - 1)}, & \text{if $\nu < 1 - \frac{k - 2}{2}$} \\[2pt]
        \omega^{k - 2} \oplus \omega, & \text{if $\nu = 1 - \frac{k - 2}{2}$} \\[2pt]
        \omega_2^{k - 2 + p} \oplus \omega_2^{p(k - 2) + 1}, & \text{if $1 - \frac{k - 2}{2} < \nu < 2 - \frac{k - 2}{2}$} \\[2pt]
        \omega^{k - 3} \oplus \omega^2, & \text{if $\nu = 2 - \frac{k - 2}{2}$} \\[2pt]
        \>\> \qquad \vdots & \qquad \vdots \\[2pt]

        \omega_2^{\frac{k}{2}+1 + p\left(\frac{k}{2} - 2\right)} \oplus \omega_2^{\frac{k}{2} - 2 + p\left( \frac{k}{2} + 1 \right)}, & \text{if $-1 < \nu <  0$} \\[-8pt]
        \omega^{\frac{k}{2}} \oplus \omega^{\frac{k - 2}{2}}, & \text{if $\nu = 0$}  \qquad \qquad {\Large{\Bigg{\}}}} \text{ and k is even,}\\[-8pt]
        \omega_2^{\frac{k}{2} + p\left(\frac{k}{2} - 1\right)} \oplus \omega_2^{\frac{k}{2} - 1 + p\frac{k}{2}}, & \text{if $\nu > 0$}, \\[2pt]
        \qquad \text{or} & \\[4pt]
        \omega_2^{\frac{k+1}{2} + p\left(\frac{k-3}{2}\right)} \oplus \omega_2^{\frac{k-3}{2} + p\left( \frac{k+1}{2}  \right)}, & \text{if $-\frac{1}{2} < \nu <  \frac{1}{2}$} \\[-16pt]
            &  \>\> \quad \qquad \qquad \qquad {\Large{\Bigg{\}}}} \text{ and $k$ is odd.} \\[-16pt]
        \omega^{\frac{k - 1}{2}} \oplus \omega^{\frac{k - 1}{2}}, & \text{if $\nu \geq \frac{1}{2}$}, 
  \end{cases}
\]
\end{theorem}


\noindent {\bf Remark. } The theorem holds when $\sL = \infty$ even without assuming the zig-zag conjecture by the work of Fontaine and Edixhoven \cite{Edi}, since, by convention, $v_p(\infty) = -\infty$. \vspace{.1cm}

A generalization of the $\nu < 1 - \frac{k-2}{2}$ case of Theorem~\ref{thm from conj} has recently been proved for all weights $k \geq 4$
(and odd primes $p$) by Bergdall-Levin-Liu \cite[Theorem 1.1]{BLL22}. Note that the term $v_p((k-2)!)$ in their result vanishes in our setting.

Also Theorem~\ref{thm from conj} above indeed matches with the results in \cite{BM} and \cite{GP} for $\sL \neq \infty$ and for $k \in [3,p-1]$: 
\begin{itemize}
    \item \underline{Even $k \in [3, p - 1]$}:
    In \cite[Theorem 1.2]{BM},  Breuil and M\'ezard have computed the reduction of $V_{k, \sL}$  in terms of the valuations $v_p(a)$ and $v_p(\sL)$, where
    \[
        a = (-1)^{k/2}\left(-1 + \frac{k}{2}\left(\frac{k}{2} - 1\right)(-\sL + 2H_{k/2 - 1})\right).
    \]
    Since $k \in [3, p - 1]$, we see that $v_p\left(\frac{k}{2}\left(\frac{k}{2} - 1\right)\right) = 0$, and therefore
\begin{eqnarray*}
    v_p(a) & = & v_p\left(\frac{-1}{(k/2)(k/2 - 1)} + (-\sL + 2H_{k/2 - 1}) \right) \\
    & = & v_p\left(\frac{1}{k/2} - \frac{1}{k/2 - 1} - \sL + 2\sum_{i = 1}^{k/2 - 1}\frac{1}{i}\right)  \\
    & = & v_p\left(-\sL + \sum_{i = 1}^{k/2}\frac{1}{i} + \sum_{i = 1}^{k/2 - 2}\frac{1}{i}\right)  \\
    & = & v_p(\sL - H_{k/2} - H_{k/2 - 2}) \\
    &  = &  \nu, 
\end{eqnarray*}
where we have used $k$ is even in the last equality. Now: 
\begin{itemize}
    \item If $\nu > 0$, then Theorem \ref{thm from conj} yields $\br{V}_{k, \sL}\vert_{I_{\qp}} \sim \omega_2^{k/2 + p(k/2 - 1)} \oplus \omega_2^{k/2 - 1 + pk/2}$, which agrees with \cite[Theorem 1.2 (ii)]{BM}.
    \item If $\nu = 0$, then Theorem \ref{thm from conj} yields $\br{V}_{k, \sL}\vert_{I_{\qp}} \sim \omega^{k/2} \oplus \omega^{(k - 2)/2}$, which agrees with \cite[Theorem 1.2 (i)]{BM}.
    \item So assume $\nu < 0$. Then $\nu = v_p(\sL)$. If $\nu < 2 - k/2$, then Theorem \ref{thm from conj} yields $\br{V}_{k, \sL}\vert_{I_{\qp}} \sim \omega_2^{k - 1} \oplus \omega_2^{p(k - 1)}$, which agrees with \cite[Theorem 1.2 (iii)]{BM}. Now if $2 - k/2 \leq \nu < 0$ and $\nu \in \mathbb{Z}$, then \cite[Theorem 1.2 (iii)]{BM} yields $\br{V}_{k, \sL}\vert_{I_{\qp}} \sim \omega^{k/2 - \nu} \oplus \omega^{k/2 + \nu - 1}$, which is the same as the reduction computed in Theorem \ref{thm from conj}. Finally,  if $2 - k/2 \leq \nu < 0$ and $\nu \not \in \mathbb{Z}$, then \cite[Theorem 1.2 (iii)]{BM} yields $\br{V}_{k, \sL}\vert_{I_{\qp}} \sim \omega_2^{k/2 - \lfloor \nu \rfloor + p(k/2 + \lfloor \nu \rfloor  - 1)} \oplus \omega_2^{k/2 + \lfloor \nu \rfloor - 1 + p(k/2 - \lfloor \nu \rfloor)}$, which matches with the reduction computed in Theorem \ref{thm from conj}.
    \end{itemize}
    
        \noindent {\bf Remark. }
    Since the zig-zag conjecture has already been proved for $p \geq 5$ and
    slope $1$ in \cite[Theorem 1.1]{BGR18}, even on $\mathrm{Gal}(\br{{\mathbb Q}}_p/\qp)$, we recover  the following result of Breuil and M\'ezard when $k = 4$ (cf. \cite[Theorem 1.2]{BM}):\footnote{The computation
\[
    \br{a}\br{\left(\frac{a_p}{p}\right)} = \br{(-1 + 2(-\sL + 2))}\br{\left(\frac{p}{p}\right)} = \br{-2 \left(\sL - \frac{3}{2} \right)}
\]
shows that the reduction agrees with the reduction computed in \cite[Theorem 1.2]{BM}.}
    if $p \geq 5$ and $k = 4$, then 
    \[
        \br{V}_{k, \sL} \sim
        \begin{cases}
            \mathrm{ind}(\omega_2^{3}), & \text{if $\nu < 0$} \\[5pt]
            \mu_{\lambda}\>\omega^{2} \oplus \mu_{\lambda^{-1}}\>\omega, & \text{if $\nu = 0$} \\[5pt]
            \mathrm{ind}(\omega_2^{2 + p}), & \text{if $\nu > 0$},
        \end{cases}
    \]
    where 
    \[
        \lambda = \br{-2\left(\sL - \frac{3}{2}\right)}.
    \]
    If all nearby \'etale points in $\til{\sT}_2$ close to a given \'etale point, which has a lattice
    with non-split reduction, also have lattices with isomorphic reduction (an assumption which is stronger than the
      reductions being isomorphic up to semi-simplifcation, and which might follow
      by extending results of \cite{Che14} from pseudorepresentations to Galois representations),
      then one can read off more subtle information
      about whether the representation
      $\br{V}_{k,\sL}$ is peu or tr\`es ramifi\'ee (when it is reducible and $\lambda = \pm 1$) from the corresponding
      information of a sufficiently close
      crystalline representation $V_{k_n, a_n}$, 
      which, in turn, is controlled by the size of $v_p(u_n-\varepsilon_n)$ in the notation of \cite[Theorem 1.3]{BGR18}.
      Indeed, in the middle case of the trichotomy above, we have $v_p(3-2\sL) = 0$, putting us in
      part (i) of \cite[Theorem 1.2]{BM}. Note that $\br{\frac{a_n}{p}} = 1$ and $\br{k_n-2} = 2$. There are now two
      cases to consider depending on whether $\varepsilon_n = \pm 1$.
      If $v_p(\sL - 2) = 0$, then a small check shows that $\varepsilon_n = 1$, so by
      \cite[Theorem 1.3 (1) (a)]{BGR18},  the reduction $\br{V}_{k,\sL}$ (without semi-simplification)
      is peu ramifi\'ee. If $v_p(\sL-2) > 0$, then another small check shows
      that $\varepsilon_n = -1$ and that $v_p(u_n-\varepsilon_n) = v_p(\sL-2)$. Thus, by \cite[Theorem 1.3 (1) (b)]{BGR18},
      the reduction $\br{V}_{k,\sL}$ (without semi-simplification) is
      peu ramifi\'ee if and only if $v_p(\sL-2) < 1$, at least if $\sL$ lies in an unramified extension
      of $\qp$. Both these conclusions are consistent with
      the corresponding conclusions in  \cite[Theorem 1.2 (i)]{BM}!
        

\item \underline{Odd $k \in [3, p -1]$}:
 In \cite[Theorem 5.0.5]{GP}, Guerberoff and Park compute the reduction of $V_{k, \sL}$ in terms of the valuation $v_p(\sL - a(k - 1))$, where
\[
  a(j) = H_{j/2} + H_{j/2 - 1},
\]
for $j \geq 1$.
We clearly have $a(k - 1) = H_{-} + H_{+}$. Therefore, the regions used to classify the reductions in Theorem \ref{thm from conj} match with those used in \cite[Theorem 5.0.5]{GP}. Now: 
\begin{itemize}
    \item If $\nu < 1 - (k - 2)/2$, then Theorem \ref{thm from conj} yields $\br{V}_{k, \sL}\vert_{I_{\qp}} \sim \omega_2^{k - 1} \oplus \omega_2^{p(k - 1)}$, which agrees with the reduction computed in \cite[Theorem 5.0.5 (2)]{GP}.
    \item Assume $-1/2 - l < \nu < 1/2 - l$ for some $l \in \{0, 1, \cdots, (k - 5)/2\}$. This region can be written as $(- l + (k - 3)/2) - (k - 2)/2 < \nu < (-l + (k - 1)/2) - (k - 2)/2$. Theorem \ref{thm from conj} yields $\br{V}_{k, \sL}\vert_{I_{\qp}} \sim \omega_2^{(k - 1) + (p - 1)(-l + (k - 3)/2)} \oplus \omega_2^{p((k - 1) + (p - 1)(-l + (k - 3)/2))}$, which agrees with the reduction computed in \cite[Theorem 5.0.5 (1)]{GP}.
    \item Assume $\nu = -1/2 - l$ for some $l \in \{0, 1, \cdots, (k - 5)/2\}$. This can be written as $\nu = (-l + (k - 3)/2) - (k - 2)/2$. Theorem \ref{thm from conj} yields $\br{V}_{k, \sL}\vert_{I_{\qp}} \sim \omega^{k - 1 - (-l + (k - 3)/2)} \oplus \omega^{(-l + (k - 3)/2)}$, which agrees with the reduction computed in \cite[Theorem 5.0.5]{GP}.
    \item Finally, if $\nu \geq 1/2$, then Theorem \ref{thm from conj} yields $\br{V}_{k, \sL}\vert_{I_{\qp}} \sim \omega^{(k - 1)/2}\oplus \omega^{(k - 1)/2}$, which agrees with the reduction computed in \cite[Theorem 5.0.5]{GP}.
\end{itemize}

For the small weights $k = 3$ and $5$, we may in fact improve on \cite[Theorem 5.0.5]{GP} by computing the reductions $\br{V}_{k, \sL}$ on the full
Galois group $\mathrm{Gal}(\brqp/\qp)$ (by giving formulas for $\lambda$). 
Indeed, since zig-zag has been proved for slopes $1/2$ and $3/2$ in \cite[Theorem A]{BG13} and \cite[Theorem 1.1]{GR}, respectively,
we obtain the following theorems.

   \begin{theorem}
     \label{theorem k=3}
        Let $p \geq 3$ and $k = 3$. We have the following dichotomy for the shape of the semi-simplification of the reduction of the semi-stable representation $V_{k, \sL}$:
        \[
            \br{V}_{k, \sL} \sim
            \begin{cases}
                    \mathrm{ind}(\omega_2^{2}), & \text{if $\nu < 1/2$} \\
                    \mu_{\lambda}\omega \oplus \mu_{\lambda^{-1}}\omega, & \text{if $\nu \geq 1/2$},
            \end{cases}
        \]
        where
        \begin{eqnarray}
            \lambda + \frac{1}{\lambda} & = & \br{-p^{-1/2}(\sL - 1)}.\nonumber
        \end{eqnarray}
    \end{theorem}

    \begin{theorem}
       \label{theorem k = 5}
        Let $p \geq 5$ and $k = 5$. We have the following tetrachotomy for the shape of the semi-simplification of the reduction  of the semi-stable representation $V_{k, \sL}$:
        \[
            \br{V}_{k, \sL} \sim
            \begin{cases}
                    \mathrm{ind}(\omega_2^{4}), & \text{if $\nu < -1/2$} \\[5pt]
                    \mu_{\lambda_1}\omega^3 \oplus \mu_{\lambda_1^{-1}}\omega, & \text{if $\nu = -1/2$} \\[5pt]
                    \mathrm{ind}(\omega_2^{3 + p}), & \text{if $-1/2 < \nu < 1/2$} \\[5pt]
                    \mu_{\lambda_2}\omega^{2} \oplus \mu_{\lambda_2^{-1}}\omega^2, & \text{if $\nu \geq 1/2$},
            \end{cases}
        \]
        where the constants $\lambda_i$ are given by
        \begin{eqnarray}
            \lambda_1 & = & \br{-3p^{1/2}\left(\sL - \frac{5}{2}\right)}, \nonumber \\
            \lambda_2 + \frac{1}{\lambda_2} & = & \br{2p^{-1/2}\left(\sL - \frac{5}{2}\right)}. \nonumber
        \end{eqnarray}
    \end{theorem}
\end{itemize}

\noindent {\bf Remark.} We remark that recently the second author noticed that it is possible to reverse the arguments used to
prove Theorem~\ref{thm from conj} to give a {\it proof} of the zig-zag conjecture, at least on inertia and for {\it all}
half-integral slopes $0 < v_p(a_p) \leq \frac{p-3}{2}$ \cite[v1]{Gha22} (these restrictions can be removed \cite[v2]{Gha22}
using very recent work of the first two authors \cite{CG23} which computes $\br{V}_{k,\sL}$ directly on $\mathrm{Gal}(\brqp/\qp)$
using the Iwahori mod $p$ Local Langlands Correspondence \cite{Chi23}).

\subsection{Relation with local constancy}
In \cite{Ber}, Berger proved the following theorem on the local constancy in the weight
of the semi-simplification of the reduction of crystalline representations.
\begin{theorem}[{\cite[Theorem B]{Ber}}]\label{Berger's Theorem B}
    Let $a_p \not = 0$ and $k > 3v_p(a_p) + \alpha(k - 1) + 1$, where $\alpha(j) = \sum_{n \geq 1}\left\lfloor\frac{j}{p^{n - 1}(p - 1)}\right\rfloor$. Then there exists $m = m(k, a_p)$ such that $\br{V}_{k', a_p} = \br{V}_{k, a_p}$, if $k' \geq k$ and $k' - k \in p^{m - 1}(p - 1)\mathbb{Z}$.
\end{theorem}
This local constancy result does not extend to small weights $k$ in the sense that the bound in the theorem
is sharp. This was noticed in \cite{Gha} using the following examples.
\begin{itemize}
    \item Let $k = 4$ and $a_p = p \geq 5$. Then $4 \not > 3(1) + 1$ and the crystalline representation $V_{k, a_p}$ does not satisfy Berger's bound on the weight. Now let $k' = 4 + p^n(p - 1)$, for large $n$. If the above local constancy result were to hold for $V_{k, a_p}$, then we would have $\br{V}_{k', a_p} \sim \br{V}_{k, a_p}$. However, $\br{V}_{k, a_p}$ is irreducible as $k$ belongs to the range $[2, p + 1]$ treated by Fontaine and Edixhoven \cite{Edi}, but $\br{V}_{k', a_p} \sim \mu_3\>\omega^2 \oplus \mu_{3^{-1}}\>\omega$ is reducible by \cite[Theorem 1.1]{BGR18}.
    \item Let $k = 5$ and $a_p = p^{3/2}$, for $p \geq 7$. Then $5 \not > 3(1.5) + 1$ and the crystalline representation $V_{k, a_p}$ again does not satisfy the bound on the weight. Now let $k' = 5 + p^n(p - 1)$, for large $n$. If the above local constancy result were to hold for $V_{k, a_p}$, then we would have $\br{V}_{k', a_p} \sim \br{V}_{k, a_p}$. However, $\br{V}_{k, a_p} \sim \mathrm{ind}(\omega_2^4)$ since $k$ belongs to the Fontaine-Edixhoven range, but $\br{V}_{k', a_p} \sim \mathrm{ind}(\omega_2^{3 + p})$ by \cite[Theorem 1.1]{GR}.
\end{itemize}

However, 
we have the following
corollary to Theorem \ref{any l invariant}.

\begin{Corollary}\label{Main}
    For $k \geq 3$, the sequence of crystalline representations $V^*_{k + p^n(p - 1), p^{r/2}}$ converges to the semi-stable representation $V^*_{k, \sL}$ for $\sL = 0$.
\end{Corollary}
\noindent 
Therefore, for large $n$ the reductions of the crystalline representations $V_{k+p^n(p-1),p^{r/2}}$ in the examples above 
should be isomorphic
to the reduction of the semi-stable representation $V_{k,\sL}$ with $\sL = 0$, and not necessarily to
the reduction of the crystalline representation $V_{k,p^{r/2}}$ (though all may be the same, as is the case when $k = 3$). Indeed, one checks that
\begin{itemize}
\item if $k = 4$ and $p \geq 5$, then 
     by \cite[Theorem 1.2]{BM},
     the reduction of the semi-stable representation $V_{4, \sL}$ for $\sL = 0$ is $\mu_3\>\omega^2 \oplus  \mu_{3^{-1}}\>\omega$, and,
   \item if $k = 5$ and $p \geq 7$, then 
     by \cite[Theorem 5.0.5]{GP} (or, better, by Theorem~\ref{theorem k = 5}), the reduction of the semi-stable representation $V_{5, \sL}$ for $\sL = 0$ is
     $\mathrm{ind}(\omega_2^{3 + p})$.
    \end{itemize}
    Thus, there is no apparent contradiction to local constancy in the weight for the reductions of the crystalline representations above when $k$ is small if one 
    works in the Colmez-Chenevier space which includes semi-stable representations of weight $k$.

\section{The blow-up of $\U$}
  \label{blow-up}

In this section, we provide details about blow-ups in the rigid analytic setting which may be of independent
interest. This is inspired by the work of Schoutens (cf. \cite{Sch}), but we carefully work out the details and also do not work over algebraically
closed fields. For background on rigid analytic geometry, see \cite{BGR84}.
We also recall the important Proposition $\ref{fiber and tangents}$, which for each
finite extension $E$ of $\qp$ establishes a bijection between the $E$-valued points of $\tU$ that lie above $(p^r, (1 + p)^{k - 1} - 1)$ and the tangent directions in $\U$ at
the exceptional point. 

\subsection{The blow-up as a  rigid analytic variety}\label{L-valued points here}

The blow-up of $\U$ consists of a rigid analytic variety $\tU$ and a map $\pi : \tU \to \U$ satisfying the properties given in \cite[Definition 1.2.1]{Sch}. In this subsection, we will construct $\tU$.

Recall that $\U = \text{Sp }\co(\U)$, where
\[
    \co(\U) = \qp \langle S_1, S_2, T_1, T_2, T_3 \rangle /(p^rT_1 - S_1, 1 - T_1T_2, pT_3 - S_2).
\]
The $\qp$-valued point $(p^r, (1 + p)^{k - 1} - 1)$ corresponds to the maximal ideal $m = (f_1, f_2)$ of $\co(\U)$, where $f_1 = S_1 - p^r$ and $f_2 = S_2 - ((1 + p)^{k - 1}-1)$. By the blow-up of $\U$ at $(p^r, (1 + p)^{k - 1} - 1)$, we mean the blow-up at the maximal ideal $m$.

We will construct $\tU$ by a patching argument. Let $Q_1, Q_2, Q_1'$ and $Q_2'$ be indeterminates. For $i = 0, 1, \cdots$, consider the affinoid algebra
\[
    A_i = \co(\U)\langle p^iQ_1 \rangle/(f_2 - Q_1f_1),
\]
which describes the subset of $\U$ where $\norm{f_2} \leq p^i\norm{f_1}$. Sending  $p^{i+1}Q_1$ to $p \cdot p^{i}Q_1$ and $X$ to $p^iQ_1$ induces an isomorphism $A_{i + 1}\langle X \rangle/(pX - p^{i + 1}Q_1) \cong A_i$, showing that $\text{Sp }A_i$ is the affinoid subdomain (hence an open subvariety) of $\text{Sp }A_{i + 1}$ 
consisting of maximal ideals $\gm$ such that $\norm{Q_1 \mod{\gm}} \leq p^i$.
We therefore get a sequence of inclusions $\text{Sp }A_0 \to \text{Sp }A_1 \to \text{Sp }A_2 \to \cdots$ associated to the sequence of affinoid algebra homomorphisms $\cdots \to A_2 \to A_1 \to A_0$. Using \cite[Proposition $9.3.2/1$]{BGR84} we paste together the $\text{Sp }A_i$ for $i \geq 0$ to get a rigid analytic variety $\til{V}_1$. 
The same proposition also states that $\{\text{Sp }A_i\}_{i \geq 0}$ is an admissible cover of $\til{V}_1$. 
Similarly, for $i \geq 0$, we define
\[
    B_i = \co(\U) \langle p^iQ_2 \rangle/(f_1 - Q_2f_2)
\]
and glue the corresponding affinoid spaces to get a rigid analytic variety $\til{V}_2$. The blow-up $\tU$ is the rigid analytic variety obtained by
glueing $\til{V}_1$ and $\til{V}_2$ along certain open subvarieties. 

We describe these open subvarieties now. For each $i \geq 0$, consider the affinoid algebra
\[
    A_i' = \co(\U)\langle p^iQ_1, p^iQ_1'\rangle/(f_2 - Q_1f_1, 1 - Q_1Q_1'),
\]
which describes a Laurent subdomain $\text{Sp }A_i'$ of $\text{Sp }A_i$ given by the condition $\norm{Q_1} \geq p^{-i}$.  Since we have an isomorphism $A_{i + 1}'\langle X, Y\rangle/(pX - p^{i + 1}Q_1, pY - p^{i + 1}Q_1') \cong A_i'$ (sending $X$ to $p^{i}Q_1$ and $Y$ to $p^{i}Q_1'$), we see that $\text{Sp }A_{i}'$ is an affinoid subdomain (and hence an open subvariety) of $\text{Sp }A_{i + 1}'$. Using \cite[Proposition $9.3.2/1$]{BGR84}, we paste together $\text{Sp }A_i'$ for $i \geq 0$ to get a rigid analytic variety $\til{V}_1'$. Moreover, using \cite[Proposition $9.3.3/1$]{BGR84}, we see that the canonical inclusions $\text{Sp }A_i' \to \text{Sp }A_i$ induce an inclusion $\til{V}_1' \xhookrightarrow{} \til{V}_1$, identifying $\til{V}_1'$ as a subvariety of $\til{V}_1$. Similarly, we construct a subvariety $\til{V}_2'$ of $\til{V}_2$ by glueing all the $\text{Sp }B_i'$ together, where
\[
    B_i' = \co(\U)\langle p^iQ_2, p^iQ_2' \rangle/(f_1 - Q_2f_2, 1 - Q_2Q_2').
  \]
It turns out that $\text{Sp }A_i'$ and $\text{Sp }B_i'$ are the intersections of $\text{Sp }A_i$ and $\text{Sp }B_i$ in the blow-up $\tU$ that we will soon construct. 
The spaces above are summarized by the following picture:
\vspace{.22cm}
\begin{center}
\includegraphics[scale=0.6]{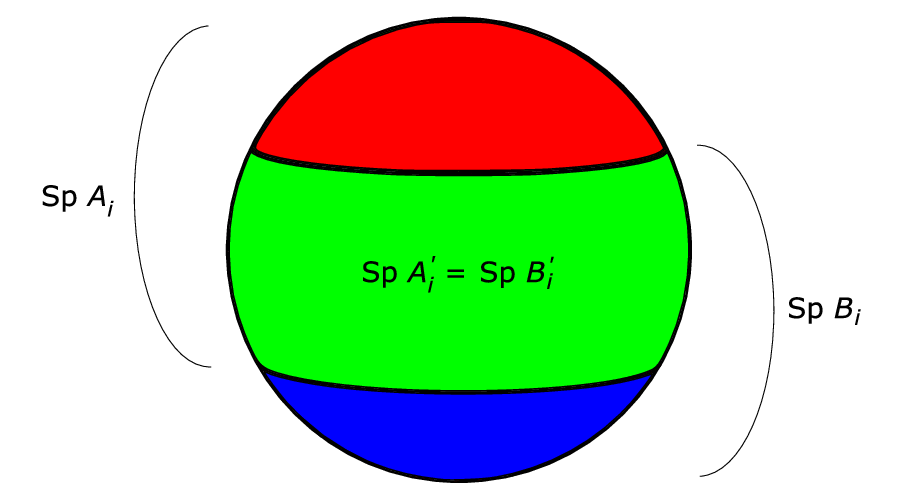}
\captionof{figure}{Exceptional fiber of the blow-up in the rigid setting.}
\end{center}

The space $\til{V}_1$ corresponds to the surface of the sphere except the south pole in the figure above whereas $\til{V}_2$ corresponds to the surface of the sphere except the north pole.


\par We claim that $\til{V}_j'$ is an open subvariety of $\til{V}_j$, for $j = 1, 2$. Without loss of generality, assume $j = 1$. To prove this claim, we need to show that $\til{V}_1' \cap \text{Sp }A_i$ is an admissible open subset of $\text{Sp }A_i$ for each $i \geq 0$. Write $\til{V}_1' \cap \text{Sp }A_i = \cup_{k = i}^{\infty}\>(\text{Sp }A_k' \cap \text{Sp }A_i)$. If $k \geq i$, then $\text{Sp }A_k' \cap \text{Sp }A_i$ is the set of maximal ideals $\gm$ of $A_i$ satisfying $\norm{Q_1 \text{ mod }\gm} \geq p^{-k}$. Therefore $\til{V}_1' \cap \text{Sp }A_i$ is the set of maximal ideals $\gm$ of $A_i$ such that $\norm{Q_1\text{ mod }\gm} > 0$. In other words, it is the complement of the vanishing set of $Q_1$ in $\text{Sp }A_i$, i.e., it is a Zariski open subset. Since Zariski open subsets are admissible open, we have proved the claim.
\vspace{1mm}

To glue $\til{V}_1$ and $\til{V}_2$, 
we define an isomorphism $\phi : \til{V}_1' \to \til{V}_2'$ as follows. Recall that for $i \geq 0$,
\vspace{-1mm}
\[
    \begin{array}{l}
    A_i' = \co(\U)\langle p^iQ_1, p^iQ_1' \rangle/(f_2 - Q_1f_1, 1 - Q_1Q_1'), \\
    B_i' = \co(\U)\langle p^iQ_2, p^iQ_2' \rangle/(f_1 - Q_2f_2, 1 - Q_2Q_2').
    \end{array}
\]
Consider an isomorphism of affinoid algebras $B_i' \to A_i'$ given by
\vspace{-1mm}
\[
    \begin{tikzcd}[row sep = -1mm]
        p^iQ_2 \ar[r, mapsto] & p^iQ_1', \\
        p^iQ_2' \ar[r, mapsto] & p^iQ_1.
    \end{tikzcd}
\]
This isomorphism gives rise to an isomorphism $\phi_i : \text{Sp }A_i' \to \text{Sp }B_i'$ of the corresponding affinoid spaces. For $i, j \geq 0$, these maps clearly restrict to the same map on $\text{Sp }A_i' \cap \text{Sp }A_j'$. Therefore applying \cite[Proposition $9.3.3/1$]{BGR84}, we get an isomorphism of rigid analytic varieties $\phi : \til{V}_1' \to \til{V}_2'$. 

Define $\tU$ to be the rigid analytic variety obtained by glueing $\til{V}_1$ and $\til{V}_2$ along $\til{V}_1'$ and $\til{V}_2'$, respectively, using the isomorphism $\phi$. This rigid analytic variety is analogous to the blow-up in the classical algebraic geometry setting.

We see this by computing its $E$-valued points for any finite extension $E$ of $\qp$.
Since $\tU$ is covered by $\til{V}_1$ and $\til{V}_2$, we compute $\til{V}_1(E)$ and $\til{V}_2(E)$ first. Recall for $i \geq 0$,
\[
    A_i = \co(\U)\langle p^iQ_1 \rangle/(f_2 - Q_1f_1).
\]
We identify the set of $E$-valued points of $\text{Sp }A_i$ with
\[
    \{(s_1, s_2, q_1) \in \U(E) \times E \text{ } \big\vert \text{ } (1 + s_2 - (1+p)^{k - 1}) = q_1(s_1 - p^r), \> \norm{q_1} \leq p^i\}.
\]
The first condition is a consequence of the relation $f_2 = Q_1f_1$. For $i \geq j$, the glueing map $A_i \to A_j$ induces the canonical inclusion map on $E$-valued points $\text{Sp }A_j(E) \to \text{Sp }A_i(E)$. Therefore, the set of $E$-valued points of $\til{V}_1$ is the union of all the $\text{Sp }A_i(E)$, i.e.,
\[
    \til{V}_1(E) = \{(s_1, s_2, q_1) \in \U(E) \times E \text{ } \big\vert \text{ } (1 + s_2 - (1+p)^{k - 1}) = q_1(s_1 - p^r)\}.
\]
We similarly have
\[
    \til{V}_2(E) = \{(s_1, s_2, q_2) \in \U(E) \times E \text{ } \big\vert \text{ } (s_1 - p^r) = q_2(1 + s_2 - (1 + p)^{k - 1})\}.
\]
We can now  describe $\tU(E)$.
If $P$ is an $E$-valued point of $\til{V}_1 \cap \til{V}_2$, it is an $E$-valued point of $\text{Sp }A_i \cap \text{Sp }B_i = \text{Sp }A_i'$, for some $i \geq 0$. Being a point of $\text{Sp }A_i$, it is of the form $(s_1, s_2, q_1)$, for some $s_1, s_2, q_1 \in E$. 
Since the glueing map between $B_i'$ and $A_i'$ takes $S_1$ to $S_1$, $S_2$ to $S_2$ and $Q_2$ to the inverse of $Q_1$ in $A_i'$, we see that $(s_1, s_2, q_1^{-1})$ represents $P$ as an $E$-valued point of $\text{Sp }B_i$ (the condition $\norm{q_1^{-1}} \leq p^i$ is satisfied because $P \in \text{Sp }A_i'$
implies $\norm{q_1} \geq p^{-i}$). We can therefore identify $(s_1, s_2, q_1) \in \til{V}_1(E)$ with the point $(s_1, s_2, 1 : q_1) \in \U(E) \times \mathbb{P}^1(E)$ and $(s_1, s_2, q_2) \in \til{V}_2(E)$ with the point $(s_1, s_2, q_2 : 1) \in \U(E) \times \mathbb{P}^1(E)$, so that a point in the intersection goes to the same point in $\U(E) \times \mathbb{P}^1(E)$ under both of these identifications. As a result of this discussion, we see that
\begin{eqnarray}
  \label{points-blow-up}
    \tU(E) = \{(s_1, s_2, \> \xi_1 : \xi_2) \in \U(E) \times \mathbb{P}^1(E) \text{ } \big\vert \text{ } (s_1 - p^r)\xi_2 = (1 + s_2 - (1 + p)^{k - 1})\xi_1\}
\end{eqnarray}
exactly as in the classical algebraic geometry setting. 

\subsection{The blow-up map $\pi : \tU \to \U$}

In this subsection, we define a candidate for the blow-up map $\pi : \tU \to \U$ using \cite[Proposition $9.3.3/1$]{BGR84}.

To define $\pi$, we first define its restrictions $\pi_i$ to $\til{V}_i$, for $i = 1, 2$.
We first define the map $\pi_1 : \til{V}_1 \to \U$. 
For each $i \geq 0$, consider the map $\text{Sp }A_i \to \U$ associated to the canonical homomorphism
\[
    \co(\U) \to \co(\U)\langle p^iQ_1 \rangle/(f_2 - Q_1f_1) = A_i.
\]
These maps clearly restrict to the same map on $\text{Sp }A_i \cap \text{Sp }A_j$, for $i, j \geq 0$. So using \cite[Proposition $9.3.3/1$]{BGR84}, we glue these maps together to get a map $\pi_1 : \til{V}_1 \to \U$ of rigid analytic varieties. The other map $\pi_2 : \til{V}_2 \to \U$ is constructed similarly.

Now, to get a map $\pi : \tU \to \U$, we have to check that $\pi_1$ and $\pi_2$ agree on the intersection of $\til{V}_1$ and $\til{V}_2$ in $\tU$. Since this intersection is equal to $\til{V}_1'$ $(\cong \til{V}_2')$, we have to check that the following diagram commutes
\[
    \begin{tikzcd}[row sep = 0.3cm, column sep = 1.2cm]
        \til{V}_1' \arrow{dr}[pos = 0.3]{\pi_1 |_{\til{V}_1'}} \arrow{dd}[swap]{\phi}& \\
        & \U. \\
        \til{V}_2' \arrow{ur}[pos = 0.3, swap]{\pi_2 |_{\til{V}_2'}}
    \end{tikzcd}
\]
Since $\phi$ is obtained by pasting the maps $\phi_i : \text{Sp }A_i' \to \text{Sp }B_i'$, the commutativity of the diagram above is equivalent to the commutativity of the following diagram, for all $i \geq 0$
\[
    \begin{tikzcd}[row sep = 0.5cm]
        \text{Sp }A_i' \ar{dr}[pos = 0.4]{\pi_1 |_{\text{Sp }A_i'}} \arrow{dd}[swap]{\phi_i} \\
        & \U. \\
        \text{Sp }B_i' \ar{ur}[pos = 0.4, swap]{\pi_2 |_{\text{Sp }B_i'}}
    \end{tikzcd}
\]
The commutativity of the diagram above can be checked by reversing all the arrows and noting that the resulting maps are $\co(\U)$-algebra homomorphisms. This shows that $\pi_1$ and $\pi_2$ agree on the intersection of $\til{V}_1$ and $\til{V}_2$ in $\tU$. Therefore, we glue $\pi_1$ and $\pi_2$ using \cite[Proposition $9.3.3/1$]{BGR84} to get a map $\pi : \tU \to \U$.

\subsection{Proof that $\pi : \tU \to \U$ is the blow-up}


In this subsection, we prove that the map $\pi : \tU \to \U$ defined in the previous subsection is the blow-up of $\U$ at
the maximal ideal $m = (f_1, f_2)$.
To prove this, we need to check that $\pi$ satisfies the two properties stated in \cite[Definition 1.2.1]{Sch}, namely the invertibility of a certain sheaf, and a corresponding universal property.

We first show that the ideal sheaf on $\U$ corresponding to the ideal $m$ of $\co(\U)$ extends to an invertible ideal sheaf on $\tU$ in the sense of \cite[Definition $1.1.1$]{Sch}. Since this is a local criterion (this means that the criterion can be checked over an admissible cover), we need to check that for all $i \geq 0$, the maps $\pi\big|_{\text{Sp }A_i} : \text{Sp }A_i \to \U$ and $\pi\big|_{\text{Sp }B_i} : \text{Sp }B_i \to \U$ satisfy the same property.
Let us prove this statement for $\pi\big|_{\text{Sp }A_i} : \text{Sp }A_i \to \U$. 
In this case, the extended ideal sheaf corresponds to the ideal $mA_i$ on $\text{Sp }A_i$. Moreover, \cite[Proposition 1.1.4]{Sch} states that the invertibility of this sheaf is equivalent to the invertibility of the ideal $mA_{i, \gm}$ of $A_{i, \gm}$ for all maximal ideals $\gm$ of $A_i$, where $A_{i, \gm}$ is the localization of $A_i$ at $\gm$. To show this, it is enough to prove that $mA_i$ is generated by a regular element of $A_i$ because regular elements go to regular elements under flat base change. 

Since $mA_i$ is generated by $f_1$, we prove that $f_1$ is not a zero divisor in $A_i$ for $i \geq 0$. Note that we only need to prove this statement for $i = 0$ because of the injections $A_i \xhookrightarrow{} A_0$.
\begin{lemma}
    $f_1$ is not a zero divisor in $A_0 = \co(\U)\langle Q_1 \rangle/(f_2 - Q_1f_1)$.
\end{lemma}
\begin{proof}
    Recall that 
    \[
        \co(\U) = \qp \langle S_1, S_2, T_1, T_2, T_3 \rangle/(p^rT_1 - S_1, 1 - T_1T_2, pT_3 - S_2).
    \]
    To prove that $f_1$ is not a zero divisor in $A_0$, it is enough to prove that $f_1$ is not a zero divisor in $\qp \langle S_1, S_2, Q_1 \rangle/(f_2 - Q_1f_1)$. Indeed, since $\text{Sp }A_0$ is an affinoid subdomain (a Laurent subdomain) of $\text{Sp }\qp\langle S_1, S_2, Q_1\rangle/(f_2 - Q_1f_1)$, the canonical map $\qp\langle S_1, S_2, Q_1\rangle/(f_2 - Q_1f_1) \to A_0$ is flat by \cite[Corollary $7.3.2/6$]{BGR84}. Therefore, if $f_1$ is not a zero divisor in the former then, by flatness, it is not a zero divisor in the latter.

So let us now prove that $f_1$ is not a zero divisor in $\qp\langle S_1, S_2, Q_1\rangle/(f_2 - Q_1f_1)$.
Assume that there exists $h \in \qp\langle S_1, S_2, Q_1\rangle$ such that $f_2 - Q_1f_1$ divides $f_1h$. We know that $f_2 - Q_1f_1 = 1 + S_2 - (1+p)^{k - 1} - Q_1(S_1 - p^r)$ is $S_2$-distinguished of degree $1$ (cf. \cite[Definition 5.2.1/1]{BGR84}). Applying the Weierstrass division theorem (cf. \cite[Theorem 5.2.1/2]{BGR84}) to $h$ and $f_2 - Q_1f_1$, we see that there exist unique power series $q \in \qp\langle S_1, S_2, Q_1\rangle$ and $r \in \qp\langle S_1, Q_1\rangle$ (note that $S_2$ does not appear in $r$ because $f_2 - Q_1f_1$ is of degree $1$ in $S_2$) such that $h = q(f_2 - Q_1f_1) + r$. The fact that $f_2 - Q_1f_1$ divides $f_1h$ implies that $f_2 - Q_1f_1$ divides $f_1r$. This is possible only if $r = 0$, i.e., only if $f_2 - Q_1f_1$ divides $h$. Therefore $h = 0$ in $\qp \langle S_1, S_2, Q_1 \rangle/(f_2 - Q_1f_1)$.
\end{proof}

Having checked the first property of blow-ups for $\pi$, we now
check the second property. Using \cite[Lemma 1.2.4]{Sch}, we note that it is enough to check it for affinoid spaces. In other words, given any affinoid space $Y = \text{Sp }R$ and a map of rigid analytic varieties $f : Y \to \U$ such that the ideal sheaf on $Y$ associated to the ideal $mR$ is invertible, we prove that there exists a unique map $g : Y \to \tU$ such that the following diagram commutes
\[
    \begin{tikzcd}
        & \tU \ar[d, "\pi"] \\
        Y \ar[r, "f"] \ar[ur, "g"] & \U.
    \end{tikzcd}
\]
\par For $i = 1, 2$, define $\mathfrak{a}_i = (f^*(f_i)R : mR)$ and let $Y_i = Y \setminus V(\mathfrak{a}_i)$. We prove two lemmas about the $Y_i$. 
\begin{lemma}\label{loc gen}
    For $i = 1$, $2$, a point $y \in Y$ belongs to $Y_i$ if and only if $m\mathcal{O}_{Y, y} = f^*(f_i)\mathcal{O}_{Y, y}$.
\end{lemma}
\begin{proof}
    We thank one of the referees for pointing out the useful \cite[Corollary 3.15]{AM} which shortens the proof.  
    Let $y$ be a point in $Y$ and $\gm_y$ be the corresponding maximal ideal of $R$. Fix $i =1,2$.
    Using \cite[Corollary 3.15]{AM}, we see that $\mathfrak{a}_iR_{\gm_y} = (f^*(f_i)R_{\gm_y} : mR_{\gm_y})$. Now,
    \[
        y \in Y_i \iff \mathfrak{a}_iR_{\gm_y} = R_{\gm_y} \iff mR_{\gm_y} = f^*(f_i)R_{\gm_y}.
    \]
    Using \cite[Proposition 7.3.2/3]{BGR84}, we have a map $R_{\gm_y} \to \mathcal{O}_{Y, y}$, which induces an isomorphism $\widehat{R}_{\gm_y} \simeq \widehat{\mathcal{O}}_{Y, y}$.
    We conclude that
    \[
        mR_{\gm_y} = f^*(f_i)R_{\gm_y} \iff m\mathcal{O}_{Y, y} = f^*(f_i)\mathcal{O}_{Y, y}.
    \]
    Indeed, the forward implication is obtained by extending scalars and the reverse implication follows
    by extending scalars to the completion $\widehat{\mathcal{O}}_{Y, y}$ and using the fact that
    $R_{\gm_y} \to \widehat{R}_{\gm_y}$ is faithfully flat.
  \end{proof}

  Using \cite[Lemma 1.1.2]{Sch} and the lemma above one easily checks that $Y = Y_1 \cup Y_2$. The sets $Y_1$ and $Y_2$, being Zariski open subsets of $Y$, are admissible open and form an admissible cover of $Y$.
  
  Before proving the next lemma, we remark that, by \cite[Lemma 0.4]{Sch}, we have $\mathfrak{a}_i\co(Y_i) = \co(Y_i)$,  which implies that $m\co(Y_i) = f^*(f_i)\co(Y_i)$,
  for $i = 1, 2$.
\begin{lemma}\label{zero div}
    Let $i = 1$, $2$. Then, $f^*(f_i)$ is not a zero divisor in $\co(Y')$ for any admissible open subset $Y'$ of $Y$ contained in $Y_i$.
\end{lemma}
  
\begin{proof}
The proof has three steps. In the first step, we prove that $f^*(f_i)$ is not a zero divisor in $\co_{Y,y}$ for any $y \in Y'$. In the second step, we prove that the canonical map $\co(Y') \to \prod_{y \in Y'}\co_{Y, y}$ is injective. In the third step, we use these two facts to conclude that $f^*(f_i)$ is not a zero divisor in $\co(Y')$. Let $i = 1$ or $2$.
\begin{itemize}
\item Using the remark preceding this lemma, we get $m\co(Y') = f^*(f_i)\co(Y')$. Extending to the stalks, we get $m\co_{Y, y} = f^*(f_i)\co_{Y, y}$ for each $y \in Y'$. Since the sheaf associated to the ideal $m\co(Y)$ is invertible, we see that $m\co_{Y, y}$ is generated by a regular element of $\co_{Y, y}$ for each $y \in Y$. Therefore $f^*(f_i)$ is not a zero divisor in $\co_{Y, y}$ for each $y \in Y'$.

\item Consider the map $\co(Y') \to \prod_{y \in Y'}\co_{Y, y}$. Let $a$ be an element of $\co(Y')$ that maps to $0$ in $\co_{Y, y}$ for each $y \in Y'$. Choose an admissible cover $\{U_j\}_{j \in J}$ of $Y'$ by affinoid subdomains of $Y$. For any $j \in J$, the image of $a$ in $\co_{Y, y}$ is $0$ for each $y \in U_j$. Using \cite[Corollary $7.3.2/4$]{BGR84}, we see that the restriction of $a$ to $\co(U_j)$ is $0$ for each $j \in J$. Since the cover $\{U_j\}_{j \in J}$ is admissible, we see that $a = 0$ in $\co(Y')$.

\item Suppose there exists a $b \in \co(Y')$ such that $f^*(f_i)b = 0$. This means that $f^*(f_i)b = 0$ in $\co_{Y, y}$ for each $y \in Y'$. Using the fact that $f^*(f_i)$ is not a zero divisor in $\co_{Y, y}$ for $y \in Y'$, we see that the image of $b$ in $\co_{Y, y}$ is $0$. The injectivity of the map $\co(Y') \to \prod_{j \in J}\co_{Y, y}$ implies that $b = 0$ in $\co(Y')$. Therefore $f^*(f_i)$ is not a zero divisor in $\co(Y')$.
\qedhere
\end{itemize}
\end{proof}

We now construct $g : Y \to \tU$.
In order to do this, we use the above lemmas to define the restrictions of $g$ to the elements of a refinement of the cover $Y = Y_1 \cup Y_2$ and then
apply the usual patching argument.

As $f^*(f_2) \in m\mathcal{O}(Y_1) = f^*(f_1)\mathcal{O}(Y_1)$, there exists a $q_1 \in \mathcal{O}(Y_1)$ such that $f^*(f_2) = q_1 f^*(f_1)$. Similarly, there exists a $q_2 \in \mathcal{O}(Y_2)$ such that $f^*(f_1) = q_2 f^*(f_2)$. Furthermore, $q_1$ and $q_2$ are unique because of Lemma $\ref{zero div}$.

Since $Y_1$ is an admissible open subset of $Y$, there exists an admissible cover $\{Y_{1, m}\}_{m \in I}$ of $Y_1$ by affinoid subdomains $Y_{1, m}$ of $Y$. Here and just below
$m$ is an element of the indexing set $I$ and should not confused with the maximal ideal $m$ of $\co(\U)$ used throughout this paper.
For each $m \in I$, there is a restriction map $\co(Y_1) \to \co(Y_{1, m})$ and so we can think of $q_1$ as an element of $\co(Y_{1, m})$. For each $m \in I$, fix a natural number $\un{m}$ such that $\norm{p^{\un{m}}q_1} \leq 1$. Consider the affinoid algebra map $f^* : \co(\U) \to \co(Y)$ and compose it with the restriction map $\co(Y) \to \co(Y_{1, m})$. Extend this composition to a map $\co(\U)\langle p^{\un{m}}Q_1 \rangle \to \co(Y_{1, m})$ by sending $p^{\un{m}}Q_1$ to $p^{\un{m}}q_1$. This is possible because of \cite[Corollary 1.4.3/2]{BGR84}. Since we have the relation $f^*(f_1) = q_1f^*(f_2)$ in $\co(Y_{1, m})$, we see that this extension factors through an affinoid algebra map $A_{\un{m}} \to \co(Y_{1, m})$. Let $g_{1, m} : Y_{1, m} \to \text{Sp }A_{\un{m}}$ be the corresponding map of affinoid spaces. Similarly, there is an admissible cover $\{Y_{2, n}\}_{n \in J}$ of $Y_2$ by affinoid subdomains $Y_{2, n}$ of $Y$ and for each $n \in J$, a fixed natural number $\un{n}$ such that $\norm{p^{\un{n}}q_2} \leq 1$ and a map of affinoid spaces $g_{2, n} : Y_{2, n} \to \text{Sp }B_{\un{n}}$ associated to the  map of affinoid algebras $B_{\un{n}} \to \co(Y_{2, n})$ sending $p^{\un{n}}Q_2$ to $p^{\un{n}}q_2$. 
Since $g_{1, m}^*$ and $g_{2, n}^*$ are $\co(\U)$-algebra homomorphisms, we see that $\pi \circ g_{1, m} = f$ on $Y_{1, m}$ and $\pi \circ g_{2, n} = f$ on $Y_{2, n}$, for $m \in I$, $n \in J$.

  Note that $\{Y_{1, m}, Y_{2, n}\}_{m \in I, n \in J}$ is an admissible cover of $Y$. To obtain $g : Y \to \tU$, we glue all the maps $g_{1, m}$ and $g_{2, n}$.
  Without loss of generality, assume $\un{m} \geq \un{m}'$. Then, $g_{1, m}$ and $g_{1, m'}$ agree on $Y_{1, m} \cap Y_{1, m'}$ because the following diagram commutes
\[
    \begin{tikzcd}[row sep = 0.2cm, column sep = 2.5 cm]
        A_{\un{m}'} \ar{dr}[pos = 0.4]{g_{1, m'}^*} & \\
        & \mathcal{O}(Y_{1, m} \cap Y_{1, m'}), \\
        A_{\un{m}} \ar{uu} \ar{ur}[pos = 0.4, swap]{g_{1, m}^*} & 
    \end{tikzcd}
\]
%
A similar argument shows that $g_{2, n}$ and $g_{2, n'}$ agree on $Y_{2, n} \cap Y_{2, n'}$. 

We check that the maps $g_{1, m}$ and $g_{2, n}$ agree on $Y_{1, m} \cap Y_{2, n}$. Let $N \geq \un{m}$, $\un{n}$. Think of $g_{1, m}$ as a map from $Y_{1, m}$ to $\text{Sp }A_{N}$ (by composing with the map $\text{Sp }A_{\un{m}} \to \text{Sp }A_N$) and $g_{2, n}$ as a map from $Y_{2, n}$ to $\text{Sp }B_{N}$ (by composing with the map $\text{Sp }B_{\un{n}} \to \text{Sp }B_N$). We claim that $g_{1, m}$ maps $Y_{1, m} \cap Y_{2, n}$ into $\text{Sp }A_{N}'$ and that $g_{2, n}$ maps $Y_{1, m} \cap Y_{2, n}$ into $\text{Sp }B_{N}'$. To show this, we first prove that $q_1$ is the inverse of $q_2$ in $\co(Y_{1, m}\cap Y_{2, n})$. Transfer the relations $f^*(f_2) = q_1f^*(f_1)$ and $f^*(f_1) = q_2f^*(f_2)$ from $\co(Y_1)$ and $\co(Y_2)$, respectively, to $\co(Y_{1, m}\cap Y_{2, n})$ by the restriction maps. Substituting the first relation into the second relation, we get $f^*(f_1) = q_2q_1f^*(f_1)$. Using Lemma \ref{zero div}, we cancel $f^*(f_1)$ to see that $q_1$ is indeed the inverse of $q_2$ in $\co(Y_{1, m} \cap Y_{2, n})$.

Compose $g_{1, m}^* : A_{N} \to \co(Y_{1, m})$ with the restriction map $\co(Y_{1, m}) \to \co(Y_{1, m} \cap Y_{2, n})$. Extend this composition to a map $A_{N}\langle p^{N}Q_1'\rangle \to \mathcal{O}(Y_{1, m} \cap Y_{2, n})$ by sending $p^{N}Q_1'$ to $p^{N}q_2$. This is possible because $\norm{p^{N}q_2} \leq 1$ since the intersection of two affinoid subdomains of an affinoid space is again an affinoid space and because under the homomorphism $\co(Y_{2, n}) \to \co(Y_{1, m} \cap Y_{2, n})$ of affinoid algebras, the image of an element has norm at most the norm of the element itself. Using the fact that $q_2$ is the inverse of $q_1$ in $\mathcal{O}(Y_{1, m} \cap Y_{2, n})$, we may further extend the above map to $A_N'$ to obtain the following commutative diagram
\[
    \begin{tikzcd}
        A_{N} = \co(\U)\langle p^{N}Q_1 \rangle/(f_2 - Q_1f_1) \ar[d] \ar[r, "g_{1, m}^*"] & \mathcal{O}(Y_{1, m}) \ar[d] \\
        A_{N}' = \co(\U)\langle p^{N}Q_1, p^{N}Q_1' \rangle/(f_2 - Q_1f_1, 1 - Q_1Q_1') \ar[r] & \mathcal{O}(Y_{1, m} \cap Y_{2, n}).
    \end{tikzcd}
\]
This shows that the affinoid algebra map $A_{N} \to \mathcal{O}(Y_{1, m} \cap Y_{2, n})$ factors through the map $A_{N} \to A_{N}'$. Reversing the arrows, we see that the restriction of $g_{1, m}$ to $Y_{1, m} \cap Y_{2, n}$ factors through the inclusion $\text{Sp }A_{N}' \to \text{Sp }A_{N}$. In other words, $g_{1, m}$ maps $Y_{1, m} \cap Y_{2, n}$ into $\text{Sp }A_{N}'$. Similarly, we can prove that $g_{2, n}$ maps $Y_{1, m} \cap Y_{2, n}$ into $\text{Sp }B_{N}'$.

Now we can prove that $g_{1, m}$ and $g_{2, n}$ agree on $Y_{1, m} \cap Y_{2, n}$. Indeed, this statement is equivalent to the commutativity of the following diagram
\[
    \begin{tikzcd}[row sep = 0.2cm, column sep = 2.5cm]
        A_{N}' \ar{dr}[pos = 0.4]{g_{1, m}^*} & \\
        & \mathcal{O}(Y_{1, m} \cap Y_{2, n}), \\
        B_{N}' \ar{uu}{\phi_{N}^*} \ar{ur}[pos = 0.4, swap]{g_{2, n}^*} & 
    \end{tikzcd}
\]
where the vertical map is the glueing map. This diagram commutes because under the top two maps $Q_2 \mapsto Q_1' \mapsto q_2$ and 
$Q_2' \mapsto Q_1 \mapsto q_1$, and these are exactly the images under the lower map. 

By \cite[Proposition 9.3.3/1]{BGR84}, there exists a map of rigid analytic spaces $g : Y \to \tU$ such that the following diagram
commutes
\[
    \begin{tikzcd}
        & \tU \ar[d, "\pi"]\\
        Y \ar[r, "f"] \ar[ur, "g"] & \U.
    \end{tikzcd}
\]

\par To prove that $\pi : \tU \to \U$ is the blow-up of $\U$ at the ideal $m = (f_1, f_2)$, it remains to check that $g$ is unique making the above diagram
commute. Suppose $g' : Y \to \tU$ is another candidate.  We prove that $g = g'$. We need:
\begin{lemma}\label{pre-image}
For $i = 1, 2$, we have $g'^{-1}(\til{V}_i) = Y_i$.
\end{lemma}
\begin{proof}
    Without loss of generality, assume that $i = 1$.
    \begin{itemize}
    \item We prove that $g'^{-1}(\til{V}_1) \subseteq Y_1$.  Recall that $\til{V}_1 = \cup_{i \geq 0}\>\text{Sp }A_i$. Therefore it is enough to prove that
      $Y_{1, i} :=  g'^{-1}(\text{Sp }A_i) \subseteq Y_1$ for each $i \geq 0$. (The notation $Y_{1,i}$ here and $Y_{2,i}$ below
      are local to this proof.) Fix $i \geq 0$. Consider the following commutative diagram
    \[
        \begin{tikzcd}
            & \text{Sp }A_i \ar[d, "\pi"] \\
            Y_{1, i} = g'^{-1}(\text{Sp }A_i) \ar[ur, "g'"] \ar[r, "f"] & \U.
        \end{tikzcd}
    \]
    We know that $mA_i = f_1A_i$. Extending this ideal from $A_i$ to $\co(Y_{1, i})$ using $g'^*$, we see that $m\co(Y_{1, i}) = g'^*(f_1)\co(Y_{1, i}) = f^*(f_1) \co(Y_{1, i})$. Let $y \in Y_{1, i}$.  Noting that $Y_{1, i}$ is an admissible open subset of $Y$, we have $m\co_{Y, y} = f^*(f_1)\co_{Y, y}$. By Lemma \ref{loc gen}, we see that $y \in Y_1$. 
    
  \item We prove that $Y_1 \subseteq g'^{-1}(\til{V}_1)$. This is equivalent to proving $g'^{-1}(\tU \setminus \til{V}_1) \subseteq Y \setminus Y_1$.
    Recall that $\tU$ is covered by $\til{V}_1$ and $\til{V}_2$. 
    Since $\til{V}_2 = \cup_{i \geq 0}\>\text{Sp }B_i$, we have
    \[
    \tU \setminus \til{V}_1 = \cup_{i \geq 0}\> (\text{Sp }B_i \setminus \til{V}_1).
    \]
    Thus, it is enough to prove that $Y_{2,i} := g'^{-1}(\text{Sp }B_i \setminus \til{V}_1) \subseteq Y \setminus Y_1$ for each $i \geq 0$. Fix $i \geq 0$.
    We prove that if $y\in Y_{2,i}$, 
    then $y \not \in Y_1$. We do so by proving that the ideal $m\co_{Y, y}$ of $\co_{Y, y}$ is not generated by $f^*(f_1)$ (cf. Lemma \ref{loc gen}). 
    So, let $y \in Y_{2, i}$ map to $\text{Sp }B_i \setminus \til{V}_1$ under $g'$. Assume towards a contradiction that $y \in Y_1$. By Lemma $\ref{loc gen}$, we see that $m\co_{Y, y}$ is generated by $f^*(f_1)$. By the analog of the first bullet point, $Y_{2, i} \subseteq Y_2$, so we also have $m\co_{Y, y} = f^*(f_2)\co_{Y, y}$. Therefore there exists a unit $u \in \co_{Y, y}$ such that $f^*(f_1) = uf^*(f_2)$.  Consider the following commutative diagram
    \[
        \begin{tikzcd}
            & \text{Sp }B_i \ar[d, "\pi"] \\
            Y_{2,i} = g'^{-1}(\text{Sp }B_i \setminus \til{V}_1) \ar[ur, "g'"] \ar[r, "f"] & \U.
        \end{tikzcd}
    \]
    Note that $Y_{2, i}$ is an admissible open subset of $Y$.
    We obtain the relation $f^*(f_1) = g'^*(Q_2)f^*(f_2)$ in $\co_{Y, y}$ by pushing the relation $f_1 = Q_2f_2$ from $B_i$ to $\co(Y_{2, i})$ using $g'^*$ and then from $\co(Y_{2, i})$ to $\co_{Y, y}$. Since $f^*(f_2)$ is not a zero divisor in $\co_{Y, y}$, we see that $g'^*(Q_2) = u$, which is a unit of $\co_{Y, y}$. We prove that this is a contradiction. We know that $\text{Sp }B_i \, \cap \,\til{V}_1$ is the set of maximal ideals $\gm$ of $B_i$ satisfying the extra condition $\norm{Q_2 \mod{\gm}} > 0$. Therefore $\text{Sp }B_i \setminus \til{V}_1$ is the set of maximal ideals $\gm$ of $B_i$ satisfying $\norm{Q_2 \mod{\gm}} = 0$. Since $g'(y) \in \text{Sp }B_i \setminus \til{V}_1$, we see that $Q_2$ vanishes at $g'(y)$. In other words, if we denote the maximal ideal of $\co_{\tU, \> g'(y)}$ by $\gm_{g'(y)}$, then we have $Q_2 \> \text{mod }\gm_{g'(y)} = 0$. If $\gn_{y}$ is the maximal ideal of $\co_{Y, y}$, then using the following commutative diagram
    \[
        \begin{tikzcd}
            \co_{\tU, \> g'(y)} \ar[d] \ar[r, "g'^*"] & \co_{Y, y} \ar[d] \\
            \co_{\tU, \> g'(y)}/\gm_{g'(y)} \ar[r] & \co_{Y, y}/\gn_{y},
        \end{tikzcd}
    \]
    we get $g'^*(Q_2) \text{ mod }\gn_y = 0$. This means that $g'^*(Q_2)$ is contained in the maximal ideal of $\co_{Y, y}$. In other words, $g'^*(Q_2)$ is not a unit. This is a contradiction. Therefore $y \not \in Y_1$. 
    \qedhere
    \end{itemize}
    \end{proof}

    To show that $g = g'$, we prove that their restrictions to $Y_{1, m}$ and $Y_{2, n}$ are equal, for all $m \in I$ and $n \in J$. We prove that $g$ and $g'$ agree on $Y_{1, m}$. Using the definition of $g$, we see that the restriction of $g$ to $Y_{1, m}$ factors as $Y_{1, m} \to \text{Sp }A_{\un{m}} \to \tU$. Since $g'$ maps the affinoid space $Y_{1, m}$ to $\tU$ and $\{\text{Sp }A_i, \text{Sp }B_j\}_{i, j \geq 0}$ is an admissible cover of $\tU$, we deduce that $g'$ maps $Y_{1, m}$ into the union of finitely many $\text{Sp }A_i$ and $\text{Sp }B_j$. By Lemma \ref{pre-image}, we see that $g' : Y_{1, m} \to \tU$ factors as $Y_{1, m} \to \text{Sp }A_i \to \tU$ for some $i \gg 0$. We may assume that $i \geq \un{m}$. To check that $g$ and $g'$ agree on $Y_{1, m}$, it is therefore enough to check that the following diagram commutes
\[
    \begin{tikzcd}[row sep = 0.3cm]
        & \text{Sp }A_i \\
        Y_{1, m} \ar[ur, "g'"] \arrow{dr}[swap]{g} \\
        & \text{Sp }A_{\un{m}}. \ar[uu]
    \end{tikzcd}
\]
In other words, we have to check that the following diagram commutes
\[
    \begin{tikzcd}[row sep = 0.3cm]
        & A_i \ar[dd] \ar{dl}[pos = 0.3, swap]{g'^*} \\
        \co(Y_{1, m}) \\
        & A_{\un{m}} \ar{ul}[pos = 0.3]{g^*}.
    \end{tikzcd}
\]
Since both $g^*$ and $g'^*$ are equal to $f^*$ on $\co(\U)$, we only need to check that $g^*(Q_1) = g'^*(Q_1)$. Pushing the relation $f_2 = Q_1f_1$ from $A_i$ to $\co(Y_{1, m})$ under $g'^*$, we get $f^*(f_2) = g'^*(Q_1)f^*(f_1)$. Similarly, pushing the same relation from $A_{\un{m}}$ to $\co(Y_{1, m})$ under  $g^*$, we get $f^*(f_2) = g^*(Q_1)f^*(f_1)$. Using Lemma $\ref{zero div}$, we get $g^*(Q_1) = g'^*(Q_1)$. This proves the commutativity of the diagram above. In other words, $g$ and $g'$ agree on $Y_{1, m}$. A similar argument shows that $g$ and $g'$ agree on $Y_{2, n}$. Therefore $g = g'$ on $Y$. This finally proves that $\pi : \tU \to \U$ is the blow-up of $\U$ at the maximal ideal $(f_1, f_2)$.

\subsection{Points in the exceptional fiber as tangent directions}
The following standard fact allows us to realize $E$-valued points of the fiber over the exceptional point $(p^r, (1 + p)^{r + 1} - 1)$ as tangent directions
in $\U$ at this point. Recall that the maximal ideal $m$ of $\co(\U)$ corresponding to this exceptional point is the ideal generated by
$f_1 = S_1 - p^r$ and $f_2 = S_2 - ((1 + p)^{k - 1}-1)$.

\begin{Proposition}\label{fiber and tangents}
  There is a bijection
  $$\pi^{-1}(p^r, (1 + p)^{r + 1} - 1)(E) \rightarrow  \mathbb{P}(\mathrm{Hom}(m/m^2\otimes_{\qp}E, E))$$
  between the $E$-valued points of the fiber over the point $(p^r, (1 + p)^{k - 1} - 1)$ and the elements of the projectivization of
  the tangent space $\mathrm{Hom}(m/m^2 \otimes_{\qp}E, E)$ over $E$, which sends a point $(p^r, (1 + p)^{k - 1} - 1, a : b) \in \tU(E)$ to the class in $\mathbb{P}(\mathrm{Hom}(m/m^2\otimes_{\qp}E, E))$ represented by $v$, where
    \[
        v(\br{f_1} \otimes 1) = a \text{ and } v(\br{f_2} \otimes 1) = b, 
      \]
    with bar denoting image modulo $m^2$.   
\end{Proposition}
  
\begin{proof}
  Let $P = (p^r, (1 + p)^{k - 1} - 1, a : b)$ be an $E$-valued point in the fiber above the exceptional point. Assume $b \not = 0$. Write $P = (p^r, (1 + p)^{k - 1} - 1, a/b : 1)$.
  From the construction of $\tU$, we see that $P \in \til{V}_2(E)$.
  Therefore $P$ is an $E$-valued point of $\text{Sp }B_i$, for some $i \geq 0$. 
  Let $g_P$ be the homomorphism $B_i \to E$ associated with $P$. The discussion at the end of Section~\ref{L-valued points here} implies that $g_P$ is defined by $g_P(S_1) = p^r$, $g_P(S_2) = (1 + p)^{k - 1} - 1$ and $g_P(Q_2) = a/b$. Let $m_P$ be the kernel of $g_P$. The map $\co(\U) \to B_i$  induces
  an $E$-linear surjection
    \[
        (m/m^2)\otimes_{\qp}E \to 
        f_2B_i/f_2B_im_P \otimes_{B_i, g_P} E,
    \]
    given by
    \begin{eqnarray}
        \br{f_1}\otimes 1 & \mapsto  & \br{f_2}\otimes (a/b) \nonumber\\
        \br{f_2}\otimes 1 & \mapsto & \br{f_2}\otimes 1, \nonumber
    \end{eqnarray}
    where the bars
    over the elements on the left denote their images modulo $m^2$ and the bars over the elements on the right  denote their images modulo $f_2B_i m_P$,
    and where we have used $\br{f_1} \otimes 1 = \br{f_2}\br{Q_2} \otimes 1 = \br{f_2} \otimes (a/b)$ in the codomain.    
    The codomain of this map can be identified with $E$ up to multiplication by a non-zero scalar.
    The class of the above map in $\mathbb{P}(\mathrm{Hom}(m/m^2\otimes_{\qp} E, E))$ is the 
    same as that of the map $v$, where $v(\br{f_1}\otimes 1) = a$ and $v(\br{f_2}\otimes 1) = b$.
    A similar construction works if $a \neq 0$. One checks immediately that the resulting assignment $P \mapsto v$ is well-defined and a bijection. 
\end{proof}

\section{Explicit bases of $H^1(\calR_{\qp}(x^r\chi))$ and $\sL$-invariants}
\label{link robba ring here}

Let $\Gamma = \mathrm{Gal}(\qp(\mu_{p^\infty})/\qp)$ and think of the cyclotomic character $\chi$ as a character $\chi  : \Gamma \to \zp^*$.
In this section,  we 
study two explicit bases of the first Fontaine-Herr cohomology group 
$H^1(\calR_{\qp}(x^r\chi))$ and the corresponding formulas for the $\sL$-invariant.
Here $\calR_{\qp}(x^r\chi)$ is the Robba ring over $\qp$ in the variable $T$, 
with the standard actions of  $\varphi$ and $\Gamma$
twisted by $x^r\chi$.

Colmez has constructed two power series $G(\norm{x}, r + 1)$ and $G'(\norm{x}, r + 1)$ in $\calR_{\qp}$ which he uses to define a basis of $H^1(\calR_{\qp}(x^r\chi))$ \cite[Proposition 2.19]{Col}. In the first subsection, we explicitly compute $G(\norm{x}, r + 1)$. In the second subsection, we give a partial description of $G'(\norm{x}, r + 1)$. In the third subsection, we state Colmez's formula for the $\sL$-invariant of a non-zero element of $H^1(\calR_{\qp}(x^r\chi))$ expressed as a linear combination in this basis. In the same subsection, we describe another basis of $H^1(\calR_{\qp}(x^r\chi))$ studied by Benois \cite[Proposition 1.5.4]{Ben}. Finally, we find the change of basis matrix between these two bases and restate the formula for the $\sL$-invariant in terms of Benois' basis (Definition~\ref{def of L inv}).

We first recall that the Fontaine-Herr cohomology groups $H^i(D)$ of a $(\varphi, \Gamma)$-module $D$ over $\calR_{\qp}$  for $i = 0,1,2$
are defined to be the cohomology groups of the complex
    \begin{eqnarray}\label{small (phi, Gamma) complex}
        0 \to D \to D \oplus D \to D \to 0,
    \end{eqnarray}
    where the second map is $x \mapsto ((\varphi - 1)x, (\gamma - 1)x)$ and the third map is $(y, z) \mapsto (\gamma - 1)y - (\varphi - 1)z$,
    for $\gamma$ a fixed topological generator of $\Gamma$. 
    We note that this complex is as in \cite[Section 2.1]{Che} and is a bit different from the one in \cite[Section 2.1]{Col}.

We now recall some standard facts about Robba rings. For $M > 0$, let $\mathcal{E}_E^{]0, M]}$ be the set of bidirectional power series in the variable $T$ with coefficients in $E$ that converge on the elements of $\brqp$ with valuation belonging to $(0, M]$. More precisely, it is the set of the series $\sum_{n = -\infty}^{\infty}a_n T^n$ satisfying 
\[
    \lim_{n \to \pm \infty} v_p(a_n) + ns \to \infty, \text{ for all } 0 < s \leq M.
\]
For each $0 < s \leq M$, there is a valuation $v_p(\>\>, s)$ on $\mathcal{E}_E^{]0, M]}$ defined by $v_p(\sum_{n = -\infty}^{\infty}a_n T^n, s) = \inf\{v_p(a_n) + ns \text{ }\vert \text{ }n \in \mathbb{Z}\}$. Any sequence in $\mathcal{E}_E^{]0, M]}$ that is Cauchy with respect to $v_p(\>\>, s)$, for each $0 < s \leq M$ is convergent in $\mathcal{E}_E^{]0, M]}$ (see \cite[Definition 2.5.1]{Ked}). Moreover, choosing a sequence $0 < r_l \leq M$ converging to $0$ with $r_1 = M$, we get a countable family of valuations $v_p(\>\>, r_l)$ on $\mathcal{E}_E^{]0, M]}$. Note that if $r_l \leq s \leq M$, for some $l \geq 1$, then $v_p(f(T), s) \geq \inf\{v_p(f(T), r_l), v_p(f(T), M)\}$. Therefore to check if a sequence in $\mathcal{E}_E^{]0, M]}$ is convergent, we only need to check that it converges to a common limit with respect to the valuations $v_p(\>\>, r_l)$, for $l \geq 1$. In particular, $\mathcal{E}_E^{]0, M]}$ is a Fr\'echet space with respect to the valuations $v_p(\>\>, r_l)$. For the space $\mathcal{E}_E^{]0, \frac{1}{p - 1}]}$, we set $r_l = \dfrac{1}{\phi(p^l)}$, where $\phi$ is Euler's totient function.

The Robba ring over $E$ is defined to be $\calR_{E} = \cup_{M > 0}\mathcal{E}_{E}^{]0, M]}$. 
It is endowed with commuting actions of an operator $\varphi$ and the group $\Gamma$ via $\varphi T = (1 + T)^p - 1$ and $\gamma T = (1 + T)^{\chi(\gamma)} - 1$, for $\gamma \in \Gamma$, respectively.

We recall some facts from \cite{Laz}. Let $\phi_n(T)$ for $n \geq 1$ be the $p^{n}$-th cyclotomic polynomial defined by
\begin{equation}
   \label{eq:2}
    \phi_n(T) = \frac{(1 + T)^{p^n} - 1}{(1 + T)^{p^{n - 1}} - 1} = 1 + (1+T)^{p^{n-1}} + \cdots + (1+T)^{(p-1)p^{n-1}}.
\end{equation}
The polynomials $\phi_n(T)$ are $\frac{1}{\phi(p^n)}$-extremal in the sense of \cite[Definition 2.7]{Laz}.  This means that $v_p(\phi_n(T), 1/\phi(p^n))$ is attained at the constant and leading terms: 
\begin{equation}
  \label{extremal}
  v_p(a_0) = v_p\left(\phi_n(T), \frac{1}{\phi(p^n)}\right) = v_p(a_{p^{n - 1}(p - 1)}) + \frac{p^{n - 1}(p - 1)}{\phi(p^n)} = 1,
\end{equation}
where $\phi_n(T) = a_0 + a_1T + \cdots + a_{p^{n - 1}(p - 1)}T^{p^{n - 1}(p - 1)}$, with $a_0 = p$ and $a_{p^{n - 1}(p - 1)} = 1$. For $n \geq 1$, we also have
\begin{equation}\label{eq:3}
    \varphi(\phi_n(T)) = \phi_{n + 1}(T).
\end{equation}
Let $t = \log{(1 + T)} \in \mathcal{E}^{]0, \frac{1}{p - 1}]}$. The following formula relates $t$ and the cyclotomic polynomials:
\begin{equation}\label{eq:4}
    t = T\prod_{n \geq 1}\frac{\phi_n(T)}{p}.
\end{equation}
We also have
\begin{equation}\label{eq:5}
    \varphi(t) = pt,
\end{equation}
and
\begin{equation}\label{eq:6}
    \gamma(t) = \chi(\gamma)t,
\end{equation}
for all $\gamma \in \Gamma$.

Let $r \geq 0$. By \cite[Proposition 2.16 (i)]{Col}, there is an isomorphism
\begin{equation}
  \label{G}
     \mathcal{E}^{]0, \frac{1}{p - 1}]}/t^{r + 1} \to \prod_{n \geq 1}   \mathcal{E}^{]0, \frac{1}{p - 1}]}/\phi_n^{r+1}.
\end{equation}
 By \cite[Proposition 2.16 (ii)]{Col}, for $n \geq 1$, there is a $\Gamma$-equivariant isomorphism
\[
    \iota_n : \mathcal{E}^{]0, \frac{1}{p - 1}]}/\phi_n^{r + 1} \to \qp(\zetapn)[t]/t^{r + 1}
\]
obtained by sending $T$ to $\zetapn e^{t/p^n} - 1$. Thus the $\Gamma$-equivariant homomorphism 
\begin{equation}
  \label{G'}
    \mathcal{E}^{]0, \frac{1}{p - 1}]}/t^{r + 1} \to \prod_{n \geq 1} \qp(\zetapn)[t]/t^{r + 1}
\end{equation}
induced by the maps $\iota_n$ for $n \geq 1$ is an isomorphism.

For the rest of the paper, we fix a topological generator $\gamma$ of $\Gamma$ such that $\chi(\gamma) = \zeta_{p-1}^a  (1+ p)$ for some fixed integer $a$.

\subsection{An explicit description of $G(\norm{x}, r + 1)$}
  \label{section G}

Let $r \geq 1$.  The first basis vector of $H^1(\calR_{\qp}(x^r\chi))$ constructed by Colmez is represented by the element
\[
    c_1 = (t^{-(r + 1)}(p^{-1}\varphi - 1)G(\norm{x}, r + 1), t^{-(r + 1)}(\gamma - 1)G(\norm{x}, r + 1))
\]
in $\calR_{\qp}(x^r\chi) \oplus \calR_{\qp}(x^r\chi)$ (for the untwisted action of $\varphi$ and
$\gamma$),
see \cite[Proposition 2.19]{Col}. Here $G(\norm{x}, r + 1)$ is any
power series $f(T)$ in $\mathcal{E}^{]0, \frac{1}{p-1}]}$ mapping to $\prod_{n \geq 1} \frac{1}{p^n}$ under the map \eqref{G} (cf. \cite[Section 2.6]{Col}).
In other words,  $G(\norm{x}, r + 1)$ satisfies
the system of congruences
\[\arraycolsep=2pt
    \begin{array}{rcl}
    f(T) & \equiv & \dfrac{1}{p} \text{ } \mod{\phi_1(T)^{r + 1}} \\ [8pt]
    f(T) & \equiv & \dfrac{1}{p^2} \mod{\phi_2(T)^{r + 1}} \\
    & \vdots & \\
    f(T) & \equiv & \dfrac{1}{p^n} \mod{\phi_n(T)^{r + 1}} \\
    & \vdots & \\
    \end{array}
    \nonumber
\]
The goal of this section is to write down $G(\norm{x}, r + 1)$ explicitly.

\begin{definition}
    For each $n \geq 1$ and $r \geq 1$, define
	\begin{equation}
	G_{n, \, r + 1}(T) := \left[1 - \left(\frac{\phi_n(T)}{p}\right)^{r + 1}\right]^{r + 1} \frac{1}{p^n} \prod_{i > n}\left[1 - \left(1 - \frac{\phi_i(T)}{p}\right)^{r + 1}\right]^{r + 1}.
	\nonumber
	\end{equation}
\end{definition}
To check $G_{n, \, r + 1}(T)$ is well defined, we need to show that the infinite product converges. We prove that for any $n, r \geq 1$, the product $\prod_{i > n}(1 - (1 - (\phi_i(T)/p))^{r + 1})$ converges. 

\begin{lemma}\label{valuation of binomial coefficients}
    For any $c = 1, 2, \cdots, p-1$ and any $j = 1, 2, \cdots, cp^{i-1}$, we have
    \begin{equation}
        v_p\left({cp^{i-1} \choose j}\right) = i - 1 - v_p(j). \nonumber
    \end{equation}
\end{lemma}
\begin{proof}
Exercise.
\end{proof}

\begin{lemma}\label{valuation of cyclo poly - p}
For any $i \geq 1$, we have $v_p\left(1-\phi_i(T)/p, \frac{1}{\phi(p^i)}\right) = 0$ and for $l \geq 1$, we have $v_p\left(1-\phi_i(T)/p, \frac{1}{\phi(p^l)}\right) > i-l-1$.
\end{lemma}

\begin{proof}
Since $\phi_i(T)$ is $\frac{1}{\phi(p^i)}$-extremal, we see that by \eqref{extremal}, $v_p\left(\phi_i(T), \frac{1}{\phi(p^i)}\right) = 1$. Since all the terms except the constant terms of the polynomials $\phi_i(T)$ and $\phi_i(T)-p$ are the same, we get $v_p\left(\phi_i(T)-p, \frac{1}{\phi(p^i)}\right) = 1$. Therefore $v_p\left(1 - \phi_i(T)/p, \frac{1}{\phi(p^i)}\right) = 0$.

Write $\phi_i(T) - p = a_1T + a_2T^2 + \cdots + a_{\phi(p^i)}T^{\phi(p^i)}$. 
By expanding the powers of $(1 + T)$ in $\phi_i(T)$, we see that
\begin{equation}
    a_j = {p^{i-1} \choose j} + {2p^{i-1} \choose j} + \cdots + {(p-1)p^{i-1} \choose j}. \nonumber
\end{equation}
By Lemma \ref{valuation of binomial coefficients}, 
    \begin{equation}
        v_p(a_j) + \frac{j}{\phi(p^l)} \geq i - 1 - v_p(j) + \frac{p^{v_p(j)}}{\phi(p^l)} = i - 1 - v_p(j) + \frac{p^{v_p(j) - (l-1)}}{p-1} > i - 1 - v_p(j) + v_p(j) - (l - 1) = i - l. \nonumber
    \end{equation}
    Therefore $v_p\left(1 - \phi_i(T)/p, \frac{1}{\phi(p^l)}\right) = v_p\left(\phi_i(T) - p, \frac{1}{\phi(p^l)}\right) - 1 > i - l - 1$.
\end{proof}

 We now prove that the product $\prod_{i > n}(1 - (1 - (\phi_i(T)/p))^{r + 1})$ converges in $\mathcal{E}^{]0, \frac{1}{p-1}]}$ for $n, r \geq 1$. Consider the polynomials $1 - (1 - (\phi_i(T)/p))^{r + 1}$ as elements of $\mathcal{E}^{]0, \infty]}$. By \cite[Proposition 4.11]{Laz}, the product $\prod_{i > n}(1 - (1 - \phi_i(T)/p)^{r + 1})$ converges in $\mathcal{E}^{]0, \infty]}$ if and only if 
 \begin{equation}
       \label{valtoinfty}
       v_p((1 - \phi_i(T)/p)^{r + 1}, s) \to \infty \> \text{ as } \> i \to \infty,
\end{equation}
for each $s > 0$. 
Fix $s > 0$ and pick $l_s \geq 1$ such that $\frac{1}{\phi(p^{l_s})} < s$. As $(1 - \phi_i(T)/p)^{r + 1}$ contains no negative power of $T$, we have $$v_p((1 - \phi_i(T)/p)^{r + 1}, s) \geq 
(r + 1)v_p\left(1 - \phi_i(T)/p, \frac{1}{\phi(p^{l_s})}\right)  > (r + 1)(i - l_s - 1),$$ by Lemma \ref{valuation of cyclo poly - p},
whence \eqref{valtoinfty} holds.
    Therefore the product $\prod_{i > n}(1 - (1 - (\phi_i(T)/p))^{r + 1})$ converges in $\mathcal{E}^{]0, \infty]} \subseteq \mathcal{E}^{]0, \frac{1}{p-1}]}$.

    Our candidate for $G(\norm{x}, r + 1)$ is $\sum_{n = 1}^{\infty}G_{n, \, r + 1}(T)$. To see that the infinite sum is well defined, we prove the following 
    theorem.
    
%
  
\begin{theorem}
    For any positive integer $l$, the sequence $G_{n, \, r + 1}(T)$ converges to $0$ with respect to the valuation $v_p\left(\quad, \frac{1}{\phi(p^l)}\right)$.
\end{theorem}

\begin{proof}
  Fix $l \geq 1$. We need to show $v_p\left(G_{n, \, r + 1}(T), \frac{1}{\phi(p^l)}\right) \to \infty$ as $n \to \infty$.
  Consider $v_p\left(G_{n, \, r + 1}(T), \frac{1}{\phi(p^l)}\right) = $
\[\arraycolsep=2pt
    \begin{array}{rcl}
        (r + 1)v_p\left(1 - \left(\frac{\phi_n(T)}{p} \right)^{r + 1}, \frac{1}{\phi(p^l)}\right) - n + 
        (r + 1)\sum_{i > n} v_p\left(1 - \left(1 - \frac{\phi_i(T)}{p}\right)^{r + 1}, \frac{1}{\phi(p^l)}\right).
    \end{array}
  \]  
Note that
    $1 - \left(\frac{\phi_n(T)}{p}\right)^{r + 1} = \left(1 - \frac{\phi_n(T)}{p}\right)\left(1 + \frac{\phi_n(T)}{p} + \cdots + \left(\frac{\phi_n(T)}{p}\right)^r\right)$.
  Now if $n \geq l$, then $v_p\left(\phi_n(T), \frac{1}{\phi(p^{l})}\right) = 1$. Indeed, since $\phi_n(T) = a_0 + a_1 T + \cdots + a_{\phi(p^n)} T^{\phi(p^n)}$
  is $\frac{1}{\phi(p^n)}$-extremal, we have 
        $v_p(a_i) + \frac{i}{\phi(p^l)} \geq v_p(a_i) + \frac{i}{\phi(p^n)} \geq v_p(a_0) = 1$, 
        for any $i$. 
    Hence $v_p\left(\phi_n(T), \frac{1}{\phi(p^l)}\right) = v_p(a_0) = 1$.
    Thus $v_p\left((\phi_n(T)/p)^j, \frac{1}{\phi(p^l)}\right) = 0$, for any $j \geq 0$.
    Therefore
\begin{equation}
    v_p\left(1 - \left(\frac{\phi_n(T)}{p} \right)^{r + 1}, \frac{1}{\phi(p^l)}\right) \geq v_p\left(1 - \frac{\phi_n(T)}{p}, \frac{1}{\phi(p^l)}\right) > n - l - 1,
    \nonumber
\end{equation}
where the last inequality is a consequence of Lemma \ref{valuation of cyclo poly - p}.
On the other hand, writing $1 - (1 - (\phi_i(T)/p))^{r + 1} = a'_0 + a'_1 T + \cdots + a'_{(r + 1)\phi(p^i)} T^{(r + 1)\phi(p^i)}$, we observe that $a'_0 = 1$ and $a'_1T + \cdots + a'_{(r + 1)\phi(p^i)} T^{(r + 1)\phi(p^i)} = -(1 - (\phi_i(T)/p))^{r + 1}$. By Lemma \ref{valuation of cyclo poly - p}, $v_p\left(-(1 - \phi_i(T)/p)^{r + 1}, \frac{1}{\phi(p^l)}\right) > 0$ if $i > l$. We also know that $v_p(a'_0) = 0$. Therefore $v_p\left(1 - (1 - \phi_i(T)/p)^{r + 1}, \frac{1}{\phi(p^l)}\right) = 0$.


Using these estimates together, we see that for $n \geq l$, 
\begin{equation}
    v_p\left(G_{n, \, r + 1}(T), \frac{1}{\phi(p^l)}\right) > (r + 1)(n - l - 1) - n = nr - (r + 1)l - (r + 1).
    \nonumber
\end{equation}
Letting  $n \to \infty$, we see that $v_p\left(G_{n, \, r + 1}(T), \frac{1}{\phi(p^l)}\right) \to \infty$, as desired.
Therefore the sequence $G_{n, \, r + 1}(T)$ converges to $0$ with respect to the valuation $v_p\left(\quad, \frac{1}{\phi(p^l)}\right)$ for each $l > 0$.
\end{proof}

Using the theorem above, we see that the series $G(\norm{x}, r + 1) = \sum_{n = 1}^{\infty}G_{n, \, r + 1}(T)$ converges. 
It is easily checked, using $\phi_i(T) \equiv p \mod \phi_n(T)$ for $i > n$, that it is a solution to the system of congruences
that we started with:
\begin{equation}
    G(\norm{x}, r + 1) \equiv \frac{1}{p^n} \mod{\phi_n(T)^{r + 1}}, \mathrm{\text{ for all }} n \in \mathbb{N}. \nonumber
\end{equation}
Therefore, we have written down $G(\norm{x}, r + 1)$ explicitly.

\subsection{A partial description of $G'(\norm{x}, r + 1)$}

Let $r \geq 1$. The second basis vector of $H^1(\calR_{\qp}(x^r\chi))$ constructed by Colmez is represented by the element
\[
    c_2 = (t^{-(r + 1)}(p^{-1}\varphi - 1)(\log T - G'(\norm{x}, r + 1)), t^{-(r + 1)}(\gamma - 1)(\log T - G'(\norm{x}, r + 1)))
\]
in $\calR_{\qp}(x^r\chi) \oplus \calR_{\qp}(x^r\chi)$ (for the untwisted action of $\varphi$ and $\gamma$), see \cite[Proposition 2.19]{Col}.
Here $G'(\norm{x}, r + 1)$ is an element of $\mathcal{E}^{]0, \frac{1}{p - 1}]}$
mapping to $\prod_{n \geq 1} \log(\zetapn e^{t/p^n} - 1)$ under the map \eqref{G'} (cf. \cite[Section 2.7]{Col}).

We wish to compute $G'(\norm{x}, r + 1)$ explicitly, just as we did for $G(\norm{x}, r + 1)$. We are only able to determine
$G'(\norm{x}, r+1)$ modulo $t$ but this is sufficient for our purposes.

Consider the following element
\begin{equation}\label{eq:7}
    g'(T) = \frac{\log\gamma}{\gamma - 1}\log T - \log T - \frac{\log\chi(\gamma)t}{\chi(\gamma) - 1} \in \mathcal{E}^{]0, \frac{1}{p - 1}]}.
\end{equation}
In the proof of \cite[Theorem 1.5.7]{Ben}, Benois relates $g'(T)$ (which he calls $y$) to the element $d_1 = -t^{-1}\nabla_0(\log{T}) + (\chi(\gamma) - 1)^{-1}$ of $\calR_{\qp}[\log T, 1/t]$ using the following equation
\begin{equation}\label{eq:8}
    d_1 = -\log(\chi(\gamma))^{-1}t^{-1}(\log T + g'(T)),
\end{equation}
where $\nabla_0 = \frac{1}{\log{\chi(\gamma)}}\frac{\log{\gamma}}{\gamma - 1}$.

\par We claim that $\iota_n(-g'(T)) \equiv \log(\zetapn e^{t/p^n} - 1) \mod{t}$. Indeed,
\[\arraycolsep=2pt
    \def\arraystretch{1.75}
    \begin{array}{ll}
        \iota_n(g'(T)) & = \iota_n\left(\frac{\log\gamma}{\gamma - 1}\log T - \log T - \frac{\log\chi(\gamma)t}{\chi(\gamma) - 1}\right) \\
        & \equiv \frac{\log\gamma}{\gamma - 1}\log{(\zetapn - 1)} - \log(\zetapn - 1) \mod{t} \\
        & \equiv \frac{1 + \gamma + \cdots + \gamma^{p^{n - 1}(p - 1) - 1}}{p^{n - 1}(p-1)}\frac{\log\gamma^{p^{n - 1}(p-1)}}{\gamma^{p^{n - 1}(p-1)} - 1}\log(\zetapn - 1) - \log(\zetapn - 1) \mod{t}.
    \end{array}
\]
Since $\gamma^{p^{n - 1}(p - 1)}\log(\zetapn - 1) = \log(\zetapn - 1)$, we get
\[\arraycolsep=2pt
    \def\arraystretch{1.75}
    \begin{array}{ll}
        \iota_n(g'(T)) & \equiv \frac{1 + \gamma + \cdots + \gamma^{p^{n - 1}(p - 1) - 1}}{p^{n - 1}(p-1)} \log(\zetapn - 1) - \log(\zetapn - 1) \mod{t} \\
                       & \equiv \frac{1}{p^{n - 1}(p - 1)}\log(N_{\qp(\zetapn)/\qp}(\zetapn - 1)) - \log(\zetapn - 1) \mod{t}\\
                       & \equiv -\log(\zetapn - 1) \mod{t} \\
                       & \equiv - \log(\zetapn e^{t/p^n} - 1) \mod{t}.                        
    \end{array}
\]

Since the images of $-g'(T)$ and $G'(\norm{x}, r + 1)$ under the isomorphism \eqref{G'} are equal (with $r$ there equal to zero!), there exists
$g''(T) \in \mathcal{E}^{]0, \frac{1}{p - 1}]}$ such that 
\begin{equation}\label{eq:9}
    G'(\norm{x}, r + 1) = -g'(T) + tg''(T).
  \end{equation}
This determines $G'(\norm{x}, r + 1)$ explicitly modulo $t$.

\subsection{Another basis of $H^1(\calR_{\qp}(x^r\chi))$}
Let $r \geq 1$. In this section, we consider a different basis $\{ \br{\alpha}_{r + 1}, \br{\beta}_{r + 1}\}$ of $H^1(\calR_{\qp}(x^r\chi))$ described by Benois in \cite[Proposition 1.5.4]{Ben} which is more suitable for computations. Here bar denotes class in $H^1(\calR_{\qp}(x^r\chi))$. By \cite[Corollary 1.5.5]{Ben}, for any class in $H^1(\calR_{\qp}(x^r\chi))$ represented by $(a, b) \in \calR_{\qp}(x^r\chi)\oplus\calR_{\qp}(x^r\chi)$, we have
\begin{equation}
    \label{eq:10}
    \br{(a,b)} = \lambda' \cdot \br{\alpha}_{r + 1} + \mu' \cdot \br{\beta}_{r + 1}, \text{ where } \lambda' = \mathrm{res}(at^rdt) \text{ and } \mu' = \mathrm{res}(bt^rdt).
  \end{equation}

  We want to compute the change of basis relation between Colmez's basis and Benois' basis. This allows us to state the formula for the $\sL$-invariant of a non-zero cohomology class 
  in terms of certain residues.
\begin{Proposition}\label{change of basis}
    Let 
    \[\arraycolsep=2pt
        \begin{array}{rcl}
            c_1 & = & (t^{-(r + 1)}(p^{-1}\varphi - 1)G(\norm{x}, r + 1), \> t^{-(r + 1)}(\gamma - 1)G(\norm{x}, r + 1)), \\ [3pt]
            c_2 & = & (t^{-(r + 1)}(p^{-1}\varphi - 1)(\log T - G'(\norm{x}, r + 1)), \> t^{-(r + 1)}(\gamma - 1)(\log T - G'(\norm{x}, r + 1)))
        \end{array}
    \]
    be elements of $\calR_{\qp}(x^r\chi)\oplus\calR_{\qp}(x^r\chi)$ representing Colmez's basis of $H^1(\calR_{\qp}(x^r\chi))$. Then 
    \[
         \br{c}_1 = - \frac{p - 1}{p} \cdot \br{\alpha}_{r + 1}, \qquad \br{c}_2 =  \log(\chi(\gamma)) \cdot \br{\beta}_{r + 1}.
    \]
\end{Proposition}
\begin{proof}
We prove the proposition using $(\ref{eq:10})$. We therefore need to compute the four residues
\[
    \begin{array}{lll}
        \mathrm{res}(t^{-1}(p^{-1}\varphi - 1)G(\norm{x}, r + 1)dt), & & \mathrm{res}(t^{-1}(\gamma - 1)G(\norm{x}, r + 1)dt), \\ [3pt]
        \mathrm{res}(t^{-1}(p^{-1}\varphi - 1)(\log T - G'(\norm{x}, r + 1))dt), & & \mathrm{res}(t^{-1}(\gamma - 1)(\log T - G'(\norm{x}, r + 1))dt).
    \end{array}
\]
We shall use the following formulas: for any $f(T) \in \calR_{\qp}$, we have
\begin{eqnarray}\
    \mathrm{res}(\varphi(f(T))dt) & = & \mathrm{res}(f(T)dt), \label{eq:11} \\
    \mathrm{res}(\gamma(f(T))dt) & = & \chi(\gamma)^{-1}\mathrm{res}(f(T)dt). \label{eq:12}
\end{eqnarray}
\begin{itemize}
\item Recall that $G(\norm{x}, r + 1) = \sum_{n = 1}^{\infty}G_{n, \, r + 1}(T)$, where
\[
    \begin{array}{ll}
    & G_{n, \, r + 1}(T) = \Big[1 - \Big(\frac{\phi_n(T)}{p}\Big)^{r + 1}\Big]^{r + 1} \frac{1}{p^n} \prod_{i > n}\Big[1 - \Big(1 - \frac{\phi_i(T)}{p}\Big)^{r + 1}\Big]^{r + 1}.
    \end{array}
\]
Using equation $(\ref{eq:3})$, we see that $(p^{-1}\varphi)G_{n, \, r + 1}(T) = G_{n+1, \, r + 1}(T)$. Therefore
\begin{equation}
    (p^{-1}\varphi - 1)G(\vert x \vert, r + 1) = -G_{1, \, r + 1}(T).
    \nonumber
\end{equation}
Hence $\mathrm{res}\left(t^{-1}\hspace{-3pt}\left(p^{-1}\varphi - 1\right)G(\vert x \vert, r + 1)dt\right) = - \mathrm{res}(t^{-1}G_{1, \, r + 1}(T)dt)$. Now
\[\arraycolsep=2pt
    \begin{array}{ll}
        \mathrm{res}(t^{-1}G_{1, \, r + 1}(T)dt) & = \mathrm{res}\Big(t^{-1}\Big[1 - \Big(\frac{\phi_1(T)}{p}\Big)^{r + 1}\Big]^{r + 1}\frac{1}{p} {\prod_{i > 1}}\Big[1 - \Big(1 - \frac{\phi_i(T)}{p}\Big)^{r + 1}\Big]^{r + 1} dt\Big)  \\
         & \stackrel{(\ref{eq:4})}{=} \mathrm{res}\Big(T^{-1}\Big(\frac{\phi_1(T)}{p}\Big)^{-1}\Big[1 - \Big(\frac{\phi_1(T)}{p}\Big)^{r + 1}\Big]^{r + 1}\frac{1}{p}f(T) dt\Big),
    \end{array}
\]
where $f(T) = \prod_{i > 1}\Big(\frac{\phi_i(T)}{p}\Big)^{-1}\Big[1 - \Big(1 - \frac{\phi_i(T)}{p}\Big)^{r + 1}\Big]^{r + 1}$. An argument similar to those above
shows that $f(T) \in \mathcal{E}^{]0, \frac{1}{p-1}]}$. Moreover, since the value of $\phi(T)/p$ at $T = 0$ is $1$, we see that $f(T)  \in 1 + T\qp[[T]]$.
Using equation (\ref{eq:3}), we see that there exists $g(T) \in 1 + T\qp[[T]] \cap \mathcal{E}^{]0, \frac{1}{p-1}]}$ such that $f(T) = \varphi(g(T))$. We therefore have
\[
        \mathrm{res}(t^{-1}G_{1, \, r + 1}(T)dt) =  \mathrm{res}\Big(T^{-1}\Big(\frac{\phi_1(T)}{p}\Big)^{-1}\Big[1 - \Big(\frac{\phi_1(T)}{p}\Big)^{r + 1}\Big]^{r + 1}\frac{1}{p}\varphi(g(T))dt\Big).
\]
Adding and subtracting $1$ from the term $[1 - (\frac{\phi_1(T)}{p})^{r + 1}]^{r + 1}$, we get
\[
    \begin{array}{ll}
        & \mathrm{res}(t^{-1}G_{1, \, r + 1}(T)dt) \\
        & \stackrel{(\ref{eq:2})}{=} \mathrm{res}\Big(\frac{1}{(1+T)^p - 1}\varphi(g(T))dt\Big) + \mathrm{res}\Big(T^{-1}\Big(\frac{\phi_1(T)}{p}\Big)^{-1}\Big(\Big[1 - \Big(\frac{\phi_1(T)}{p}\Big)^{r + 1}\Big]^{r + 1} - 1\Big)\frac{1}{p}\varphi(g(T))dt\Big).
    \end{array}
\]
The first term  is equal to $\mathrm{res}\Big(\varphi\Big(\frac{g(T)}{T}\Big)dt\Big)$, which by (\ref{eq:11}), 
equals $\mathrm{res}\Big(\frac{g(T)}{T}dt\Big) = 1$, since the constant term of $g(T)$ is $1$. 
For the second term, we see that $\frac{\phi_1(T)}{p}$ divides $\Big[1 - \Big(\frac{\phi_1(T)}{p}\Big)^{r + 1}\Big]^{r + 1} - 1$ as polynomials. Also, the constant term
 of $\Big(\frac{\phi_1(T)}{p}\Big)^{-1}\Big(\Big[1 - \Big(\frac{\phi_1(T)}{p}\Big)^{r + 1}\Big]^{r + 1} - 1\Big)\varphi(g(T))$ is $-1$. Hence, the second residue
is $-1/p$. Therefore,
\[
    \mathrm{res}\left(t^{-1}\left(p^{-1}\varphi - 1\right)G(\vert x \vert, r + 1)dt\right) = \frac{1}{p} - 1.
\]

\item We prove $\mathrm{res}\left(t^{-1}(\gamma - 1)G(\vert x \vert, r + 1)dt\right) = 0$, by showing that for all $n \geq 1$,
\begin{equation}
    \mathrm{res}\left(t^{-1}(\gamma - 1)G_{n, \, r + 1}(T)dt\right) = 0.
    \nonumber
\end{equation}
Again by (\ref{eq:3}), we see that $(p^{-1}\varphi)G_{n, \, r + 1}(T) = G_{n+1, \, r + 1}(T)$. Therefore
\[\arraycolsep=2pt
    \def\arraystretch{1.2}
    \begin{array}{rcl}
         \mathrm{res}\left(t^{-1}(\gamma - 1)G_{n, \, r + 1}(T)dt\right) & = & \mathrm{res}\left(t^{-1}(\gamma - 1)(p^{-1}\varphi)^{n - 1}G_{1, \, r + 1}(T)dt\right) \\
         & \stackrel{(\ref{eq:5})}{=} & \mathrm{res}\left(\varphi^{n-1}t^{-1}(\gamma - 1)G_{1, \, r + 1}(T)dt\right) \\
         & \stackrel{(\ref{eq:6})}{=} & \mathrm{res}\left(\varphi^{n - 1}(\chi(\gamma)\gamma - 1)(t^{-1}G_{1, \, r + 1}(T))dt\right),
    \end{array}
\]
which vanishes by (\ref{eq:11}) and (\ref{eq:12}). 

\item Using (\ref{eq:9}), we get $\log{T} - G'(\norm{x}, r + 1) = \log{T} + g'(T) - tg''(T)$, whereas rearranging the terms in equation (\ref{eq:8}), we get $\log T + g'(T) = -\log(\chi(\gamma))td_1$. These two equations imply
\[
    \log{T} - G'(\norm{x}, r + 1) = -\log{(\chi(\gamma))}td_1 - tg''(T).
\]
Therefore we have
\[
    t^{-1}(p^{-1}\varphi - 1)(\log{T} - G'(\norm{x}, r + 1)) \stackrel{(\ref{eq:5})}{=} -\log{(\chi(\gamma))}(\varphi - 1)d_1 -(\varphi - 1)g''(T).
\]
By the discussion on \cite[p. 1604,  p. 1601]{Ben}, $\mathrm{res}(((\varphi - 1)d_1)dt) = 0$. By (\ref{eq:11}), we get
\[
    \mathrm{res}(t^{-1}(p^{-1}\varphi - 1)(\log T - G'(\norm{x}, r + 1))dt) =  0.
\]
\item Similarly, we have
\[
    t^{-1}(\gamma - 1)(\log{T} - G'(\norm{x}, r + 1)) \stackrel{(\ref{eq:6})}{=} -\log{(\chi(\gamma))}(\chi(\gamma)\gamma - 1)d_1 - (\chi(\gamma)\gamma - 1)g''(T).
\]
Using $(\chi(\gamma)\gamma - 1)d_1 = -\frac{1}{T}$ (see the proof of \cite[Theorem 1.5.7 iiib)]{Ben} and using equation (\ref{eq:12}), we get
\[
        \mathrm{res}(t^{-1}(\gamma - 1)(\log T - G'(\norm{x}, r + 1))dt) = \log{\chi(\gamma)}.
        \nonumber
\]
\end{itemize}
Using equation (\ref{eq:10}), we see that $\br{c}_1$ is $-\frac{p-1}{p} \cdot \br{\alpha}_{r+1}$ and $\br{c}_2$ is $\log{(\chi(\gamma))}\cdot \br{\beta}_{r + 1}$.
\end{proof}
Let us state the formula for the $\sL$-invariant using  \cite[Definition $2.20$]{Col}. Let $(a, b)$ represent a non-zero element of $H^1(\calR_{\qp}(x^r\chi))$. Writing $\br{(a, b)} = \lambda \cdot \br{c}_1 + \mu \cdot \br{c}_2$, where the bar over an element denotes its class in $H^1(\calR_{\qp}(x^r\chi))$, Colmez defines the $\sL$-invariant associated to the class of $(a, b)$ to be $\frac{p - 1}{p}\frac{\lambda}{\mu} \in \mathbb{P}^1(\qp)$.\footnote{
      The sign is the opposite of the one in \cite[Definition 2.20]{Col}.
      We believe that this choice
      of sign is correct.
      Indeed, consider the image of $p(1 + p)$ under the Kummer map $\kappa : \qp^* \to H^1(\qp, \qp(1))$. By the discussion in \cite[Section 1.5.6]{Ben}, we see that with the original choice of sign, the $\sL$-invariant associated to the image of $\kappa(p(1 + p))$ under the canonical isomorphism $H^1(\qp, \qp(1)) \simeq H^1(\calR_{\qp}(x\norm{x}))$ is $-\log(1 + p)$. However, using Tate's formula, the $\sL$-invariant associated to $\kappa(p(1 + p))$ is equal to $\frac{\log(p(1 + p))}{v_p(p(1 + p))} = \log(1 + p)$. The original sign is also incompatible with
    our results in the Introduction.}
Using Proposition \ref{change of basis}, we can restate the formula for the $\sL$-invariant in terms of Benois' basis.

\begin{definition}
  \label{def of L inv}
  If $\br{(a, b)} = \lambda' \cdot \br{\alpha}_{r + 1} + \mu' \cdot \br{\beta}_{r + 1} \neq 0$, where $\br{(a, b)}$ is the class of $(a, b)$ in $H^1(\calR_{\qp}(x^r\chi))$, then the $\sL$-invariant associated to this class is
\[
    \sL = -\mathrm{log(\chi(\gamma))} \cdot \frac{\lambda'}{\mu'}.
\]
\end{definition}

\section{Cohomology of $m\RU(\delta_{\U})$}
\label{big phi gamma modules}

The aim of this section is to find a generator of $H^1(m\RU(\delta_{\U}))$ for the tautological character $\delta_{\U} : \qp^* \to \co(\U)^*$, which will be defined in the first subsection. This generator is an essential ingredient in the proof of Theorem \ref{The L invariant}. 

\subsection{The tautological character $\delta_{\U} : \qp^* \to \co(\U)^*$}
\label{tautological}

Define $\delta_{\U}$ by the following equations
\begin{equation}\label{def of taut char}
\arraycolsep=2pt
    \begin{array}{rcl}
        \delta_{\U}(p) & = & S_1, \\
        \delta_{\U}(\zetap) & = & \zetap^{r+1}, \\
        \delta_{\U}(1 + p) & = & 1 + S_2.
    \end{array}
\end{equation} 
This character is tautological in the following sense. If $\delta$ is an $E$-valued character of $\qp^*$ and if $(\delta(p), \delta(\zetap), \delta(1 + p) - 1) \in \U(E)$, then we have the following commutative diagram
\[
    \begin{tikzcd}
          &  \co(\U)^* \ar[d] \\
         \qp^* \ar[ur, "\delta_{\U}"] \ar[r, "\delta"] & E^*,
    \end{tikzcd}
\]
where the vertical map is induced by the affinoid algebra map $\co(\U) \to E$ associated to the $E$-valued point $(\delta(p), \delta(\zetap), \delta(1 + p) - 1)$ of $\U$.
\par We prove that $\delta_{\U}$ is bien pla\c{c}\'e in the sense of \cite[Definition $2.31$]{Che}.

\begin{lemma}
$\delta_{\U}$ is bien pla\c{c}\'e.
\end{lemma}
\begin{proof}
\begin{itemize}
    \item Given any $i \in \mathbb{Z}$, we have to check that $1 - \delta_{\U}(p)p^i = 1 - S_1 p^i$ is not a zero divisor in $\mathcal{O}(\U)$. But this follows immediately from the fact that $\qp \langle S_1, S_2 \rangle$ is a domain and  $\qp \langle S_1, S_2  \rangle \to \co(\U)$ is flat (\cite[Corollary 7.3.2/6]{BGR84}).
    \item Given any $i \geq 0$, we have to check that the image of $1 - \delta_{\U}(\chi(\gamma))\chi(\gamma)^{1 - i} = 1 - \zeta_{p-1}^{a(r+2-i)} (1+S_2) (1+p)^{1-i}$ in the quotient $\mathcal{O}(\U)/(1 - \delta_{\U}(p)p^{-i})$ is not a zero divisor. 
      This again follows from the facts that $\qp \langle S_1, S_2  \rangle / (1 - \delta_{\U}(p)p^{-i}) \to \co(\U) / (1 - \delta_{\U}(p)p^{-i})$ is flat (\cite[Corollary 7.3.2/6]{BGR84}), that the above element is non-zero on the left (easy to check) and that 
      the ring on the left is a domain (indeed, if $f(S_1,S_2) g(S_1,S_2) = (1-S_1p^{-i}) h(S_1, S_2)$, then applying the Weierstrass division theorem
 (\cite[Corollary 5.2.1/2]{BGR84}), we have $f = q (1-S_1 p^{-i}) + r$, $g  = q' (1-S_1 p^{-i}) + r'$, for some $q, q' \in 
 \qp \langle S_1, S_2 \rangle$  and  $r, r' \in \qp \langle S_2 \rangle$, so that $(1-S_1 p^{-i}) | rr'$ and hence one 
 of $r$ or $r'$ must be zero). 
%
%
\qedhere
\end{itemize}
\end{proof}

\subsection{Computation of $H^1(m\RU(\delta_{\U}))$}
In this section, we make explicit some facts in \cite{Che} about the first Fontaine-Herr cohomology groups of certain `big' 
$(\varphi, \Gamma)$-modules.

Let $\RU(\delta_{\U})$ be the Robba ring over $\co(\U)$ in the variable $T$, 
with the standard actions of  $\varphi$ and $\Gamma$
twisted by $\delta_{\U}$: $\varphi T = \delta_{\U}(p) ((1+T)^p - 1)$ and $\gamma T = \delta_{\U}(\chi(\gamma)) ((1+T)^{\chi(\gamma)} - 1)$.
Let $\psi$ be the usual left inverse of $\varphi$,
twisted by $\delta_{\U}(p^{-1})$. Let $\RU^+(\delta_{\U})$ be the power series in $\RU(\delta_{\U})$ 
consisting of non-negative powers of $T$. 

We work with two different versions of the Fontaine-Herr cohomology groups for $(\varphi, \Gamma)$-modules $D$ over the Robba ring $\RU$, namely, the $(\varphi, \Gamma)$-version
  $H^1_{(\varphi, \Gamma)}(D)$ and the $(\psi, \Gamma)$-version $H^1_{(\psi, \Gamma)}(D)$.  The former groups are defined as usual as
  the cohomology groups of the complex \eqref{small (phi, Gamma) complex}.
  For the definition of the latter groups, one replaces $\varphi$ by $\psi$ in the maps in the complex \eqref{small (phi, Gamma) complex}. There is a map $\eta$ from the
  $(\varphi, \Gamma)$-complex to the $(\psi, \Gamma)$-complex
\[
    \begin{tikzcd}
        0 \ar[r] & D \ar[r] \ar[d, "\eta_0"] & D \oplus D \ar[r] \ar[d, "\eta_1"] & D \ar[r] \ar[d, "\eta_2"] & 0 \\
        0 \ar[r] & D \ar[r] & D \oplus D \ar[r] & D \ar[r] & 0,
    \end{tikzcd}
\]
where
\[
    \eta_0(x) = x, \quad \eta_1(x, y) = (-\psi(x), y) \quad \text{ and } \eta_2(x) = -\psi(x).
\]
If $\gamma - 1$ is bijective on $D^{\psi = 0}$, then $\eta$ induces an isomorphism on cohomology:
$H^1_{(\varphi, \Gamma)}(D) \simeq H^1_{(\psi, \Gamma)}(D)$ (see \cite[Section 2]{Che}).

By \cite[Theorem 2.33]{Che},  the cohomology group $H^1_{(\varphi, \Gamma)}(\RU(\delta_{\U}))$ is a free module of rank $1$ over $\co(\U)$.
First, we compute an $\co(\U)$-generator of $H^1_{(\varphi, \Gamma)}(\RU(\delta_{\U}))$.  By  \cite[Theorem 2.33]{Che}, the inclusion $m\RU(\delta_{\U}) \to \RU(\delta_{\U})$ induces an isomorphism $H^1_{(\varphi, \Gamma)}(m\RU(\delta_{\U})) \to H^1_{(\varphi, \Gamma)}(\RU(\delta_{\U}))$. Second, we use this isomorphism to lift the $\co(\U)$-generator of $H^1_{(\varphi, \Gamma)}(\RU(\delta_{\U}))$ to an $\co(\U)$-generator of $H^1_{(\varphi, \Gamma)}(m\RU(\delta_{\U}))$ (see Definition~$\ref{gen of H1mD}$).

Let $D = \RU(\delta_{\U})$. In Section~\ref{tautological}, we showed that $\delta_{\U}$ is bien pla\c{c}\'e. 
By \cite[Proposition 2.32]{Che} we have the following isomorphisms
\[
    \begin{tikzcd}
    C(D^+)/(\gamma - 1) \ar[r] & C(D)/(\gamma - 1) & D^{\psi = 1}/(\gamma - 1) \ar[l] \ar[r] & H_{(\psi, \Gamma)}^1(D),
    \end{tikzcd}
\]
where $C(D) = (1 - \varphi)D^{\psi = 1}, D^+ = \RU^+(\delta_{\U})$ and $C(D^+) = (1 - \varphi)(D^+)^{\psi = 1}$, and where we recall that the actions of $\varphi$ and $\psi$ are the usual ones twisted by $\delta_{\U}$. 
The first map is induced by the canonical injection of $D^+$ into $D$. The second map is induced by $f(T) \mapsto (1 -
\varphi)f(T)$. The third map is induced by $f(T) \mapsto (0, f(T))$. We emphasize that the cohomology group appearing in the display above is in the sense of $(\psi, \Gamma)$-modules. 

We wish to find an $\co(\U)$-generator of $H^1_{(\psi, \Gamma)}(D)$. We do so by finding an $\co(\U)$-generator of $C(D^+)/(\gamma - 1)$ and then using the isomorphisms above to get an $\co(\U)$-generator of $H^1_{(\psi, \Gamma)}(D)$. To find an $\mathcal{O}(\U)$-generator of $C(D^+)/(\gamma - 1)$, we prove a couple of lemmas.
\begin{lemma}
    $C(D^+) = (D^+)^{\psi = 0}$.
\end{lemma}
\begin{proof}
By definition, $C(D^+)$ is the image of the map $1 - \varphi : (D^+)^{\psi = 1} \to (D^+)^{\psi = 0}$. 
Applying \cite[Lemma 2.9(vi)]{Che} with $\lambda = S_1$ and $N = 0$ ($r \geq 1 \Rightarrow \norm{S_1} < 1$), we see that this map is a bijection. The lemma follows.
\end{proof}
\begin{lemma}
$S_1(1+T)$ generates $(D^+)^{\psi = 0}/(\gamma - 1)$ as a free $\co(\U)$-module of rank $1$.
\end{lemma}
\begin{proof}
In order to prove this, we apply \cite[Proposition $2.14$]{Che}. We know that $D^+$ is a free $\mathcal{R}_{\mathcal{O}(\U)}^+$-module of rank one, say with basis $w$. To check that $D^+$ is $\Gamma$-bounded in the sense of \cite[Section 1.8]{Che},
we choose the following model of $\mathcal{O}(\U)$
\[
     \mathcal{A} = \zp \langle S_1, S_2, T_1, T_2, T_3 \rangle/(p^rT_1 - S_1, 1 - T_1T_2, pT_3 - S_2).
\]
Then for $\gamma' \in \Gamma$ we have
$\mathrm{Mat}(\gamma') = [\delta_{\U} (\chi(\gamma'))] 
\in M_1(\mathcal{A}[[T]])$, so $D^+$ is $\Gamma$-bounded. Using \cite[Proposition $2.14$]{Che}, we see that $\{S_1(1+T)w\}$ is a basis of $(D^+)^{\psi = 0}$ over $\mathcal{R}_{\mathcal{O}(\U)}^+(\Gamma)$ (see \cite[Section 2.12]{Che} for the definition of the last ring). Now $\calR_{\co(\U)}^+(\Gamma)/(\gamma - 1) \cong \co(\U)$. Therefore $(D^+)^{\psi = 0}/(\gamma - 1)$ is a free $\co(\U)$-module of rank one generated by $S_1(1+T)w$.
\end{proof}

For the remainder of section, we drop $w$ when we talk about the $\co(\U)$-generator $S_1(1 + T)w$ of $(D^+)^{\psi = 0}/(\gamma - 1)$.
Using these lemmas, we see that the quotient $C(D^+)/(\gamma - 1)$ is generated by $S_1(1+T)$ as a free $\mathcal{O}(\U)$-module. Pushing this generator to $C(D)/(\gamma - 1)$ under the first isomorphism $C(D^+)/(\gamma - 1) \to C(D)/(\gamma - 1)$, we see that $S_1(1+T)$ is an $\co(\U)$-generator of $C(D)/(\gamma - 1)$. Now we want a generator of $D^{\psi = 1}/(\gamma - 1)$. 

\begin{lemma}\label{y and phi - 1}
    The element $y = S_1\sum_{n = 0}^{\infty}S_1^n(1 + T)^{p^n}$ is the preimage of $S_1(1 + T)$ under the isomorphism $1 - \varphi : D^{\psi = 1}/(\gamma - 1) \to C(D)/(\gamma - 1)$. 
\end{lemma}
\begin{proof}
Note  $y \in D^{\psi = 1}$ (for the twisted $\psi$). To see this, write (for the twisted $\varphi$) 
\[
    \begin{array}{rcl}
    S_1\sum_{n = 0}^{\infty}S_1^n(1+T)^{p^n} & = &  \varphi(S_1\sum_{n = 0}^{\infty}S_1^n(1+T)^{p^n}) + (1+T) \varphi(1).
    \end{array}
\]
From the definition of $\psi$ (cf. \cite[Section 2.1]{Che}), we see that $y \in D^{\psi = 1}$. Clearly $(1 - \varphi)y = S_1(1+T)$.  
\end{proof}
Finally, consider the third isomorphism $D^{\psi = 1}/(\gamma - 1) \to H^1_{(\psi, \Gamma)}(D)$ induced by sending $f(T)$ to $(0, f(T))$. Using this isomorphism, we see that the class of $(0, y)$ is a generator of $H^1_{(\psi, \Gamma)}(D)$ as a free $\mathcal{O}(\U)$-module. We repeat that this cohomology group is as a $(\psi, \Gamma)$-modules. In order to get a generator of the $(\varphi, \Gamma)$-cohomology group, we note that $D$ is a rank $1$ module so it is tautologically trianguline, so $\gamma - 1$ is bijective on $D^{\psi = 0}$ by \cite[Corollary 2.5]{Che}, so
\[
\begin{tikzcd}[row sep = 0cm, column sep = 1.5cm]
    H^1_{(\varphi, \Gamma)}(D) \ar[r, "\text{($-\psi$, id)}"] & H^1_{(\psi, \Gamma)}(D)
\end{tikzcd}
\]
is an isomorphism of cohomology groups, by the discussion at the beginning of this section. 
Since $y \in D^{\psi = 1}$, we have $(\varphi - 1)y \in D^{\psi = 0}$.  
Since $\gamma - 1$ is bijective on $D^{\psi = 0}$, there exists a unique  $x \in D^{\psi = 0}$ such that 
$(\gamma - 1)x = (\varphi - 1)y$. 
Therefore $(x, y)$ represents a class in $H^1_{(\varphi, \Gamma)}(D)$. Moreover, the class of $(x, y)$ maps to the class of $(0, y)$ under the isomorphism above. Therefore, the class of $(x, y)$ generates $H^1_{(\varphi, \Gamma)}(D)$ as a free $\mathcal{O}(\U)$-module.

For the rest of this paper, every cohomology group that we write will be the $(\varphi, \Gamma)$-version, and we drop the subscript `${(\varphi, \Gamma)}$' from the notation.

We conclude this section by describing an $\co(\U)$-generator of $H^1(m\RU(\delta_{\U}))$.
%

Using the isomorphism $H^1(mD) \to H^1(D)$, we see that 
the class of $(x, y)$ is represented by elements of $mD$: there exists a $d \in D$ such that
$x - (\varphi - 1)d, \, y - (\gamma - 1)d \in mD$.

\begin{definition}
  \label{gen of H1mD}
  We denote by $e$ the class of  $(x - (\varphi - 1)d, \, y - (\gamma - 1)d)$ in $H^1(mD)$.
\end{definition}
Thus $e$ is an explicit generator of the free $\mathcal{O}(\U)$-module $H^1(m\RU(\delta_{\U}))$ of rank $1$.

\section{The $\sL$-invariant of the limit point}

Recall that the fiber in the blow-up $\tU$ of $\U$ over the exceptional  point $(p^r, (1 + p)^{r + 1} - 1)$ corresponding to the maximal ideal $m = (f_1,f_2)$
is parameterized by $\mathbb{P}^1(\brqp)$. The main theorem in this section gives us a formula for the $\sL$-invariant (in the sense
of Definition~\ref{def of L inv}) of the $(\varphi, \Gamma)$-module
associated to a general $E$-valued point $(p^r, (1 + p)^{r + 1} - 1, a : b)$ in the exceptional fiber. 

Fix such a point $(p^r, (1 + p)^{r + 1} - 1 , \> a : b)$. Let $v \in \mathrm{Hom}(m/m^2 \otimes_{\qp}E, E)$ be a tangent vector representing the tangent direction associated with this
point
by Proposition~\ref{fiber and tangents}. Recall that $D = \RU(\delta_{\U})$.  Consider the specialization 
map $v^* : E \otimes_{\qp}  mD \to 
\calR_E(x^r\chi)$ induced by the following composition of maps
\[
         \qquad  mD  \to m \otimes D \to m/m^2 \otimes_{\qp} D/mD \xrightarrow{v} E \otimes_{\qp} \calR_{\qp}(x^r\chi) = \R_{E}(x^r\chi) 
\]
given by
\[
        fg \mapsto f \otimes g \mapsto\br{f} \otimes \br{g}  \mapsto v(\br{f})\br{g}, \qquad
\]
for all $f \in m$, $g \in D$, where the first map is the inverse of the multiplication map $m \otimes D \to mD$
(which is an isomorphism since $\RU$ is flat over $\co(\U)$ by \cite[Lemma 1.3 (v)]{Che}),
and the second map is the map obtained by going mod $m$ on each factor. In particular, we have  
\[
  v^*(f_1g_1 + f_2g_2) 
  = v(\br{f_1})\br{g_1} + v(\br{f_2})\br{g_2}, \qquad
\]
where $g_1, g_2 \in D$ are arbitrary, $\br{g_1}, \br{g_2}$ are their images in $D/mD$,  and $\br{f_1}, \br{f_2}$ are the images of $f_1 = S_1 - p^r$ and $f_2 = 1 + S_2 - (1 + p)^{r + 1}$ in $m/m^2$, respectively.

The  map $v^*$ 
yields a specialization map $H^1(v^*) : E \otimes_{\qp} H^1(mD) \to H^1(\R_{E}(x^r\chi))$ induced by
    \[
        H^1(v^*)(\eta_1, \eta_2) = (v^*(\eta_1), v^*(\eta_2)),
    \]
for $(\eta_1, \eta_2) \in mD \oplus mD$ representing a cohomology class in $H^1(mD)$. 
Recall that $e$ is represented by $$(x - (\delta_{\U}(p)\varphi - 1)d, \, y - (\delta_{\U}(\chi(\gamma))\gamma - 1)d) \in mD\oplus mD$$
and is an $\co(\U)$-generator of $H^1(mD)$, where now $\varphi$ and $\gamma$ are the usual untwisted operators (see Definition~\ref{gen of H1mD}).
Here we have rewritten $e$ in terms of these untwisted operators since this explicit
formula for $e$ will be needed below. We make the following definition.

\begin{definition}
  \label{the phi,Gamma module}
  The $(\varphi, \Gamma)$-module associated\footnote{There is a more conceptual definition of this $(\varphi, \Gamma)$-module in terms of
    a cover $U_i$ of $\tU$ as the module $(D_i)_z$, with notation as in the proof of \cite[Proposition 3.9]{Che}, but this definition coincides with the more computational
    one given here by \cite[Theorem 2.33]{Che} (see the remarks at the end of that proof). } to the point $(p^r, (1 + p)^{r + 1} - 1, \> a : b)$ is  the image of
  $e$ in $H^1(\R_{E}(x^r\chi))$ under $H^1(v^*)$.
\end{definition}

\noindent Note that the generator $e$ and the tangent vector $v$ are only well defined up to scalars, but the isomorphism class of this
$(\varphi, \Gamma)$-module is unchanged under multiplication by scalars.  

The theorem below gives us a formula for the $\sL$-invariant of the $(\varphi, \Gamma)$-module
associated to $(p^r, (1 + p)^{r + 1} - 1, a : b)$ in the exceptional fiber in terms of $a$ and $b$.
Note that the $\sL$-invariant only depends on the projective image of the cohomology class $H^1(v^*)(e)$.




\begin{theorem}\label{The L invariant}
     The $\sL$-invariant of the $(\varphi, \Gamma)$-module associated to the  $E$-valued point in the exceptional fiber with coordinates $(p^r, (1 + p)^{r + 1} - 1,\> a : b)$ is
    \[
        \sL = -\frac{(1+p)^{r+1}\log(1+p)}{p^r} \cdot \frac{a}{b} \in \mathbb{P}^1(E).
    \]
\end{theorem}
\begin{proof}
  We compute the $\sL$-invariant of the $(\varphi, \Gamma)$-module given by  $H^1(v^*)(e)$ using Definition~\ref{def of L inv}.
  For convenience of notation, let $\mathrm{Res}(f(T)) = \mathrm{res}(f(T)dt)$ for $f(T) \in \calR_{E}, \, \RU$. 
Write
\[
    H^1(v^*)(e) = \lambda_v \cdot \br{\alpha}_{r+1} + \mu_v \cdot \br{\beta}_{r+1},
\]
where, by \eqref{eq:10},
\[\arraycolsep=2pt
  \begin{array}{r}
    \lambda_v  =  \mathrm{Res}(t^rv^*(x - (\delta_{\U}(p)\varphi - 1)d )), \quad
    \mu_v  =  \mathrm{Res}(t^rv^*(y - (\delta_{\U}(\chi(\gamma))\gamma - 1)d )).
    \end{array}
\]

To compute these residues, we first note that the following square commutes
\[
    \begin{tikzcd}
      mD\otimes_{\qp}E \ar[r, "v^*"] \ar[d, "\mathrm{Res}(t^r\_)"] & 
                                                                          \calR_{E}(x^r\chi) \ar[d, "\mathrm{Res}({t^r} \_)"]\\
        m \otimes_{\qp} E \ar[r, "v"] & E,
    \end{tikzcd}
  \]
where we write $v$ again for the map $m \otimes E \xrightarrow{\br{\>\>}} m/m^2 \otimes E \xrightarrow{v} E$.
Indeed, given any element $f_1g_1 + f_2g_2 \in mD$, we have
\[\arraycolsep = 2pt
    \begin{array}{rcl}
      \mathrm{Res}(t^r v^*(f_1g_1 + f_2g_2))
      & = & \mathrm{Res}(t^r (v(\br{f_1}) \br{g_1} + v(\br{f_2}) \br{g_2})) \\
      & = & v(\br{f_1}) \mathrm{Res}(t^r\br{g_1}) + v(\br{f_2})\mathrm{Res}(t^r\br{g_2}) \\
      & = &   v(\mathrm{Res}(t^r g_1) f_1 +  \mathrm{Res}(t^rg_2)   f_2)    \\
      & = &   v(\mathrm{Res}(t^r(f_1g_1 + f_2g_2))). 
    \end{array}
\]

Now we prove that $\mathrm{Res}(t^r x) = 0$ and $\mathrm{Res}(t^r y) = 0$. 
We use the following formulas, which are similar to (\ref{eq:11}) and (\ref{eq:12}). 
For any $f(T) \in \RU$, we have
\begin{eqnarray}
    \mathrm{Res}(\varphi(f(T))) & = & \mathrm{Res}(f(T)), \label{eq:13} \\
    \mathrm{Res}(\gamma(f(T))) & = & \chi(\gamma)^{-1}\mathrm{Res}(f(T)). \label{eq:14}
\end{eqnarray}
\
Using the definition of $x$ and $y$ (see the discussion before Definition~$\ref{gen of H1mD}$), we get
\[
    (\delta_{\U}(\chi(\gamma))\gamma - 1)x = (\delta_{\U}(p)\varphi - 1)y = -S_1(1 + T),
\]
by Lemma~\ref{y and phi - 1}. Multiplying by $t^r$ and taking residues, by (\ref{eq:6}) and (\ref{eq:5}), we get
\[\arraycolsep=2pt
    \begin{array}{lcl}
    \mathrm{Res}((\delta_{\U}(\chi(\gamma))\chi(\gamma)^{-r}\gamma - 1)t^rx) & = & \mathrm{Res}(-t^rS_1(1 + T)) = 0, \\
    \mathrm{Res}((\delta_{\U}(p)p^{-r}\varphi - 1)t^ry) & = & \mathrm{Res}(-t^rS_1(1 + T)) = 0.
    \end{array}
  \]
Applying \eqref{eq:14} and \eqref{eq:13} to the terms on the left, we get $\mathrm{Res}(t^rx) = 0 = \mathrm{Res}(t^ry)$.

Thus, using the commutativity of the diagram above, we get
\[\arraycolsep=2pt
    \begin{array}{lcl}
    \lambda_v & = & v(\mathrm{Res}(t^r(x - (\delta_{\U}(p)\varphi - 1)d ))) \\
    & = & - v(\mathrm{Res}(t^r(\delta_{\U}(p)\varphi - 1)d )) \\
    & \stackrel{(\ref{eq:5})}{=} & - v(\mathrm{Res}((\delta_{\U}(p)p^{-r}\varphi - 1)t^rd)).
    \end{array}
\]
    Using equation (\ref{eq:13}), we get
\[\arraycolsep=2pt
    \begin{array}{lcl}
   \lambda_v & = & - v((\delta_{\U}(p)p^{-r} - 1)\mathrm{Res}(t^rd))\\
    & \stackrel{(\ref{def of taut char})}{=} & - v((S_1 - p^r)p^{-r}\mathrm{Res}(t^rd)) \\
    & = &  - v(f_1 p^{-r}\mathrm{Res}(t^rd)) \\
    & = & - p^{-r}v(f_1\mathrm{Res}(t^rd)).
    \end{array}
\]
The expression in the brackets belongs to $m$. Recall that $v : m \otimes E \to E$ is the composition of the maps ${\>}^-: m \to m/m^2$ and $v : m/m^2 \otimes E \to E$.
We therefore get
\[
    \lambda_v = -p^{-r} \br{\mathrm{Res}(t^rd)} \cdot v(\br{f_1}) =  -p^{-r} \br{\mathrm{Res}(t^rd)} \cdot a.
\]
Similarly,
\[\arraycolsep=2pt
    \begin{array}{rcl}
        \mu_v & = & v(\mathrm{Res}(t^r(y - (\delta_{\U}(\chi(\gamma))\gamma - 1)d))) \\

        & = & - v(\mathrm{Res}(t^r(\delta_{\U}(\chi(\gamma))\gamma - 1)d)) \\
        & \stackrel{(\ref{eq:6})}{=} & - v(\mathrm{Res}((\delta_{\U}(\chi(\gamma))\chi(\gamma)^{-r}\gamma - 1)t^rd)).
    \end{array}
\]
    Using equation (\ref{eq:14}), we get
\[
        \mu_v = - v((\delta_{\U}(\chi(\gamma))\chi(\gamma)^{-(r+1)} - 1)\mathrm{Res}(t^rd)).
\]
Since $\chi(\gamma) = \zetap^a (1+p)$, for some $a$, we get
\[\arraycolsep=2pt
        \begin{array}{rcl}
        \mu_v & \stackrel{(\ref{def of taut char})}{=} & - v((\zetap^{a(r+1)}(1 + S_2)\zetap^{-a(r+1)}(1+p)^{-(r+1)} - 1)\mathrm{Res}(t^rd))  \\
        & = & - v(f_2(1+p)^{-(r+1)}\mathrm{Res}(t^rd)) \\
        & = & - (1+p)^{-(r+1)}v(f_2\mathrm{Res}(t^rd)).
        \end{array}
\]
    The expression inside the brackets is an element of $m$. As above, we get
\[ 
        \mu_v = - (1+p)^{-(r+1)}\br{\mathrm{Res}(t^rd)} \cdot v(\br{f_2}) =  - (1+p)^{-(r+1)}\br{\mathrm{Res}(t^rd)} \cdot b.
\]
Both $\lambda_v$ and $\mu_v$ cannot be equal to $0$ simultaneously because $H^1(v^*)(e) \neq 0$.
So $\br{\mathrm{Res}(t^rd)} \not = 0$. By Definition~\ref{def of L inv} (or directly by \cite[Proposition 2.3.7]{Ben}),
the $\sL$-invariant of  $H^1(v^*)(e)$ is
\[
  -\log(\chi(\gamma)) \cdot \frac{\lambda_v}{\mu_v} = -\frac{(1+p)^{r+1}\log(1+p)}{p^r} \cdot \frac{a}{b}.
  \qedhere
\]
\end{proof}


\section{Proof of Theorem~\ref{any l invariant} and generalizations}
\label{Section proof of main theorem}

\subsection{{Proof of Theorem~\ref{any l invariant}}}
\label{proof of main theorem}

  Let $k \geq 3$ and let $r = k - 2$. Let $(k_n, a_n)$ for $n \geq 1$ be as in \eqref{kn an}. These quantities depend on
  $\sL \in {\mathbb P}^1(\br{\mathbb Q}_p)$.
  We prove that the sequence of crystalline representations $V^*_{k_n,a_n}$ 
  converges to the semi-stable representation $V^*_{k, \sL}$.

We recall some facts from Section~\ref{computing L}. Recall that the ordered pair of characters
$(\delta_{1, n}, \delta_{2, n})$ is  associated to the representation $V^*_{k_n, a_n}$
and the sequence of characters $\delta_{1, n}\delta_{2, n}^{-1} = \mu_{y_n^2} \chi^{k_n - 1}$, where 
$y_n$ is as in \eqref{yn},
converges to the exceptional character $x^r\chi$. Moreover, the character $\mu_{y_n^2}\chi^{k_n-1}$ corresponds to the following point of $\tU$
\begin{equation}\label{eq:15}
    (y_n^{2}, (1 + p)^{k_n - 1} - 1, \> y_n^2 - p^r : (1 + p)^{k_n-1} - (1 + p)^{k - 1}).
\end{equation}
We compute the limit of the above points in $\tU$ as $n \to \infty$ using the following lemmas.

\begin{lemma}\label{rate of the third coordinate}
    We have
    \[
      \lim_{n \to \infty}\frac{y_n^2 - p^r}{p^n(p-1)} =
     \begin{cases}
         \sL p^r, & \text{if $\sL \not = \infty$} \\
        2p^r/(p-1), & \text{if $\sL = \infty$}.
      \end{cases}
    \]
\end{lemma}

\begin{proof}
Assume $\sL \neq \infty$. We have
    \[\arraycolsep=2pt
        \def\arraystretch{1.2}
        \begin{array}{rcl}
             y_n^2 - p^r & = & \left(\dfrac{a_n + \sqrt{a_n^2 - 4p^{k_n-1}}}{2}\right)^2 - p^r 
             =   a_n^2 \left(\dfrac{1 + \sqrt{1 - 4p^{k_n-1}a_n^{-2}}}{2}\right)^2 - p^r                   
        \end{array}
    \]
    For large $n$, the valuation of $a_n = p^{r/2} + \sL p^{n + r/2}(p - 1)/2$ is equal to $r/2$. Therefore for large $n$, the expression inside
    the second radical sign is an element of $1 + p^{1 + p^n(p - 1)}\co_{E}$, where $\co_{E}$ is the ring of integers of $E$. Since taking square
    roots is an automorphism of this group, we see that $\sqrt{1 - 4p^{k_n - 1}a_n^{-2}} \in 1 + p^{1 + p^n(p - 1)}\co_{E}$. Therefore we can write
    \[
        \frac{1 + \sqrt{1 - 4p^{k_n - 1}a_n^{-2}}}{2} = 1 + up^{1 + p^n(p - 1)},
    \]
    for some $u \in \co_{E}$. Substituting this in the previous equation, we get
    \[\arraycolsep=2pt
        \def\arraystretch{1.2}
        \begin{array}{rcl}
            y_n^2 - p^r 
            & = & (p^r + \sL p^{n + r}(p - 1) + \sL^2p^{2n + r}(p - 1)^2/4)(1 + 2up^{1 + p^n(p - 1)} + u^2p^{2 + 2p^n(p - 1)}) - p^r \\
            & = & p^r + \sL p^{n + r}(p - 1) + p^{2n}(u') - p^r,
        \end{array}
    \]
    for some $u'$ whose valuation is bounded below. So
    \begin{equation*}
            \lim_{n \to \infty} \frac{y_n^2 - p^r}{p^n(p-1)} = \sL p^r. 
          \end{equation*}
    The limit in the case $\sL = \infty$ is computed similarly.       
  \end{proof}
  
\begin{lemma}\label{rate of the fourth coordinate}
    We have
    \[
      \lim_{n \to \infty} \frac{(1 + p)^{k_n-1} - (1 + p)^{k - 1}}{p^n(p-1)} =
      \begin{cases}
          (1 + p)^{k - 1}\log{(1 + p)}, & \text{if $\sL \not = \infty$} \\
          0, & \text{if $\sL = \infty$}.
      \end{cases}
    \]
\end{lemma}

\begin{proof}
    Assume $\sL \neq \infty$. Since
    \[
        (1 + p)^{k_n - 1} - (1 + p)^{k - 1} = (1 + p)^{k - 1}[(1 + p)^{p^n(p - 1)} - 1],
    \]
    we see that
    \begin{eqnarray*}
            \lim_{n \to \infty} \frac{(1 + p)^{k_n - 1 } - (1 + p)^{k - 1}}{p^n} & = & (1 + p)^{k - 1}\lim_{n \to \infty}\frac{(1 + p)^{p^n(p - 1)} - 1}{p^n}  \\
            & = & (1 + p)^{k - 1}\lim_{n \to \infty} \frac{\{1 + [(1 + p)^{(p - 1)} - 1]\}^{p^n} - 1}{p^n}  \\
            & \stackrel{(\ref{eq:4})}{=} & (1 + p)^{k - 1}\log{(1 + p)^{p - 1}}. \qquad \qquad \qquad \qquad \qquad \qquad \quad 
    \end{eqnarray*}
    The limit in the case $\sL = \infty$ is proved similarly.
  \end{proof}
  
Write the sequence (\ref{eq:15}) as follows
\[
    \left(y_n^{2}, (1 + p)^{k_n - 1} - 1, \> \frac{y_n^2 - p^r}{p^n(p-1)} : \frac{(1 + p)^{k_n - 1} - (1 + p)^{k - 1}}{p^n(p-1)}\right).
\]
Using Lemma $\ref{rate of the third coordinate}$ and Lemma $\ref{rate of the fourth coordinate}$, we see that the sequence above converges to the point
\[
  \begin{cases}
    (p^r, (1 + p)^{k - 1} - 1, \> \sL p^r : (1 + p)^{k - 1}\log{(1 + p)}),  & \text{if $\sL \not = \infty$} \\
    (p^r, (1 + p)^{k - 1} - 1, \> 1 : 0), &  \text{if $\sL = \infty$}.
  \end{cases}
\]

Assume $\sL \neq \infty$. By Theorem \ref{The L invariant}, the $\sL$-invariant of the $(\varphi, \Gamma)$-module associated to this limit point
is $-\sL$.  By \cite[Th\'eor\`eme 0.5 (i)]{Col}, this $(\varphi, \Gamma)$-module is also \'etale
  since $(\mu_{p^{r/2}}, \mu_{1/p^{r/2}}\chi^{1 - k},-\sL) \in {\mathscr S}_* \setminus {\mathscr S}_*^{\mathrm{ncl}}$ in the notation of \cite{Col}
  since $r/2 - r/2 = 0$, $r/2>0$ and $r/2 \not > k-1$. The corresponding Galois representation is $V(\mu_{p^{r/2}}, \mu_{1/p^{r/2}}\chi^{1 - k}, -\sL)$
  in the notation of \cite{Col}. Comparing the filtered $(\varphi, N)$-module associated to $V_{k, \sL}^*$ given in the Introduction
  with the one associated to $V(\mu_{p^{r/2}}, \mu_{1/p^{r/2}}\chi^{1 - k}, -\sL)$ using \cite[Section 4.6]{Col} (more specifically,
  \cite[Proposition 4.18]{Col} with $a = 1-k$, $b = 0$, $\alpha = \mu_{p^{k/2}}$, and $\sL$ replaced by $-\sL$), we see that
  $V(\mu_{p^{r/2}}, \mu_{1/p^{r/2}}\chi^{1 - k}, -\sL) \simeq V_{k, \sL}^*$. Thus the sequence of crystalline representations 
  $V^*_{k_n,a_n}$ converges to the semi-stable representation $V^*_{k,\sL}$ for $\sL \in E$. 

  \par Now assume we are in the $\sL = \infty$ case.
By Theorem~\ref{The L invariant}, 
the $\sL$-invariant associated to the above limit point is $\infty$. The corresponding Galois representation
is $V(\mu_{p^{r/2}}, \mu_{1/p^{r/2}}\chi^{1 - k}, \infty)$  in the notation of \cite{Col}.
Comparing the filtered $\varphi$-module associated to
$V_{k,\infty}^*$ in the Introduction with the one associated to $V(\mu_{p^{r/2}}, \mu_{1/p^{r/2}}\chi^{1 - k}, \infty)$
in \cite[Section 4.5]{Col} (more precisely, take $a = 1-k$, $b = 0$, $\beta = \mu_{p^{r/2}}$ and $\alpha = \mu_{p^{k/2}}$ in
  the second isomorphism in \cite[Proposition 4.13]{Col}), we see that
  $V(\mu_{p^{r/2}}, \mu_{1/p^{r/2}}\chi^{1 - k}, \infty) \simeq V_{k,\infty}^*$. (Alternatively, the corresponding 
Galois representation is $V(\mu_{p^{r/2}}, \mu_{1/p^{r/2}}\chi^{1 - k})$ in the notation of \cite{Ber} and, by \cite[Proposition 3.1]{Ber}, this last representation
is isomorphic to the crystalline representation $V_{k,a_p}^*$ with $a_p = p^{k/2} + p^{k/2-1}$, which as mentioned in the Introduction, is isomorphic to $V^*_{k, \infty}$.)
  Thus, the sequence of crystalline representations $V^*_{k_n, a_n}$ again converges to the (crystalline) representation $V^*_{k, \infty}$. 

\subsection{$\sL$-invariants as logarithmic derivatives}
  \label{Comparing with Greenberg-Stevens}

  In this subsection, all $\sL$-invariants are finite.   

  We prove formula \eqref{log-derivative-intro} from the Introduction showing that $\sL$ is twice the logarithmic derivative of $a_p$. 
More precisely, we prove that if $a_p: \zp \to E$ is a differentiable function of $l$ with $a_p(k) = p^{r/2}$,
for $r = k-2$ and $k \geq 3$, then the crystalline representations $V_{l, a_p(l)}^*$ converge in $\til{\sT}_2$ to the semi-stable representation $V_{k, \sL}^*$ with $$\sL = 2a_p(k)^{-1}a_p'(k)$$ as $l$ tends to $k$ in the $p$-adic topology through integers $l \equiv k \mod (p - 1)$ with $l \neq k$.  The condition $l \equiv k \mod{(p - 1)}$ is necessary because $V_{l, a_p(l)}^*$ converges to $V_{k, \sL}^*$ implies that (the tame part of) $\det V_{l, a_p(l)}^* = \chi^{1 - l}$ converges to (the tame part of)
$\det V_{k, \sL}^* = \chi^{1 - k}$.

This formula is a variant of a classical formula due to \cite[Theorem 3.18]{GS}, \cite[Theorem B]{Ste}, \cite[Theorem 4]{BDI}, \cite[Th\'eor\`eme 0.5, Corollaire 0.7]{Col10}, \cite[Theorem 2]{Ben10} and others (see Remark \ref{Different G-S}).
Our proof of the formula seems new. It uses some elementary $p$-adic analysis (see the two lemmas below) and
Theorem~\ref{The L invariant}, which in turn uses some geometry (the blow-up space $\til{\sT}_2$) and some algebra
(the interpretation of this space in terms of trianguline $(\varphi, \Gamma)$-modules over the Robba ring).

Recall that for integer $l \geq 2$, the $(\varphi, \Gamma)$-module $\Dr(V_{l, a_p(l)}^*)$ is an extension of $\calR_E(\mu_{1/y(l)}\chi^{1 - l})$ by $\calR_E(\mu_{y(l)})$, where
    \[
        y(l) = \frac{a_p(l) + \sqrt{a_p(l)^2 - 4p^{l - 1}}}{2}.
    \]
    Let $\delta_1(l) = \mu_{y(l)}$ and $\delta_2(l) = \mu_{1/y(l)}\chi^{1 - l}$.
    Since $l \equiv k \mod (p-1)$, the characters $\delta_1(l)\delta_2(l)^{-1}$ converge to the exceptional character
    $x^{r}\chi$ as $l \to k$. Therefore the characters $\delta_1(l)\delta_2(l)^{-1}$ eventually belong to $\tU$.
    Moreover, $\delta_1(l)\delta_2(l)^{-1}$ corresponds to the following point of $\tU$
    \begin{equation}\label{points for arbitrary a_p}
        (y(l)^2, (1 + p)^{l - 1} - 1, \>  y(l)^2 - p^r : (1 + p)^{l - 1} - (1 + p)^{k - 1}).
    \end{equation}
    We compute the limit of these points in $\tU$ as $l \to k$ using the following two lemmas.
    
    \begin{lemma}\label{Derivative of y(l)^2}
        We have
        \[
            \lim_{l \to k} \frac{y(l)^2 - p^r}{l - k} = 2a_p(k)a_p'(k).
        \]
    \end{lemma}
    \begin{proof}
      The statement generalizes that of Lemma~\ref{rate of the third coordinate} in the case $\sL \neq \infty$ and
      we give a slightly different proof. 
      Note that
        \begin{eqnarray*}
            y(l)^2 - p^r & = & \frac{2a_p(l)^2 + 2a_p(l)\sqrt{a_p(l)^2 - 4p^{l - 1}} -4p^{l - 1}}{4} - p^r \\
            & = & \frac{a_p(l)^2 - p^r}{2} + \frac{a_p(l)^2\sqrt{1 - 4a_p(l)^{-2}p^{l - 1}} - p^r}{2} - p^{l - 1}.
        \end{eqnarray*}
        Therefore
        \begin{eqnarray*}
          \frac{y(l)^2 - p^r}{l - k}
            & = & \frac{a_p(l)^2 - p^r}{2(l - k)} + \frac{a_p(l)^2 - p^r}{2(l - k)}\sqrt{1 - 4a_p(l)^{-2}p^{l - 1}} + p^r\frac{\sqrt{1 - 4a_p(l)^{-2}p^{l - 1}} - 1}{2(l - k)} - \frac{p^{l - 1}}{l - k}.
        \end{eqnarray*}
        
By the definition of the derivative, the first and the second summands in the equation above converge to $a_p(k)a_p'(k)$ as $l \to k$. The last summand clearly converges to $0$. The third summand also converges to $0$. Indeed, 
as $l \to k$, the valuation of $a_p(l)$ becomes $r/2$. Therefore for such $l \geq k$ we can write
        \[
            \sqrt{1 - 4a_p(l)^{-2}p^{l - 1}} = \sum_{i = 0}^{\infty} {1/2 \choose i}(-4a_p(l)^{-2}p^{l - 1})^i.
        \]
        Using the fact that $\frac{p^{l - 1}}{l - k}$ converges to $0$ as $l \to k$, we see that 
        \[
            \lim_{l \to k} \frac{\sqrt{1 - 4a_p(l)^{-2}p^{l - 1}} - 1}{2(l - k)} = 0.
        \]

Putting everything together, we see that
        \[
            \lim_{l \to k} \frac{y(l)^2 - p^r}{l - k} = 2a_p(k)a_p'(k).  \qedhere
        \]
    \end{proof}
    
    \begin{lemma}\label{Derivative of (1 + p)^{l - 1}}
        We have
        \[
            \lim_{l \to k} \frac{(1 + p)^{l - 1} - (1 + p)^{k - 1}}{l - k} = (1 + p)^{k - 1}\log(1 + p).
        \]
    \end{lemma}
    \begin{proof}
      This is the same as Lemma~\ref{rate of the fourth coordinate} in the case $\sL \neq \infty$,
      so we give a slightly different proof. Write
        \begin{eqnarray*}
            \frac{(1 + p)^{l - 1} - (1 + p)^{k - 1}}{l - k} & = & (1 + p)^{k - 1}\frac{[(1 + p)^{l - k} - 1]}{l - k} \\
            & = & (1 + p)^{k - 1}\left[\sum_{i = 1}^{\infty} \frac{1}{i}{l - k - 1 \choose i - 1}p^i\right].
        \end{eqnarray*}
Since ${- 1 \choose i - 1} = (-1)^{i-1}$, taking the limit as $l \to k$ we obtain
        \[
            \lim_{l \to k} \frac{(1 + p)^{l - 1} - (1 + p)^{k - 1}}{l - k} = (1 + p)^{k - 1}\log(1 + p). \qedhere
        \]
    \end{proof}
    
    Now rewrite the point \eqref{points for arbitrary a_p} in the blow-up as
    \[
        \left(y(l)^2, (1 + p)^{l - 1} - 1, \> \frac{y(l)^2 - p^r}{l - k} : \frac{(1 + p)^{l - 1} - (1 + p)^{k - 1}}{l - k}\right).
    \]
    Using Lemma \ref{Derivative of y(l)^2} and Lemma \ref{Derivative of (1 + p)^{l - 1}}, we see  that as $l \to k$, the
    points above converge to
    \[
        (p^r, (1 + p)^{k - 1} - 1, \> 2a_p(k)a_p'(k) : (1 + p)^{k - 1}\log(1 + p)).
    \]
    By Theorem \ref{The L invariant}, we see that the $\sL$-invariant of the $(\varphi, \Gamma)$-module associated to this limit point is $-2a_p(k)^{-1}a_p'(k)$. Working as at the end of Section~\ref{proof of main theorem}, we see
    that the corresponding Galois representation is $V(\mu_{p^{r/2}}, \mu_{1/p^{r/2}}\chi^{1 - k}, -\sL)$ with $\sL = 2a_p(k)^{-1}a_p'(k)$, so the crystalline representations $V_{l, a_p(l)}^*$ converge to the semi-stable representation
$V_{k, \sL}^*$ with $\sL = 2a_p(k)^{-1}a_p'(k)$.

\begin{Remark}
  \label{Smoothening the rank one subobject}
    We thank one of the referees for pointing out the following simplification to the computation of the limits in Lemmas~\ref{rate of the third coordinate} and \ref{Derivative of y(l)^2}.
    Assume as above that $a_p : \zp \to E$ is a differentiable function of $l$ with $a_p(k) = p^{r/2}$ with $r = k - 2$ and $k \geq 3$. If we replace the
    crystalline representations $V_{l, a_p(l)}^*$ by $V_{l, a_p(l) + p^{l - 1}/a_p(l)}^*$, then the limit computation in Lemma \ref{Derivative of y(l)^2} becomes easier.
    Indeed, the $y(l)$ for the crystalline representation $V_{l, a_p(l) + p^{l - 1}/a_p(l)}^*$ is equal to $a_p(l)$.
    Therefore, for $l$ close to $k$ and $l \equiv k \mod (p - 1)$, the point in $\til{\sT}$ associated to $V_{l, a_p(l) + p^{l - 1}/a_p(l)}^*$ belongs to $\tU$
    and is given by
    \[
        (a_p(l)^2, (1 + p)^{l - 1} - 1, a_p(l)^2 - p^r : (1 + p)^{l - 1} - (1 + p)^{k - 1}).
    \]
    To evaluate the limit of these points in the third coordinate, we have to evaluate the limit 
    \[
        \lim_{\substack{l \to k \\ l \equiv k \!\!\!\!\mod p - 1}}\frac{a_p(l)^2 - p^{r}}{l - k}.
    \]
    But, this limit is just the derivative of $a_p(l)^2$ at $l = k$ and so is immediately equal to  $2a_p(k)a_p'(k)$.
    Similarly, for the limits in Lemma~\ref{rate of the third coordinate}. In principle, since the limit of the above sequence
    is the same as that of the original sequence, the crystalline representations $V_{l, a_p(l) + p^{l - 1}/a_p(l)}$ can also be used to compute
    the reduction of the semi-stable representation $V_{k,\sL}$. Note that $\br{V}_{l, a_p(l) + p^{l - 1}/a_p(l)} \simeq \br{V}_{l, a_p(l)}$
    for $l$ close to $k$ by \cite[Theorem A]{Ber}.

    Also, Lemmas~\ref{rate of the fourth coordinate} and \ref{Derivative of (1 + p)^{l - 1}}
    follow immediately by differentiating the function $f:\zp \to \qp$ defined by $f(x) = (1+p)^{x-1}$ at $x = k$.

\end{Remark}

\begin{Remark}[Relation to work of Greenberg-Stevens,...]
  \label{Different G-S}
  Our formula \eqref{log-derivative-intro} is slightly different from the classical formula due to Greenberg, Stevens, Colmez, Bertolini, Darmon, Iovita, Benois and others (see \cite{GS, Ste, Col10, BDI, Ben10},...
          ).
        We thank D. Benois for pointing this out.
          
        Let us explain the difference.
        In the classical setting, one starts with a newform $f$ of weight $k \geq 2$ for the subgroup $\Gamma_0(Np)$, where $(N, p) = 1$ (for simplicity, we assume that the nebentypus at $N$ is trivial) with $U_p$-eigenvalue $\alpha_p(f) = p^{\frac{k - 2}{2}}$ (as opposed to $-p^{\frac{k - 2}{2}}$ for simplicity). Then one takes the Hida family $F$ if $f$ is ordinary (equivalently k = 2), or the Coleman family $F$ if $f$ has positive slope (equivalently $k > 2$), passing through $f$ with $U_p$-eigenvalue $\alpha_p$, which is an analytic function of the weight. Note that $\alpha_p(k) = \alpha_p(f)$.
        The classical formula states that the $\mathrm{\sL}$-invariant of the form $f$ is
        \[
            \sL(f) = -2\frac{\alpha_p'(k)}{\alpha_p(k)}.
        \]
        Incidentally, the above authors use various definitions of the $\sL$-invariant of the form $f$ but these are all equal (see \cite{Col05} for a survey comparing these alternative definitions).
        
        For $l \equiv k \mod (p - 1)$, the nebentypus of the weight $l$ member $F_l$ of the family $F$ is trivial.
        Since forms living in a Hida or Coleman family have the same slope and since the slope of the $U_p$-eigenvalue of a $\Gamma_0(Np)$-newform of weight $l$ is equal to $(l - 2)/2$, we see that $f = F_k$ is the unique classical member in the family $F$ of weight $l \equiv k \mod (p - 1)$ that is $p$-new. In fact, any classical member $F_l$ with $l \equiv k \mod (p - 1)$ and $l \neq k$ arises as a $p$-stabilization of a form $\til{F}_l$ that is only $N$-new. The $U_p$-eigenvalue $\alpha_p(l)$ of $F_l$ is a root of
        \[
            x^2 - a_p(l)x + p^{l - 1},
        \]
        where $a_p(l)$ is the $T_p$-eigenvalue of the form $\til{F}_l$. The local Galois representation $$\rho_{F_l}\vert_{\mathrm{Gal}(\brqp/\qp)} \simeq \rho_{\til{F}_l}\vert_{\mathrm{Gal}(\brqp/\qp)}$$ is isomorphic to the crystalline representation $V_{l, a_p(l)}$. Moreover, $\rho_{f}\vert_{\mathrm{Gal}(\brqp/\qp)}$ is isomorphic to the semi-stable representation $V_{k, \sL}$ with $\sL = -\sL(f)$. The change in sign is because the $\sL$-invariant in the filtration on $D_{k, \sL}$ given in the Introduction
        is the negative of the one in the filtration on the filtered module in \cite{Maz}. Therefore in the classical setup, 
        \[
            \sL = -\sL(f) = 2\frac{\alpha_p'(k)}{\alpha_p(k)}.
          \]
          
        In our setup, we have a smooth function $a_p : \zp \to E$ such that $a_p(k) = p^{\frac{k - 2}{2}}$. We prove that the sequence of crystalline representations $V_{l, a_p(l)}^*$ converge to the semi-stable representation $V_{k, \sL}^*$ with
        \[
            \sL = 2\frac{a_p'(k)}{a_p(k)}.
        \]

        Note that there does not seem to be a common framework within which one may compare these two formulas. Since the
        functions $a_p(l)$ and $\alpha_p(l)$ associated to the Hida or Coleman families above are related by the formula
        \[
            a_p(l) = \alpha_p(l) + \frac{p^{l - 1}}{\alpha_p(l)},
        \]
        the $a_p(l)$ above cannot be interpolated by a smooth function $a_p : \zp \to E$. Indeed, $p^{l - 1}$ is not even defined on the whole of $\zp$. 

        However, we can obtain the classical formula above using the trick in 
        Remark~\ref{Smoothening the rank one subobject}. 
        Indeed, suppose $\alpha_p : \zp \to E$ is a differentiable function of $l$ with $\alpha_p(k) = p^{\frac{k - 2}{2}}$ with $k \geq 3$. Consider the crystalline representations $V_{l, \alpha_p(l) + \frac{p^{l - 1}}{\alpha_p(l)}}^*$ for integer $l \equiv k \mod (p - 1)$ with $l \neq k$. The $y(l)$ for these representations is $\alpha_p(l)$. Therefore, for $l$ close to $k$ and $l \equiv k \mod (p - 1)$, the point in $\til{\sT}$ associated to $V_{l, \alpha_p(l) + \frac{p^{l - 1}}{\alpha_p(l)}}^*$ belongs to $\tU$ and is given by
    \[
        (\alpha_p(l)^2, (1 + p)^{l - 1} - 1, \> \alpha_p(l)^2 - p^r : (1 + p)^{l - 1} - (1 + p)^{k - 1}).
    \]
    Arguing as in Remark~\ref{Smoothening the rank one subobject} (with $a_p(l)$ there replaced by $\alpha_p(l)$),
    we see that  as $l$ tends to $k$ the limit of the above sequence of points is 
    \[
        (p^r, (1 + p)^{k - 1} - 1, \>2\alpha_p(k)\alpha_p'(k) : (1 + p)^{k - 1}\log(1 + p)).
    \]
    Using Theorem \ref{The L invariant}, we see that the $\sL$-invariant of the $(\varphi, \Gamma)$-module associated to the limit point is $-\sL$ with $$\sL = 2\frac{\alpha_p'(k)}{\alpha_p(k)}.$$ In the notation of \cite{Col}, the corresponding Galois representation is $V(\mu_{p^{r/2}}, \mu_{1/p^{r/2}}, -\sL)$, which is isomorphic to $V_{k, \sL}^*$ as explained at the end of Section~\ref{proof of main theorem}.

    \end{Remark}

\section{Proof of Theorem \ref{thm from conj}}
  \label{harmonic sums are here}

  In this section, we use Theorem \ref{any l invariant} along with local constancy for the reduction 
  and the zig-zag conjecture to compute the
  reductions of semi-stable representations $V_{k,\sL}$ with weights $k$ in the range $[3, p + 1]$ for $p$ odd.

  Consider the crystalline representation $V_{k_n, a_n}$, with $(k_n, a_n)$ as in \eqref{kn an}. Since
  $v_p(a_n) 
  =  \frac{r}{2}$ for $r = k-2$,
  for $n$ sufficiently large, the weight $k_n$ 
  is an {\it exceptional} weight with respect to the half-integral slope
  $0 < v = \frac{r}{2} \leq \frac{p-1}{2}$
  in the sense of \cite[Section 1]{Gha}, namely $k_n \equiv 2v+2 \mod (p-1)$, for $n$ large.
  Therefore, we can apply Conjecture~\ref{zig-zag-conj}
  to compute the reduction of $V_{k_n, a_n}$ for large $n$. The zig-zag conjecture specifies the reduction in terms of an alternating sequence of
  irreducible and reducible mod $p$ representations depending on the relative size of a rational parameter
  $\tau$ with respect to (certain integer shifts of) another integer parameter $t$.

  For $V_{k_n, a_n}$, the formula for $\tau$ (cf. \cite[(1.1)]{Gha}), which we call $\tau_n$, is
\[
    \tau_n = v_p\left(\dfrac{a_n^2 - {k_n - 2 - v_{-} \choose v_{+}}{k_n - 2 - v_{+}\choose v_{-}}p^r}{pa_n}\right),
\]
where $v_{-}$ and $v_{+}$ are the largest and the smallest integers, respectively, such that $v_{-} < v < v_{+}$.
Let us simplify this. We first treat the case $\sL \neq \infty$.  We have
\[
    \tau_n = v_{p}\left(\dfrac{p^{r}(1 + p^{n}(p - 1)\sL/2)^2 - {k + p^n(p - 1) - 2 - v_{-} \choose v_{+}}{k + p^n(p - 1) - 2 - v_{+} \choose v_{-}}p^r}{p^{1+r/2}(1 + p^{n}(p - 1)\sL/2)}\right).
\]
For large $n$, we have $v_{-} + v_{+} = k - 2$, so
    \begin{eqnarray}
    {k + p^n(p - 1) - 2 - v_{-} \choose v_{+}} & = & \frac{(1 + p^n(p - 1))(2 + p^n(p - 1))\cdots(v_{+} + p^n(p - 1))}{v_{+}!} \nonumber\\
                                              & = & 1 + p^n(p - 1) H_+ + \text{ terms involving }p^{2n}, \nonumber 
    \end{eqnarray}
for the harmonic sum\footnote{Thus, the provenance of the harmonic sums
  that are prevalent in all the computation of the reductions of semi-stable representations in the literature can now be traced back to the $p$-adic
  expansions of the
  binomial coefficients appearing in the zig-zag conjecture.} $H_{+} = \sum_{i = 1}^{v_{+}}\frac{1}{i}$.   Similarly,
\[
    {k + p^n(p - 1) - 2 - v_{+} \choose v_{-}} = 1 + p^n(p - 1) H_{-}+ \text{ terms involving }p^{2n},
  \]
  for $H_{-} = \sum_{i = 1}^{v_{-}}\frac{1}{i}$ (if $v_{-} \geq 1$; $H_{-} = 0$ if $v_{-} = 0$). 
Substituting, we get
\begin{eqnarray}
    \tau_n 
    & = & r + v_p((1 + \sL p^n(p - 1)/2)^2 - [1 + p^n(p - 1)(H_{-} + H_{+}) + \text{ terms involving }p^{2n}]) \nonumber \\
    & & \qquad -v_p(p^{1 + r/2}(1 + \sL p^n(p - 1)/2)) \nonumber \\
    & = & r + v_p(\sL p^n(p - 1) - (H_{-} + H_{+})p^n(p - 1) + \text{terms involving }p^{2n})  \nonumber \\
    & & \qquad - (1 + r/2) - v_p(1 + \sL p^n(p - 1)/2). \nonumber
\end{eqnarray}
Thus, for large $n$,
\[
    \tau_n = r/2 - 1 + n + v_p(\sL - H_{-} - H_{+}),
  \]
  if $\sL \neq H_{-} + H_{+}$ (and $\tau_n \geq r/2-1 + 2n + c$ for some $c \in {\mathbb Q}$ independent of $n$
  if  $\sL = H_{-} + H_{+}$).

As mentioned above, the zig-zag conjecture  involves another parameter $t$ which for $V_{k_n,a_n}$ we call $t_n$. We have 
\[
    t_n = v_p(k_n - 2 - r) = n.
  \]
  
  These formulas for $\tau_n$ and $t_n$ show that for large $n$ the parameter $\tau_n - t_n$ lies, independently of $n$, in 
  one of the intervals (which may possibly be a point) that appear in
  the statement of \cite[Conjecture 1.1]{Gha}. This interval is determined by the size of $\nu = v_p(\sL - H_{{-}} - H_{{+}})$
  (and is the rightmost one when $\nu = \infty$).
  The conjecture accordingly specifies the exact shape of the reductions of the crystalline representations $V_{k_n, a_n}$  for large $n$ in terms of the size of $\nu$.
  Since taking duals commutes with taking reduction, we obtain the reductions of the $V^*_{k_n, a_n}$ for large $n$
  in terms of the size of $\nu$. 
  Using
  Theorem~\ref{any l invariant} and 
  local constancy for the reduction,
  we get the reduction of the limiting semi-stable
  representation $V^*_{k,\sL}$ in terms of the size of $\nu$.
  Finally, taking duals again, we obtain Theorem \ref{thm from conj} for $\sL$ finite.

  Now assume we are in the case $\sL = \infty$. Then a computation similar to the one above shows that $\tau_n = r/2-1+n$ and $t_n = n^2$, so that
  for large $n$, we have $\tau_n < t_n$ (and so $\tau_n - t_n$ is in the leftmost interval appearing in the conjecture).
  Using \cite[Conjecture 1.1]{Gha} and Theorem~\ref{any l invariant} again, we see that
  $\br{V}_{k, \infty} \sim \mathrm{ind}(\omega_2^{k - 1})$, at least on the inertia subgroup $I_{\qp}$.     
  Thus Theorem~\ref{thm from conj} also holds for $\sL = \infty$ as $\nu = -\infty$ (but as remarked after the theorem,
  the result in this case is classical by \cite{Edi} and even holds without assuming zig-zag!). 


\section{Bounded Hodge-Tate weights}
 \label{section further implication}

 In this section, we use the techniques of this paper to give a proof of the fact that the limit of a sequence of two-dimensional (irreducible) crystalline representations $V_n$ for $n \geq 1$ with Hodge-Tate weights in an interval $[a, b]$ such that the difference of the Hodge-Tate weights is at least two infinitely often is also (irreducible) crystalline with Hodge-Tate weights in the interval $[a, b]$.
However, note that Berger \cite[Th\'eor\`eme 1]{Ber04} has proved more generally 
 that the limit of subquotients of crystalline representations with bounded Hodge-Tate weights belonging to $[a, b]$ is also crystalline with Hodge-Tate weights in $[a, b]$.

 Twisting by a fixed power of the cyclotomic character, we may assume that infinitely often the Hodge-Tate weights of $V_n$ are $(0, k_n-1)$ with $k_n \geq 3$.
 This sequence of integers $k_n$ is not to be confused with the one defined in \eqref{kn an}.
Then 
there exists a sequence of unramified characters $\mu_n$ for $n \geq 1$ such that $V_n  \simeq V_{k_n, a_{n}} \otimes \mu_n$,
for some $a_n$ with  
$v_p(a_n) > 0$.
For each $n \geq 1$, consider the ordered pair of characters $(\delta_{1, n}, \delta_{2, n})$ associated with 
$V_n^*$ under triangulation. 
We have $\delta_{1, n} = \mu_{y_n}\cdot\mu_n^{-1}$ and $\delta_{2, n} = \mu_{1/y_n}\chi^{1 - k_n}\cdot\mu_n^{-1}$, where $y_n$ is as in \eqref{yn}. 
The assumption that the sequence $V_n$ converges means that the associated sequence of points in $\til{\sT}_2$ converge to a point in $\til{\sT}_2$. Say that the sequence $(\delta_{1, n}, \delta_{2, n})$ converges to $(\delta_1, \delta_2)$ in $(\sT\times\sT) \setminus F'$.

First assume that $\delta_1\delta_2^{-1}$ is not of the form $x^r\chi$ for $r \geq 0$. Then, by convention, the $\sL$-invariant associated to the limit point is $\infty$.
Using the convergence of (the unramified parts of) $\delta_{1, n}\delta_{2, n}$ and $\delta_{1, n}\delta_{2, n}^{-1}$, we see that $\mu_n^{-2}$ converges to $\mu_{\alpha^2}$ for some $\alpha \in \brqp^*$ and $\mu_{y_n^2}$ converges to $\mu_{y^2}$ for some $y \in \brqp^*$. Therefore $\delta_{1, n} = \mu_{y_n}\cdot\mu_n^{-1}$ converges to $\delta_1 = \mu_{\pm \alpha y}$. Similarly, $\delta_{2, n}$ converges to $\delta_2 = \mu_{\pm \alpha/y} \chi^{1 - k}$ for some $k$. So the sequence $V_n^*$ converges to $V(\mu_{\pm \alpha y}, \mu_{\pm \alpha / y}\chi^{1 - k}, \infty)$, for some $k$, in the notation of \cite{Col}. Now $k_n$ is a bounded sequence of integers converging to $k$. Therefore $k_n = k$ for large $n$. In particular, $k \geq 3$.
But $V(\mu_{\pm\alpha y}, \mu_{\pm\alpha/y}\chi^{1 - k}, \infty)$ is equal to the crystalline representation $V_{k, y + p^{k -1}/y}^* \otimes \mu_{\pm\alpha}$.
Taking duals, we see that the sequence $V_n$ converges to the crystalline representation $V_{k, y + p^{k - 1}/y} \otimes \mu_{\pm\alpha}^{-1}$. Untwisting by the fixed power of the cyclotomic character, the original sequence $V_n$ also converges to a crystalline representation with Hodge-Tate weights in $[a, b]$.
    

Now assume that $\delta_1\delta_2^{-1} = x^r\chi$, for some $r \geq 0$. This means that $\delta_{1, n}\delta_{2, n}^{-1} = \mu_{y_n^2}\chi^{k_n - 1}$ converges to $x^r\chi$ in the neighborhood $\U$ of $x^r\chi$ in $\sT$. Let $k = r+2$.
Since $k_n$ is a bounded sequence of integers, the convergence implies that
$k_n = k$ for all but finitely many $n$. In particular, we have $k \geq 3$.
The point corresponding to the character  $\mu_{y_n^2}\chi^{k - 1}$ in the blow-up $\tU$ of $\U$ is 
  \[
    (y_n^2,\> (1 + p)^{k - 1} - 1, \> y_n^2 - p^r : 0) = (y_n^2,\> (1 + p)^{k - 1} - 1, \> 1 : 0),
  \]
at least if $y_n^2 \neq p^r$ (if $y_n^2 = p^r$, then it is also the last point since this is the only crystalline point in the fiber above $x^r\chi$). 
 Since $y_n^2$ converges to $p^r$, this sequence converges to the point $(p^r,\> (1 + p)^{k - 1} - 1, 1 : 0)$ in $\tU$. 
 By Theorem \ref{The L invariant}, we see that the $\sL$-invariant associated with this limit point is $\infty$. So the limit of the sequence $V_n^*$ is
 $V(\delta_1, \delta_2, \infty)$ in the notation of \cite{Col}. 
  
  Now $\delta_1\delta_2^{-1} = x^r\chi = \mu_{p^{r/2}}\cdot (\mu_{1/p^{r/2}}\chi^{1 - k})^{-1}$ implies that there is a character $\delta$ such that $\delta_1 = \delta \cdot \mu_{p^{r/2}}$ and $\delta_2 = \delta \cdot \mu_{1/p^{r/2}}\chi^{1 - k}$. 
  Using \cite[Proposition 4.4 (i)]{Col}, we see that $(\delta_1\delta_2)(p)$ is a unit. Hence $\delta(p)$ is a unit. Consequently, $V(\delta_1, \delta_2, \infty) \simeq V(\mu_{p^{r/2}}, \mu_{1/p^{r/2}}\chi^{1 - k}, \infty) \otimes \delta$. Therefore the limit of the crystalline representations $V^*_n$ is the crystalline representation $V_{k,\infty}^* \otimes \delta$. Taking duals, we see that the limit of the $V_n$ is $V_{k,\infty} \otimes \delta^{-1}$. Note that the $\delta_{1, n} = \mu_{y_n}\cdot \mu_n^{-1}$ are unramified characters converging to $\delta_1 = \delta \cdot \mu_{p^{r/2}}$. This forces $\delta$ to be unramified and hence crystalline. Therefore $V_{k, \infty} \otimes \delta^{-1}$ is crystalline. The last representation has Hodge-Tate weights $(0, k - 1)$. 
Untwisting by the fixed power of the cyclotomic character, we see that the Hodge-Tate weights of the limit representation of the original sequence again lie in $[a, b]$.

\vspace{.3cm}

\noindent {\bf Acknowledgements:} We would like to thank D. Benois, J. Bergdall, L.
Berger, C. Breuil, P. Colmez, G. Chenevier, H. Schoutens and C. Park for
useful conversations and the anonymous referees for numerous remarks which
greatly helped improve the paper. We would also like to thank the
participants
of the rigid analytic geometry seminar held at TIFR in 2019.
The first and second authors acknowledge support of the Department of Atomic Energy under project number 
12-R\&D-TFR-5.01-0500. During this work, the third author was partially supported by JSPS KAKENHI Grant Numbers
JP15H03610, JP21H00969, JP23K20782.

\end{document}